\documentclass[11 pt,noinfoline]{imsart} 
\usepackage[a4paper, margin= 3cm]{geometry}
\usepackage[english]{babel} 
\usepackage{parskip}    
\usepackage{amsmath,amsfonts,amsthm, amssymb,mathtools}
\usepackage{graphicx}  
\usepackage[shortlabels]{enumitem}  
\usepackage{xcolor, latexsym}  
\usepackage{subfigure} 
\usepackage[pdftex, bookmarks=true, breaklinks, hypertexnames=false, pdfborder={0 0 0 0}]{hyperref} 
\usepackage[normalem]{ulem} 
\usepackage{soul} 
\soulregister\cite7 
\soulregister\ref7
\soulregister\pageref7
\usepackage{bm} 
\newcommand{\bb}{\bm{b}}
\newcommand{\ba}{\bm{a}}

\DeclarePairedDelimiter{\abs}{\lvert}{\rvert}
\DeclarePairedDelimiter{\braces}{\{}{\}}
\DeclarePairedDelimiter{\norm}{\lVert}{\rVert}
\DeclarePairedDelimiter{\sqbrackets}{[}{]}
\DeclarePairedDelimiter{\brackets}{(}{)}
\DeclarePairedDelimiter{\floor}{\lfloor}{\rfloor}
\DeclarePairedDelimiter{\ceil}{\lceil}{\rceil}

\DeclareMathOperator*{\argmax}{argmax}
\DeclareMathOperator{\Var}{Var}

\DeclareMathOperator{\sign}{sign}

\DeclareMathOperator{\FDP}{FDP}
\DeclareMathOperator{\FNP}{FNP}
\DeclareMathOperator{\FDR}{FDR}
\DeclareMathOperator{\mFDR}{mFDR}
\DeclareMathOperator{\mR}{m\mathfrak{R}}

\DeclareMathOperator{\lc}{L_C}
\newcommand{\vphi}{\varphi}
\newcommand{\eps}{\varepsilon}
\newcommand{\iidsim}{\overset{iid}{\sim}}

\newcommand{\dif}{\mathop{}\!\mathrm{d}}
\newcommand{\azeta}{\mathcal{A}_\zeta}

\usepackage{colordvi}  
\usepackage{graphicx}  

\newcommand{\fr}{\mathfrak{R}}

\newcommand{\given}{\,|\,}

\definecolor{blendedblue}{rgb}{0.2,0.2,0.7}

\DeclareMathOperator{\FNR}{FNR}

\newcommand{\ind}[1]{\mbf{1}\{ #1 \}}

\newcommand{\mtc}{\mathcal}
\newcommand{\mbf}{\mathbf}

\newcommand{\be}{\beta}

\newcommand{\ka}{\kappa}
\newcommand{\la}{\lambda}
\newcommand{\La}{\Lambda}

\newcommand{\te}{\theta}
\newcommand{\ta}{\tau}
\newcommand{\veps}{\varepsilon}

\newcommand{\leqa}{\lesssim}
\newcommand{\geqa}{\gtrsim}

\newcommand{\EM}{\ensuremath}
\newcommand{\cA}{\EM{\mathcal{A}}}
\newcommand{\cC}{\EM{\mathcal{C}}}

\newcommand{\cE}{\EM{\mathcal{E}}}

\newcommand{\cJ}{\EM{\mathcal{J}}}
\newcommand{\cK}{\EM{\mathcal{K}}}
\newcommand{\cL}{\EM{\mathcal{L}}}
\newcommand{\cM}{\EM{\mathcal{M}}}
\newcommand{\cN}{\EM{\mathcal{N}}}
\newcommand{\cP}{\EM{\mathcal{P}}}

\newcommand{\cT}{\EM{\mathcal{T}}}
\newcommand{\cU}{\EM{\mathcal{U}}}

\newcommand{\RR}{\mathbb{R}}

\newcommand{\R}{\mathbb{R}}
\newcommand{\II}{\mathbf{1}}
\theoremstyle{plain}  
 \newtheorem{theorem}{Theorem}
 \newtheorem{prop}{Proposition}
 \newtheorem{lemma}{Lemma}
 \newtheorem{corollary}{Corollary}
 \newtheorem{definition}{Definition}

\theoremstyle{remark}
 \newtheorem{remark}{Remark} 
 \newtheorem{example}{Example}
 \newtheorem{assumption}{Assumption}

\newcommand{\ol}{\overline}

\begin{document}

\begin{frontmatter}

\title{Sharp multiple testing boundary for sparse sequences}

\runtitle{}

\author{\fnms{Kweku} \snm{Abraham}\ead[label=e1]{lkwa2@cam.ac.uk}}, 
\author{\fnms{Isma\"el} \snm{Castillo}\ead[label=e2]{ismael.castillo@upmc.fr}}
\and
\author{\fnms{\'Etienne} \snm{Roquain}\ead[label=e3]{etienne.roquain@upmc.fr}}
 
\affiliation{University of Cambridge  \& Sorbonne Universit\'e}

\address{University of Cambridge\\ Statistical Laboratory\\ Wilberforce Road, Cambridge CB3 0WB, UK \\
  \printead{e1} }

\address{Sorbonne Universit\'e\\
  Laboratoire de Probabilit\'es, 
Statistique et Mod\'elisation\\ 4, Place Jussieu, 75252, Paris cedex 05, France\\
 \printead{e2}, \printead{e3}}

\runauthor{Abraham, Castillo \& Roquain}

\begin{abstract}
This work investigates multiple testing by considering minimax separation rates in the sparse sequence model, when the testing risk is measured as the sum FDR+FNR (False Discovery Rate plus False Negative Rate). First using the popular beta-min separation condition, with all nonzero signals separated from $0$ by at least some amount, we determine the sharp minimax testing risk asymptotically and thereby explicitly describe the  transition from ``achievable multiple testing with vanishing risk'' to ``impossible multiple testing''. Adaptive multiple testing procedures achieving the corresponding optimal boundary are provided: the Benjamini--Hochberg procedure with a properly tuned level, and an empirical Bayes $\ell$-value (`local FDR') procedure. We prove that the FDR and FNR make non-symmetric contributions  to the testing risk for most optimal procedures, the FNR part being dominant at the boundary. 
The  multiple testing hardness is then investigated for classes of {\em arbitrary} sparse signals. A number of extensions,  including  results for classification losses and convergence rates in the case of large signals, are also investigated.

 \end{abstract}

\begin{keyword}[class=MSC]
\kwd[Primary ]{62G20, 62G15}
\end{keyword}

\begin{keyword}
\kwd{Multiple testing}
\kwd{Sharp asymptotic minimaxity}
\kwd{False Discovery Rate}
\kwd{Benjamini--Hochberg procedure}
\kwd{Frequentist analysis of Bayesian procedures} 
\end{keyword}

\end{frontmatter}

\tableofcontents

\pagebreak

\section{~~Introduction}

\subsection{~~Background}

Multiple testing is a prominent topic of contemporary statistics, with a wide spectrum of applications including for example molecular biology, neuro-imaging and astrophysics.
In this framework, many individual tests have to be performed simultaneously while controlling global error rates that take into account the multiplicity of the tests. A primary aim is to build procedures that guarantee control of a form of type I error, the most popular being the False Discovery Rate (FDR, see \eqref{deffdr} below). For instance, the celebrated Benjamini--Hochberg (BH) procedure controls the FDR under independence \cite{BH1995}. Given an FDR controlling procedure, one may then ask whether it has a controlled type II error (or, equivalently, good power), measured for instance by the False Negative Rate (FNR, see \eqref{deffnr} below). 

In this context, a natural question is that of {\em optimality}: what is the best sum of type I and type II errors that is achievable by any multiple testing procedure? 
From a testing perspective and when a single test is considered, this can be answered via minimax separation rates between the two considered hypotheses, which have been investigated for a variety of nonparametric (see, e.g., \cite{ingstersuslina}, and \cite{ginenicklbook} for an overview and further references) and high-dimensional models and loss functions, see, e.g., \cite{baraud02},  \cite{ingsterverzelentsybakov}, \cite{nicklvdg}. The analogous question for multiple testing has received attention only very recently: the case of a familywise error risk  is studied in \cite{fromontetal16}, and in the case of FDR and FNR risks for deterministic sparse signals, aspects of this problem have being investigated in \cite{erychen} and subsequently also in \cite{cr20}, \cite{rabinovich20}, \cite{rabinovichpreprint}, \cite{belitser21}. {More precise connections to these works are made below, see Section~\ref{sec:litothers}}. 
%
%
%

\subsection{~~Sparse sequence model}

For some $\te = (\te_1,\dots,\te_n)$   
consider observing independent data $X=(X_1,\dots,X_n)$ satisfying, for a family of density functions $(f_a : a \in \RR)$,
\begin{equation}\label{eqn:generalnoisemodel} X_i \sim f_{\te_i}, \quad i=1,\dots ,n.
\end{equation} 
The vector $\theta$ is assumed to be sparse; that is, to belong to the set
\begin{equation}\label{defl0} 
 \ell_0[s_n] = \left\{\te\in\RR^n,\ \norm{\theta}_0\le s_n \right\},\:\:\:\norm{\theta}_0:=\#\{1\le i\le n:\ \te_i\neq0\},
\end{equation} 
consisting of vectors that have at most $s_n$ nonzero coordinates, where $0\le s_n\le n$;
 throughout the paper we consider the sparse asymptotic setting where
\begin{equation}\label{sparse}
\mbox{$n\to\infty$, $s_n\to\infty$ and  $n/s_n\to\infty$.}
\end{equation}

We write $P_\te$ for the law of $X$ with parameter $\theta$ in \eqref{eqn:generalnoisemodel} and $E_\theta$ for the corresponding expectation. 
Write $F_a(x) = \int_{-\infty}^x f_a(t)\dif t$ and $\overline{F}_a(x)=1-F_a(x)$ for the distribution function and the tail function, respectively. 
We will later place some regularity conditions on the $F_a$ and some signal strength conditions on the $\theta$.

\begin{example}\label{ex:GaussianSequence}
The prototypical example to have in mind is the Gaussian location model, under which $f_a$ is the  density of the distribution  $\cN(a,1)$, so that
\begin{equation}
	X_i = \theta_i + \veps_i,\quad \eps_i\iidsim \cN(0,1), \quad i=1,\ldots,n.  \label{model}
\end{equation}
We first present our results in this model in Section~\ref{sec:main}, and then in Section~\ref{sec:mainmult} we give conditions for the more general model~\eqref{eqn:generalnoisemodel} under which corresponding results hold, allowing for a diverse range of models including for example Gaussian {\em scale}  models.
\end{example}

The multiple testing problem consists of testing simultaneously, for each $1\le i\le n$, the null hypothesis that there is no signal against the alternative hypothesis:
\[ H_{0,i}:``\te_i=0" \qquad \text{vs.}\qquad H_{1,i}: ``\te_i\neq 0". \]


\subsection{~~Multiple testing risks}

A multiple testing procedure is formally defined as a measurable function of the data $\vphi: x\in\R^n \mapsto (\vphi_i(x))_{1\leq i\leq n}\in \{0,1\}^n$, where, by convention, $\vphi_i(X)=1$ corresponds to rejecting the null $H_{0,i}$. As such, the procedure will depend on $n$, and with some slight abuse of terminology, when dealing with asymptotics in terms of $n$, a sequence of such procedures is sometimes simply referred to as a `procedure' for short. 

For any $\theta\in\RR^n$ and any procedure $\vphi$, the false discovery rate (FDR) and the false discovery proportion (FDP) of $\vphi$ at the parameter $\theta$ are respectively defined as
\begin{equation}\label{deffdr}
	\FDR(\theta,\vphi) = E_\te[\FDP(\theta,\vphi)], \:\:\:\FDP(\theta,\vphi)=\frac{\sum_{i=1}^n \ind{\theta_i = 0}\vphi_i(X)}{1\vee \sum_{i=1}^n \vphi_i(X)}.
\end{equation}

The false negative rate (FNR) at $\theta$ is here defined as (see, e.g., \cite{erychen})
\begin{equation} \label{deffnr}   
\FNR(\theta,\vphi) 
= E_{\theta}\left[\frac{\sum_{i=1}^n \ind{\theta_{i}\neq 0} (1-\vphi_i(X))}{1\vee \sum_{i=1}^n \ind{\theta_{i}\neq 0}}\right].
\end{equation}
The (multiple testing) combined risk at $\theta\in\RR^n$ of a procedure $\vphi$ is the sum
\[ \mathfrak{R}(\theta,\vphi) = \FDR(\theta,\vphi)+\FNR(\theta,\vphi). \] 
Given the popularity of FDR and FNR, this can be considered as a canonical notion of testing risk in the multiple testing context: it has indeed been considered for example in \cite{erychen} and the papers mentioned above. While the above is the main notion of risk used in this paper, other choices, including the classification risk, are discussed in Section~\ref{sec:orisk}.
 
\subsection{~~Separation of hypotheses and minimax testing risk} 

	To investigate questions of optimality, a natural benchmark is the minimax multiple testing risk, defined as \begin{equation}\label{minimax}
		\fr(\Theta)=\inf_{\vphi} \sup_{\te\in\Theta}\, \fr(\te,\vphi), 
	\end{equation} 
	where the infimum is over all multiple testing procedures and the parameter set $\Theta$ is some appropriate subset of $\ell_0[s_n]$. Typically, interesting parameter sets $\Theta$ are given to be ``as large as possible'', while keeping $\fr(\Theta)$ in \eqref{minimax} at least strictly smaller than $1$. Noting that the risk of the trivial procedure $\vphi_i=0$ for all $i$ is equal to $1$, the latter corresponds to parameter configurations for which `non-trivial multiple testing' is achievable.

	To investigate such `separation' rates in the present high-dimensional setting, perhaps the most popular approach is via a `beta-min' condition (see, e.g., \cite{bvdgbook}, Section 7.4) meaning that all nonzero signals are above a certain threshold value. For instance, this condition is used for the (related but different) task of consistent model selection for estimators such as the LASSO; see Section~\ref{sec:litothers} for more detail on this, and on other losses.
	\\



To fix ideas,  let us consider  the collection $\Theta(M)$ of vectors $\theta\in\ell_0[s_n]$ with nonzero coordinates taking only one possible value $M=M(n,s_n)$, in the Gaussian sequence model of Example~\ref{ex:GaussianSequence}. 
It follows from results in \cite{erychen} (Theorem~2 therein) that if 
\begin{equation} \label{boundrough}
M >a\sqrt{2\log(n/s_n)}
\end{equation}
for some constant $a>1$, then the BH procedure with appropriately vanishing parameter has a vanishing $\fr$-risk. In addition, it is proved that no {\em thresholding-type procedure} can have a non-trivial $\fr$-risk uniformly over $\Theta(M)$ if $a<1$ (Theorem~1 therein). These results suggest that, at least for thresholding procedures, the boundary of possible multiple testing is ``close to'' the threshold $\sqrt{2\log(n/s_n)}$, which we refer to as the `oracle threshold' in the sequel.
  
\subsection{~~Questions of interest} The previous discussion raises the following questions:  
\begin{itemize}
\item how does the minimax risk $\fr(\Theta)$ behave for separated alternatives (e.g., for $\Theta=\Theta(M)$ with $M$ as in \eqref{boundrough}) when the infimum in \eqref{minimax} if taken over all possible procedures, not only thresholding ones?
\item since  multiple testing is ``easy'' when $M$ is large and impossible when $M$ is too small, 
what is the precise (asymptotic) boundary of signal strength $M$ for which $\fr(\Theta)$ goes from $0$ to $1$? In other words, can one describe the transition from $0$ to $1$ in $\fr$ when $M$ decreases? This requires investigating the sharp minimax  risk $\fr(\Theta)$. 
\item are there procedures that achieve the minimax risk, at least asymptotically, without  knowledge of the sparsity parameter $s_n$?
\item suppose that rather than the $s_n$ nonzero coordinates all equalling one signal value $M$, they can be divided into two different values $M_1$ and $M_2$ (e.g., as in Figure~\ref{fig:mt_pb}, middle column). Is this an easier or more difficult multiple testing problem than the former?  
\item (as asked by one referee:) under what 
 assumptions on the noise can these questions can be addressed in a similar way?
\item (as asked by one referee:) when $M$ is large, e.g. if $a>1$ in \eqref{boundrough}, what is the optimal convergence rate of the risk to zero and can one find a procedure achieving this rate?
\end{itemize}
These and more general questions are addressed in the sequel.

\begin{figure}[h!]
\begin{center}
\includegraphics[scale=.9]{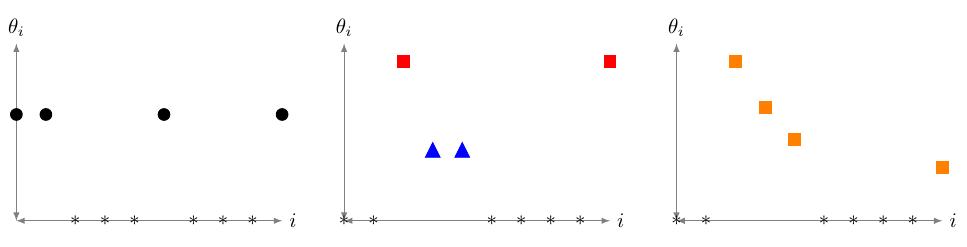}
\end{center}
\caption{Some signal configurations for nonzero $\te_i$'s to be considered.  Left: equal signal strength; Middle: two distinct signals; Right: all signal values different.  Zero $\te_i$'s are depicted by with a symbol ``$\ast$''.
 \label{fig:mt_pb}}
\end{figure}

As a specific aspect of testing problems in general, and of multiple testing in particular, it is common in practice to allow for a tolerance level, especially for the type I error, here typically for the FDR. Hence, for the combined risk under study here, it is not only the case where $\fr(\Theta)=o(1)$ that is of interest, but also the one where $\fr(\Theta)$ is of the order of a (possibly small) constant $c\in(0,1)$. In addition, since the FDR and FNR account for   errors that have different interpretations, it is also of interest to study the contribution of each error rate in the combined risk $\fr(\Theta)$. 
 The results below will show that the contributions of FDR and FNR are not symmetric in this regime.
Finally, we present our results asymptotically for simplicity -- in particular this eases the presentation of results in the new setting of multiple signal strengths considered in Section~\ref{sec:mainmult} -- but the proofs can be adapted to give some non-asymptotic bounds. 

\subsection{~~Popular procedures: BH and empirical Bayes $\ell$-values}\label{sec:popular}


Here, we describe two procedures that will be considered in the sequel. 

First, probably the most widely used multiple testing procedure is the so-called Benjamini--Hochberg procedure, introduced in \cite{BH1995}. For some level $\alpha$, it is given by 
$\vphi^{BH}_\alpha=(\ind{|X_i|\geq \hat{t} })_{1\leq i\leq n}$ where the threshold $\hat{t}=\hat{t}(\alpha) $ is defined as a specific intersection point between the empirical upper-tail distribution function of the $X_i$'s and a quantile curve of the noise distributions (see Section~\ref{secBH} for details). To achieve good performances with respect to the combined risk, we will make use of the BH procedure where $\alpha$ is chosen to be slowly decreasing with $n$, as in \cite{erychen,Bogdan2011,neuvialroquain12}. A typical choice is $\alpha=\alpha_n\asymp 1/\sqrt{\log n}$.

The second procedure uses Bayesian $\ell$-values (often also called local FDR values) with an empirical Bayes calibration. 
For a particular spike-and-slab prior $\Pi_w$ on $\R^n$ (see Section~\ref{sec:l-vals-adaptive} for details), we consider the empirical Bayes $\ell$-value procedure defined by thresholding posterior probabilities of null hypotheses at some specified level $t\in(0,1)$, i.e., $\vphi^{\hat{\ell}}_t=(\II\braces{\Pi_{\hat{w}}(\theta_i=0\mid X)<t})_{1\leq i\leq n}$, where $\hat{w}$ is the marginal maximum likelihood  estimator for $w$, as in \cite{js05, js04}. 
The choice of $t$ is not critical for obtaining a small combined risk, e.g., $t=0.3$ or $t=1/2$ are possible choices. 
The widely used empirical Bayes spike-and-slab posterior distribution was investigated in terms of estimation properties in \cite{js04, cm18}, confidence sets in \cite{cs20}, and the resulting $\ell$-value procedure was recently shown to control the FDR  in \cite{cr20}; see also \cite{bcg20} for an overview on the analysis of Bayesian high-dimensional posteriors and \cite{acr21} for a related $\ell$-value multiple testing algorithm. 

\subsection{~~Related literature, other modelling assumptions and risks} \label{sec:litothers}
In the sparse sequence model,  the interesting recent works \cite{erychen,rabinovich20} provide bounds for the $\fr$-risk. In \cite{erychen}, the separation condition \eqref{boundrough} is considered and the risk is shown to asymptotically converge to $0$ for some procedures when $a>1$, while it converges to $1$ for any thresholding based procedure when $a<1$. In \cite{rabinovich20}, non-asymptotic lower bounds and upper bounds are further derived, for thresholding procedures, in the regime where $a=a_n>1$ (possibly approaching $1$) is known. 
 These bounds are shown to be matched for the BH procedure with a suitably decreasing level. This analysis is further broadened and extended to more general models in the recent preprint \cite{rabinovichpreprint} (see also Section~\ref{sec:rab} for further discussion). In these works, which unlike the present work consider only thresholding procedures, the case where the risk converges to an arbitrary constant is not studied, hence the problem of identifying the sharp transition of the minimax risk from $0$ to $1$ was left open, as was the question of adapting to the signal strength for large signals.
 

In the multiple testing literature, a related way to measure optimality consists of finding a solution that minimizes the FNR while controlling a FDR-type error rate at level $\alpha$, see, e.g., \cite{RW2009,IH2017,Dur2019} in case of weighted procedures. This task is often done under the so-called `two-group mixture model', introduced in \cite{ETST2001}, which assumes that each null hypothesis is true with some probability, and relies on specific $\ell$-value (or local FDR) thresholding procedures, see \cite{SC2007,sun2009large,CS2009,CSWW2019} and the  recent work \cite{heller2021optimal}.  
One way the present work differs from these references is in seeking not to minimise the FNR under a constraint, but rather to minimise the combined risk $\fr$. A more important difference is that we do not posit a mixture distribution for the true parameter $\theta$, but rather assume it is deterministic and arbitrary (up to the sparsity constraint). 
 Despite these differences, we will see that an $\ell$-value thresholding based procedure (namely, $\vphi^{\hat{\ell}}_t$ as mentioned above) still achieves minimax performance.

Regarding related testing problems, the important work by Donoho and Jin \cite{DJ2004} studies the detection problem for a single null and multiple alternatives; see also the subsequent works \cite{Hall_2008,Hall_2010,Arias_Castro_2011,li2020optimality}. In addition, the sparse sequence model has been much studied in terms of estimation for quadratic or $\ell^p$-losses, here we only mention 
\cite{ABDJ2006, sucandes16} for their connections to multiple testing, where the authors use estimators related to the BH procedure.

Let us also mention that another FDR+FNR risk has been considered in \cite{GW2002}, in a model with a single alternative distribution and using the knowledge of the number of true nulls. However, the FNR definition there is different from here: the denominator equals the number of accepted nulls, rather than the number of alternatives as here. In the case of sparse signal, these two FNR notions scale very differently; the former scaling is not well-suited to deal with sparsity, while using the FNR notion considered herein enables us to exhibit a sharp phase transition phenomenon. The recent work \cite{belitser21}   
derives some robust results for model selection based procedures, including for various notions of sums FDR+FNR, but only in a range where the multiple testing risk tends to zero at a certain rate.

Finally, a different but somewhat related loss function is the Hamming or classification loss. 
The corresponding classification risk is considered in \cite{Bogdan2011,neuvialroquain12} in a two-group mixture model, with lower bounds restricted to thresholding classifiers. There,  it is proved that the BH procedure with a suitably vanishing level achieves the oracle performance. However, these results study the Bayes risk (i.e., the minimum average risk) in this mixture model and do not provide a complete minimax analysis. Further extensions are derived in \cite{JK2016}, e.g., by handling more general dependent models.
Coming back to the Gaussian sequence model \eqref{model} (with non-random $\theta_i$'s), minimax Hamming estimators are derived in \cite{butucea18, butucea21}  (see also the earlier work \cite{But2017}), where the authors study the boundary for exact recovery (the classification risk goes to $0$) and almost sure recovery (the classification risk is a vanishing fraction of the sparsity parameter). We refer to Section~\ref{sec:orisk} for more on the classification risk. 

\subsection{~~Outline and notation}


Section~\ref{sec:main} contains a first series of results in the setting of Example \ref{ex:GaussianSequence} under a `beta-min' condition on signals. This enables us to present some main results of the paper in a simple case first. Namely, the asymptotic minimax risk is computed and adaptive procedures achieving this risk are given, with the FNR shown to be the dominating term in the risk for most procedures. Section~\ref{sec:mainmult} states the main results under general conditions: the signal strengths are arbitrary and the model signal and noise distributions more general. Illustrations are provided in Section~\ref{sec:simu}. Classification loss and some other risks are investigated in Section~\ref{sec:orisk}. Section~\ref{sec:fast} considers the case of large signals with risks rapidly decreasing toward zero  and Section~\ref{sec:disc} concludes with a brief discussion. While one key pair of results is proved in Section~\ref{sec:proofs}, other proofs are postponed to the supplement \cite{ACR21supp} (references to which are given the prefix S-).  

{\em Notation.}
The density of a standard normal variable is denoted by $\phi$ and we write $\Phi$ for the cumulative distribution function $\Phi(x)=\int_{-\infty}^x \phi(t){\dif t}$ and $\overline{\Phi}=1-\Phi$ for the tail probabilities.   
We use $\varliminf_{n\to \infty},$ $\varlimsup_{n\to\infty}$ to denote the liminf or the limsup, respectively. 
We use $x_n\lesssim y_n,$ $y_n\gtrsim x_n$ or $x_n=O(y_n)$ to signify that there exists $C>0$ such that $x_n\leq Cy_n$ for all $n$ large; we write $x_n\asymp y_n$ if $x_n\lesssim y_n$ and $y_n\lesssim x_n$; we write $x_n\sim  y_n$ if $x_n/y_n\to 1$; and we write $x_n\ll y_n,$ $y_n\gg x_n$ or $x_n=o(y_n)$ if $x_n/y_n\to 0$. We use this notation correspondingly for functions $f(x)$, $g(x)$, with the limits taken either as $x\to \infty$ or as $x\to 0$ depending on context. {A sequence of random variables is said to be $o_P(1)$ if it converges to $0$ in probability under the data generating law $P_\te$. For reals $a,b$, we set $a\vee b=\max(a,b)$ and $a\wedge b=\min(a,b)$.}  Finally, for a finite set $A$, we denote by $\abs{A}$ or $\# A$ its cardinality.

\section{~~Sharp boundaries: beta-min condition and Gaussian noise}
\label{sec:main}


In this section, we focus on the Gaussian location model presented in Example~\ref{ex:GaussianSequence}. 
To evaluate the minimax risk, we define a class of configurations for $\theta$ that measures how the alternatives are separated from the null hypothesis:
for a given $a\in\RR$, set
\begin{equation} \label{singlesignal}
 \Theta = \Theta(a,s_n) =
 \bigg\{ \theta\in \ell_0[s_n] \::\: |\theta_{i}| \geq a\ \text{ for } i\in S_{\te},\ \ |S_{\te}|=s_n\bigg\},
\end{equation}
where $S_\theta=\braces{i : \theta_i\neq 0}$ denotes the support of $\theta$ 
(recall $\ell_0[s_n]$ is defined by \eqref{defl0}). 
This choice corresponds to a so-called beta-min type condition, meaning that all intensities of nonzero coefficients are required to be above a certain value $a>0$.  For example, each of the signals depicted in Figure~\ref{fig:mt_pb} belongs to some class $\Theta(a,s_n)$; in the first panel, one may take  $a$ to be the shared value of the nonzero $\te_i$'s, whereas in the third panel $a$ can at most be taken to be the smallest nonzero value. More general classes, leading to more refined results for the second and third panels in Figure~\ref{fig:mt_pb}, are considered in Section~\ref{sec:mainmult}.

\subsection{~~Minimax multiple testing risk}

Previous works in the literature \cite{erychen,rabinovich20} suggest that the phase transition for the combined risk over $\Theta(a,s_n)$ arises when $a$ is close to the oracle threshold $\sqrt{2 \log(n/s_n)}$. 
We identify here a sharp formulation for this boundary, by considering  
  for $b\in\RR$,
\begin{equation} \label{signalb}
\Theta_b=\Theta(a_b,s_n),\:\:\: \quad a_b= \sqrt{2 \log(n/s_n)} +b,
\end{equation} 
with $\Theta(a,s_n)$ as in \eqref{singlesignal}.
The following result holds.

\begin{theorem} \label{thmgeneric}
In the Gaussian location model of Example~\ref{ex:GaussianSequence} consider $\Theta_b=\Theta(a_b,s_n)$ as in \eqref{signalb}.
For any fixed $b\in \RR$, under the sparse asymptotics \eqref{sparse} the minimax $\mathfrak{R}$-risk over $\Theta_b$ verifies 
\[  \inf_\vphi \sup_{\te\in\Theta_b}  \mathfrak{R} (\theta,\vphi)=
\overline{\Phi}(b)+o(1). \]
The result also holds in the limiting cases $b=b_n\to +\infty$ and  $b=b_n\to-\infty$. 
\end{theorem} 

Since $\overline{\Phi}(b)$ increases from $0$ to $1$ when $b$ decreases from $+\infty$ to $-\infty$, Theorem~\ref{thmgeneric} exhibits an asymptotic phase transition and shows that the considered boundary \eqref{signalb} is sharp.
It follows from the proof of the result that  
 the oracle thresholding rule 
 \begin{align}\label{equidealthresh}
 \vphi^*_i=\ind{\abs{X_i}\geq \sqrt{2 \log(n/s_n)}}, \:\:1\leq i\leq n,
 \end{align}
  is asymptotically minimax, {independent of the value of $b$.}

\subsection{~~Applications: non-trivial testing and conservative testing}

Let us now provide two consequences of Theorem~\ref{thmgeneric} for $\Theta_b=\Theta(a_b,s_n)$ as in \eqref{signalb}. 

\begin{corollary} \label{cor_nt}
For any fixed $b\in\mathbb{R}$, asymptotic non-trivial testing is possible over $\Theta_b$ in the Gaussian location model of Example~\ref{ex:GaussianSequence}, in that there exists a procedure $\vphi$ such that 
\[ \lim_{n\to\infty} \sup_{\te\in\Theta_b}  \mathfrak{R} (\theta,\vphi)<  1. \]
In contrast, for any sequence $b=b_n\to -\infty$, we have
\[ \lim_{n\to\infty}\inf_\vphi \sup_{\te\in\Theta_b}  \mathfrak{R} (\theta,\vphi) =  1, \]
that is, asymptotic  non-trivial testing is impossible.
\end{corollary}
We  identify a new sharp boundary for asymptotic  non-trivial testing: $a_b=\sqrt{2\log (n/s_n)} + b$ with $b\to-\infty$ corresponds to the regime where it is impossible to build a multiple testing procedure doing (asymptotically) better than the trivial ones $\vphi_i=0$ for all $i$ or $\vphi_i=1$ for all $i$. On the contrary, provided that $b$ is a finite constant (which may  be negative), some non-trivial control of the $\mathfrak{R}$-risk is possible. 

\begin{corollary} \label{cor_ct}
	Consider the Gaussian location model of Example~\ref{ex:GaussianSequence}. 
If  $b=b_n\to +\infty$, asymptotic conservative testing is possible over $\Theta_b$: there exists a procedure $\vphi$ such that
\[  \sup_{\te\in\Theta_b}  \mathfrak{R} (\theta,\vphi) \to 0, \]
as $n\to\infty$. In contrast, for any fixed $b\in\RR$,  we have
\[ \lim_{n\to\infty}\inf_\vphi \sup_{\te\in\Theta_b}  \mathfrak{R} (\theta,\vphi) >0, \]
that is, asymptotic conservative testing is impossible over $\Theta_b$ for any procedure.
\end{corollary}  

We thus find the following sharp boundary for asymptotic conservative testing: it is possible when $a_b=\sqrt{2\log(n/s_n)} + b$ with $b=b_n\to+\infty$, but becomes impossible when $b$ is finite.
It is interesting to note that a similar boundary with $b=b_n\to+\infty$ was identified in \cite{butucea18} for the so-called {\em almost-full recovery} problem when the risk is the Hamming loss and the goal is to correctly classify up to a $o(s_n)$ number of nonzero signals. 
 In this view, Corollary~\ref{cor_ct} can be seen as an analogue of Theorem 4.3 in \cite{butucea18} for the multiple testing risk. 
We refer to Section~\ref{sec:orisk} (and Section~\ref{sec:almostfullrecovery}) for a more detailed connection to the classification problem.

\subsection{~~Adaptation}

The asymptotically minimax procedure \eqref{equidealthresh} requires the knowledge of the sparsity parameter $s_n$.
A natural question is whether there exists an {\em adaptive} multiple testing procedure that achieves the minimax risk without using the knowledge of  $s_n$. 
Such procedures do exist. Some will require an additional 
{\em polynomial sparsity} assumption: for some unknown $c<1$,
\begin{equation} \label{polspa}
s_n \leqa n^c.
\end{equation}

\begin{theorem} \label{thm-adapt-r}
In the Gaussian location model of Example~\ref{ex:GaussianSequence}, consider $\Theta_b=\Theta(a_b,s_n)$ as in \eqref{signalb} where $b$ is either fixed in $\R$ or is a sequence $b=b_n$ tending to $+\infty$. 
There exists a multiple testing procedure $\vphi$, not depending on $s_n$ or $b$, such that in the sparse asymptotics \eqref{sparse}  
\[ \sup_{\te\in\Theta_b} \fr(\te,\vphi) = \overline\Phi(b)+o(1). \]
In particular this holds for the empirical Bayes $\ell$-value procedure \eqref{ellvalp} taken at any fixed threshold $t\in (0,1)$. 
In addition, under \eqref{polspa},  this also holds for the BH procedure \eqref{defBH} taken at a vanishing level $\alpha=\alpha_n=o(1)$ that satisfies $-\log(\alpha_n)=o(\sqrt{\log n})$.
\end{theorem}

Theorem~\ref{thm-adapt-r} establishes that two popular (and simple) procedures are asymptotically sharp minimax adaptive. In particular, they automatically achieve the non-trivial and conservative boundaries described in  Corollaries~\ref{cor_nt}~and~\ref{cor_ct}. One advantage of the empirical Bayes $\ell$-value procedure over the BH procedure is that it does not need any further parameter tuning in order to be valid. Nevertheless, 
an advantage of the BH procedure is that its FDR is always equal to $(1-\abs{S_\theta}/n)\alpha$ (see \cite{BY2001}) so that one has more explicit information about the first term of the combined risk.

 \begin{remark}[Case $b=b_n\to -\infty$]\label{rem:lvaluenotphibar}
When $b_n\to -\infty$, the lower bound of Theorem~\ref{thmgeneric} shows that non-trivial testing is impossible. One could in principle have $\frak{R}(\theta,\vphi)=2>1$ so that the upper bounds in Theorem~\ref{thm-adapt-r} are not automatic. However, the BH procedure at level $\alpha_n=o(1)$  achieves the limit $1=\overline{\Phi}(-\infty)$ in this case, as a consequence of the explicit control of its FDR noted above. Under  polynomial sparsity \eqref{polspa} the same is true of the $\ell$-value procedure (with $t\leq 3/4$) as an immediate consequence of Theorem 1 in \cite{cr20}.
\end{remark}

\begin{remark}[Polynomial sparsity assumption for the BH procedure]\label{rem:BHnotpoly}
The condition $-\log(\alpha_n)=o(\sqrt{\log n})$ for the BH procedure is typically achieved by choosing $\alpha_n=1/\sqrt{\log n}$. A rate of convergence of the risk to $\overline{\Phi}(b)$ of $(\log\log n)/\sqrt{\log n} + e^{-cs_n}$ can be obtained for such choice, see Theorem~\ref{thm:BHpointwise} and the paragraph thereafter.  
 In addition, the polynomial sparsity assumption can be relaxed slightly, see Remark~\ref{rem:polysparse} (second bullet).
\end{remark}

\subsection{~~Sparsity preserving procedures and dominating FNR at the boundary}\label{sec:FNRdominating}

Theorem~\ref{thmgeneric} determines the minimax optimal combined risk $\fr$, i.e. the sum $\FDR+\FNR$,  but not the balance between the two terms struck by optimal procedures. In this section, we investigate this tradeoff in more details.
For instance, the proof of Theorem~\ref{thmgeneric} shows that the oracle procedure  \eqref{equidealthresh}  achieves an optimal multiple testing risk by satisfying
\begin{equation}\label{err-oracle}
 \sup_{\te\in\Theta_b} \FDR(\te,\vphi^*) = o(1),\qquad
\sup_{\te\in\Theta_b} \FNR(\te,\vphi^*) = \overline{\Phi}(b)+o(1),
\end{equation}
that is, the FNR ``spends'' all the allowed ``budget'' from the overall minimax risk.
In this section, we show that this phenomenon holds true for most ``reasonable'' procedures. Clearly, some restriction is required on the class of procedures as the trivial test $\vphi\equiv 1$ that always rejects the null achieves the optimal asymptotic risk (of $1$) over $\Theta_{b}$ for $b=b_n\to-\infty$ and has FNR zero. We see below that for any `sparsity preserving' procedure, that is, one not overshooting the true sparsity index $s_n$ by more than some large multiplicative factor, the FNR {\em alone} cannot go below  
$\overline{\Phi}(b)$ asymptotically.

\begin{definition}\label{def:spapres}
We say that a  multiple-testing procedure $\vphi=\vphi(X)\in\{0,1\}^n$ {(or strictly, a sequence of such procedures, indexed by $n$),} is {\em sparsity-preserving over} $\Theta=(\Theta_n)_n$ (with $\Theta_n\subset\ell_0[s_n]$) {\em up to a multiplicative factor} $A=(A_n)_n$ if, as $n\to\infty$,
\begin{equation} \label{spa-pres}
\sup_{\te\in\Theta_n} P_{\te}\left[\sum_{i=1}^n \vphi_i(X) > A_n s_n \right]=o(1).
\end{equation} 
We denote by $\mathcal{S}_{A}(\Theta)=\mathcal{S}_{A}((s_n)_n,\Theta)$ the set of all such procedure sequences.
\end{definition}
The sparsity preserving property entails a total number of rejections not exceeding $A_ns_n$ with high probability, and can be interpreted as a weak notion of type I error rate control. Many procedures (or also estimators) encountered in the literature on sparse classes verify this condition for $A_n$ either a large enough constant or going to  infinity slowly, even when taking the supremum in \eqref{spa-pres} over the whole sparse class $\ell_0[s_n]$. Several examples are provided in Section~\ref{sec:sparsitypreserving}, including the BH procedure (with fixed level $\alpha<1$ or with $\alpha=\alpha_n\to 0$), the $\ell$-value procedure, the oracle thresholding procedure used to prove Theorem~\ref{thmgeneric}, and, more generally, procedures controlling the FDP \eqref{deffdr} in a specific sense. 

\begin{theorem}\label{thm-notrade}
In the setting of Theorem~\ref{thmgeneric}, consider a fixed real $b$, a sequence $B=(B_n)_n$ with $B_n^{2}\leq e^{(\log (n/s_n))^{1/4}}$ and $\varliminf_n B_n>1$, 
and consider $\mathcal{S}_{B}(\Theta_b)$ as in Definition~\ref{def:spapres}. Then we have in the sparse asymptotics \eqref{sparse}
\begin{equation}\label{equ:FNRminimax}
 \inf_{\vphi\in \mathcal{S}_{B}(\Theta_b)}\sup_{\te\in\Theta_b}\ \FNR(\te,\vphi) = \overline\Phi(b)+o(1).
 \end{equation}
 The result continues to hold if $b=b_n\to +\infty$ or $b=b_n\to -\infty$.
\end{theorem}

Since Theorem~\ref{thm-adapt-r} shows that the (sparsity preserving) $\ell$-value and BH procedures have their FNR suitably upper bounded, these procedures both achieve the bound \eqref{equ:FNRminimax}, under  polynomial sparsity \eqref{polspa} for the BH procedure (note that the bound \eqref{equ:FNRminimax} trivially holds for any procedure in the case $b=b_n\to-\infty$).

Theorem~\ref{thm-notrade} sharpens the lower bounds of Theorem~\ref{thmgeneric} by stating that the combined risk $\fr$ can in fact be replaced by the FNR (i.e. the type II error) only, if one is willing to restrict slightly the class of multiple testing procedures to $\vphi$'s that do not often reject more than $B_ns_n$ hypotheses, where $B_n$ is permitted to diverge at some rate specified above. Some intuition behind this result is given in Section~\ref{sec:simu} (in particular, see the discussion of Figure~\ref{fig:mFDRvsFNR}). 

This phenomenon of dominating FNR is a novel finding. We would like to underline two related points.  First, it is specific to the regime where the $\fr$--risk is bounded away from $0$ (otherwise the lower bound in \eqref{equ:FNRminimax} is trivial); in regimes with fairly strong signals, such as ones studied in Section \ref{sec:fast}, both FDR and FNR typically vary on the same level. Second, it qualitatively explains results in Section \ref{sec:orisk}, where we will show that the sharp minimax constant for the normalised classification risk is the same as for the $\fr$--risk. Indeed, both risks have the same type II--error risk (equal to the FNR).

 A straightforward but interesting consequence of Theorem~\ref{thm-notrade} is that no sparsity preserving procedure can ``trade'' some loss in $\FDR$ for some improvement of the $\FNR$ while still staying close to the optimal risk:
 if its FDR is equal to $\alpha\in (0,1)$ asymptotically (such as the standard BH procedure for a fixed level $\alpha$), it must miss the sharp combined risk $\fr$ over $\Theta_b$ by at least an additive factor $\alpha$. Another (surprising) consequence of Theorem~\ref{thm-notrade} for `top-$K$' procedures is given in Section \ref{sec:adaptr}, see Corollary \ref{cor-topsub}.

\begin{corollary} \label{cor-subopt}
In the setting of Theorem~\ref{thm-notrade}, for some $\alpha \in(0,1)$ and $b\in\RR$, let $\vphi_\alpha$ be any sparsity preserving procedure $\vphi_\alpha \in \mathcal{S}_{B}(\Theta_b)$ such that $\varliminf_{n}  \inf_{\te\in\Theta_b}\FDR(\te,\vphi_\alpha) \geq \alpha>0$. Then $\vphi_\alpha$ must miss the asymptotically minimax combined risk by at least  $\alpha$, that is,
\[ \varliminf_{n} \sup_{\te\in\Theta_b} \fr(\te,\vphi_\alpha) \geq \alpha+ \overline\Phi(b).\]
This is in particular the case of the BH procedure with fixed level $\alpha$, as defined in \eqref{defBH}.
\end{corollary}

Corollary~\ref{cor-subopt} follows from Theorem~\ref{thm-notrade}, the fact that the BH procedure with fixed parameter is sparsity preserving (Section~\ref{sec:sparsitypreserving}) and the explicit expression for the FDR of the BH procedure, see \eqref{FDRBH}.

Finally, for completeness, let us mention that for a procedure which is not sparsity preserving, trading FDR for FNR is formally possible. A (somewhat degenerate) example is given in Section~\ref{sec:sparsitypreserving}: see Example~\ref{tradepossible}. 



\section{~~Sharp boundaries: arbitrary signal strengths and beyond Gaussian noise} \label{sec:mainmult}

This section presents more refined results, by relaxing the assumptions on both the signal strength and the noise. 
These results show exactly how the limiting risk depends on a measure of the signal strengths, in a general sequence model which allows for non-additive and non-Gaussian noise. This generalisation also sheds light on the previous Gaussian results, by revealing an underlying sufficient set of  assumptions for the proofs. In addition, we are also able to prove local, ``pointwise'', versions of the results, that we largely postpone to Section~\ref{sec:local} for clarity of presentation. 


\subsection{~~Extended noise assumption}

Recall from \eqref{eqn:generalnoisemodel}--\eqref{sparse} that we assume $X_i\sim f_{\theta_i}$, $i\leq n$, independently, with $(f_a : a \in \RR)$ a family of densities and $\theta\in\ell_0[s_n]$, with $s_n\to\infty$ and $n/s_n\to \infty$. Recall the notation $F_a $ and $\overline{F}_a$ therein.

\begin{assumption}\label{ass:generalnoise}
	There exists a constant $L$ such that each $F_a$ is $L$-Lipschitz. There exist sequences of positive numbers $a_n^*\to \infty$, $\delta_n\to 0$ such that
	\begin{align} \label{eqn:modelassumption2}
		(n/s_n)\overline{F}_0\brackets[\big]{a_n^*-\delta_n} \to \infty, \\
		\label{eqn:modelassumption3}
		(n/s_n)\overline{F}_0\brackets[\big]{a_n^*} \to 0.
	\end{align}
The density $f_0$ is continuous and positive on $\RR$.
Further assume one of
\begin{enumerate}[A]
	\item$\hspace{-.2cm}.$ \label{ass:location} $f_{-a}(-x)=f_a(x)$ for $a,x\in\RR$, 
	$\overline{F}_a(x)$ is  increasing in $a\in\RR$ and, for $a>0$,
	\begin{equation}\label{eqn:monotonicity-location} 
		f_{a}(x)/f_0(x) \text { is increasing in $x\in\RR$.}
	\end{equation}
	\item$\hspace{-.2cm}.$ \label{ass:scale} 
	$f_a(-x)=f_a(x)$ for all $a,x\in \RR$, $\overline{F}_a(x)$ is increasing in $a>0$ for $x>0$, and 
	for $a\neq 0$
	\begin{equation}\label{eqn:monotonicity-scale}
		f_{a}(x)/f_0(x) \text { is increasing in $x>0$.}
		\end{equation}
	\end{enumerate}
\end{assumption}
Assumption~\ref{ass:generalnoise} was partly inspired by \cite{rabinovichpreprint}. It is explored in depth in Section~\ref{sec:VerificationOfAssumption} and we report essential facts below.

\begin{remark}\label{rem:ass1} 
Assumption~\ref{ass:generalnoise}\ref{ass:location} is designed for location models:  the symmetry property of the densities, and monotonicity and Lipschitz continuity of the distribution functions  are automatic in a symmetric location model where $f_a(x)=f_0(x-a),$ $f_0(x)=f_0(-x)$ for some bounded positive density $f_0$. Assumption~\ref{ass:generalnoise}\ref{ass:scale} is designed for scale models: the symmetry property of the densities and monotonicity and Lipschitz continuity of the distribution functions  are again automatic if $f_a(x)=f_0(x/\sqrt{1+\abs{a}})$ and $f_0(x)=f_0(-x)$ for some bounded positive density $f_0$.
Next,  \eqref{eqn:modelassumption2}--\eqref{eqn:modelassumption3} are conditions on the `null' distribution $\ol{F}_0$ that allow for a sharp description of the minimax risk on the boundary, as in the previous Gaussian case.
More precisely, \eqref{eqn:modelassumption2} means that the expected number of nulls larger than $a_n^*-\delta_n$ is much larger than $s_n$, while \eqref{eqn:modelassumption3} means that the expected number of nulls larger than $a_n^*$ is $o(s_n)$. Hence, up to a concentration argument, and assuming that a non-trivial portion of the $s_n$ true signals can be recovered, \eqref{eqn:modelassumption2} and \eqref{eqn:modelassumption3} entail that the FDR of a thresholding based procedure $\varphi_t=\II_{\abs{X_i}\geq t}$ has a sharp decay from $1$ to $0$ when $t$ increases from $a_n^*-\delta_n$ to $a_n^*$. Thus, under Assumption~\ref{ass:generalnoise}, the optimal threshold should be at $a_n^*+o(1)$. This sharp transition is an essential phenomenon in this setting; it will be further illustrated in Section~\ref{sec:illustrating-tradeoff}. 

\end{remark}

We now give two illustrative examples of settings in which Assumption~\ref{ass:generalnoise} holds.

\begin{example}\label{example:SubbotinLocation}
Assumption~\ref{ass:generalnoise}\ref{ass:location} holds with $a_n^*= (\zeta \log(n/s_n))^{1/\zeta}$ for the Subbotin (generalised Gaussian) location model, with $f_a$ denoting the law of $X=a+\eps$, $\eps\sim \phi_\zeta$,
\begin{equation}\label{eqn:def:Subbotin-density} \phi_{\zeta}(x) = L_\zeta^{-1} e^{-\abs{x}^\zeta/\zeta}, \quad \zeta>1;
\end{equation} 
see Lemma~\ref{lem:assumptionok}.
In particular, this model includes standard Gaussian noise as the case $\zeta =2$, while the excluded case $\zeta=1$ would correspond to Laplace noise.
In this setting we will write $\overline{\Phi}_\zeta (x)= \int_x^\infty \phi_\zeta(t)\dif t$ and $\Phi_\zeta=1-\overline{\Phi}_\zeta$.
Note that when we write $a_n^*$ in reference to the Subbotin model, we will mean the value $(\zeta \log(n/s_n))^{1/\zeta}$ except where specified otherwise.
\end{example}

\begin{example}\label{example:SubbotinScale}
Assumption~\ref{ass:generalnoise}\ref{ass:scale} holds with $a_n^*= (\zeta \log(n/s_n))^{1/\zeta}$ for the Subbotin scale model, with $f_a$ denoting the law of 
$X=(1+|a|)^{1/2}\eps$, $\eps\sim \phi_\zeta$, with $\phi_\zeta$ given by \eqref{eqn:def:Subbotin-density},
 see Lemma~\ref{lem:assumptionok}.
\end{example}

 In this section, we also consider a more general parameter set, with different signal strengths.
For a vector $\ba=(a_1,\dots,a_{s_n})$ (implicitly indexed by $n$) with $a_j>0$ for $1\leq j\leq s_n$, define 
\begin{equation}\label{parametersetmultiscale}
\Theta(\ba,s_n)=\braces[\Big]{ \theta\in \ell_0[s_n] \: : \: 
	\exists\,  
	i_1, \ldots,  i_{s_n} \mbox{ all distinct, }  \abs{\theta_{i_j}} \geq a_j,\ 1\leq j \leq s_n}.
	\end{equation}
	Note this generalises the definition \eqref{singlesignal}: we recover $\Theta(\ba,s_n)=\Theta(a,s_n)$ if $\ba=(a,a,\dots,a)$. In general, we have only the inclusion $\Theta(\ba,s_n)\subseteq \Theta(\min_{1\leq j\leq s_n} a_j,s_n)$, so that this definition allows for considering a more subtle minimax tradeoff. A key quantity is
	\begin{equation}\label{eqn:LambdaN}
		\Lambda_n (\ba)=s_n^{-1} \sum_{j =1 }^{s_n} F_{a_j} \left(a_n^* \right);
	\end{equation}
	this can be seen roughly as a measure of the (lack of) signal strength for the parameter set $\Theta(\ba,s_n)$. 
In a slight abuse of notation, if $\theta$ is a vector whose non-zero entries have absolute values $a_1,\dots,a_{s_n}$ 
and recalling the notation $S_\theta=\braces{i: \theta_i\neq 0}$, we also write
\begin{equation}\label{eqn:def:Lambda_n(theta)}
	\Lambda_n(\theta) = 
	s_n^{-1}\sum_{i\in S_\theta} F_{\abs{\theta_i}}(a_n^*).
\end{equation}

We show in Theorem~\ref{th:minimax-multilevel} below that the asymptotic behaviour of $\Lambda_n(\ba)$ drives the minimax risk over the class $\Theta(\ba,s_n)$. We now give three examples of location models to illustrate how different signal shapes and values affect $\Lambda_n(\ba)$.

\begin{example}[Location model with single signal strength close to $a_n^*$]\label{ex:singlesignalstrengths}
A simple case is the Subbotin location model of Example~\ref{example:SubbotinLocation}, with  $\theta_i =  a_n^*+b$ for all $i\in S_\theta$, for some fixed $b\in \R$. Then $\theta$ lies on the boundary of the beta-min parameter set and $\Lambda_n$ in \eqref{eqn:def:Lambda_n(theta)} reads $\Lambda_n(\theta) = \overline{\Phi}_{\zeta}(b)$. This corresponds to the  first case in Figure \ref{fig:mt_pb}. 
\end{example}


\begin{example}[Location model with two signal strengths near $a_n^*$]\label{ex:twosignalstrengths}
	For fixed $x,y\in \R^2$, let 
	$M=\max(x,y),$ $m=\min(x,y)$.
	Consider the Subbotin location model of Example~\ref{example:SubbotinLocation}, with $\theta_i = a_n^*+M$ for $\lfloor s_n \beta\rfloor$ coefficients in $S_\theta$ and  $\theta_i=a_n^*+m$ for the other coefficients in $S_\theta$, where
	$\beta\in (0,1)$ is a given proportion of stronger signal. 
	Then inserting $F_a=\Phi_\zeta(\cdot-a)$ in \eqref{eqn:def:Lambda_n(theta)} we obtain 	
		\begin{align*} 
		\Lambda_n(\theta)&=\frac{\lfloor s_n \beta\rfloor}{s_n} \overline{\Phi}_\zeta(M)+\frac{s_n-\lfloor s_n \beta\rfloor}{s_n} \overline{\Phi}_\zeta(m)\\
		& \to \beta \overline{\Phi}_\zeta(M)+(1-\beta) \overline{\Phi}_\zeta(m) =: \Lambda_{\infty},
	\end{align*}
as $n\to \infty$.	
This corresponds to the second picture of Figure \ref{fig:mt_pb}. Level sets of $\Lambda_{\infty}=\Lambda_{\infty}(x,y)$ for this $\theta$ 
	are displayed in Figure~\ref{fig:generalboundary} below. 
\end{example}

In the previous examples, we have taken signals around the critical threshold $a_n^*$. We can allow more generally for arbitrary nonzero signals.

\begin{example}[Gaussian location model, mixed signal] Consider to fix ideas the Gaussian
model of Example \ref{ex:GaussianSequence} and suppose the nonzero entries of $\theta$ are given by $(Ai/s_n)\sqrt{2\log(n/s_n)}$ for $i=1,\ldots,s_n$ and $A\ge 1$ fixed. This is a special example of the third case in  Figure \ref{fig:mt_pb}. Then 
\[ \La_n(\theta) = A^{-1} + o(1). \]  
(to check this, one can for example separate signal coordinates $i$'s into three subsets delimited by $(s_n/A)(1\pm r_n)$ for $r_n=o(1)$ suitably slowly). In this example, what contributes to $\La_n(\te)$  is the proportion of signals below the (asymptotically) optimal threshold $a_n^*=\sqrt{2\log(n/s_n)}$.
\end{example}


\subsection{~~Main results in the general setting}
\label{sec:genresults}

We extend Theorems~\ref{thmgeneric},~\ref{thm-adapt-r} and \ref{thm-notrade}, respectively, to the more general noise and parameter set 
 introduced in this section.

\begin{theorem}\label{th:minimax-multilevel}
Consider the sparse sequence model \eqref{eqn:generalnoisemodel}--\eqref{sparse} with Assumption~\ref{ass:generalnoise}. 
Consider a vector $\ba=(a_1,\dots,a_{s_n})\in\RR_+^{s_n}$, $\Theta(\ba,s_n)$ defined by \eqref{parametersetmultiscale} and $\Lambda_n(\ba)$ defined by \eqref{eqn:LambdaN}. Then
\begin{itemize} 
\item Under Assumption~\ref{ass:generalnoise}\ref{ass:location}, 
\[\inf_\vphi \sup_{\theta \in \Theta(\ba,s_n)} \mathfrak{R}(\theta,\vphi) = \Lambda_n(\ba) +o(1);\]
\item Under Assumption~\ref{ass:generalnoise}\ref{ass:scale}, 
\[\inf_\vphi \sup_{\theta \in \Theta(\ba,s_n)} \mathfrak{R}(\theta,\vphi) = 2\Lambda_n(\ba)-1 +o(1).\]
\end{itemize}
In each case the risk bound is achieved by a thresholding procedure $\vphi_i(X)=\II\braces{\abs{X_i}>a_n^*}$ with $a_n^*$ as in Assumption~\ref{ass:generalnoise}.
\end{theorem}
Theorem~\ref{th:minimax-multilevel} in particular applies in the Subbotin location model of Example~\ref{example:SubbotinLocation} with $\overline{F}_0=\overline{\Phi}_\zeta$, $a_n^*=(\zeta\log(n/s_n))^{1/\zeta}$ and a suitable choice of $\delta_n$, see Lemma~\ref{lem:assumptionok}. 

 {To prove this theorem, an easy heuristic can be built from Remark~\ref{rem:ass1}, since by the reasoning therein only a threshold $t=a_n^*+o(1)$ leads to non-trivial FDR and the FNR for that threshold is precisely of order $\Lambda_n(\ba) $ and $2\Lambda_n(\ba)-1$ (under Assumption~\ref{ass:generalnoise}\ref{ass:location} and ~\ref{ass:generalnoise}\ref{ass:scale}, respectively). 
 However, the rigorous proof is significantly more involved; in particular, the lower bound argument is more sophisticated, because we do not restrict to thresholding procedures, see Section~\ref{sec:proofs}.}


Recall our motivating question from Figure~\ref{fig:mt_pb}, of whether all nonzero signals being equal (left panel) is easier or harder for multiple testing than half taking one value and half another value (middle panel). 
Applying Theorem~\ref{th:minimax-multilevel} in Example~\ref{ex:twosignalstrengths}, the asymptotic difficulty of testing is governed by the quantity $\Lambda_\infty=\Lambda_\infty(x,y)$. See Section~\ref{sec:illustrating-tradeoff} for some examples of how $\Lambda_\infty$ varies with $x$ and $y$, but note immediately that $\Lambda_\infty$ is smaller than the bound $\overline{\Phi}_\zeta(m)$ attainable using that $\Theta((a_n^*+m,\dots,a_n^*+M,\dots),s_n)$ is a subset of $\Theta(a_n^*+m,s_n)$.

The procedure exhibited to prove the upper bound in Theorem~\ref{th:minimax-multilevel}, like that of Theorem~\ref{thmgeneric}, is an `oracle' procedure requiring knowledge of the $a_n^*$, which typically depends on the sparsity $s_n$. 
The next result gives an upper bound on the pointwise and uniform risk of the Benjamini--Hochberg procedure under Assumption~\ref{ass:generalnoise}; this will yield adaptivity to $s_n$ in many settings. Let us introduce the following condition on the level $\alpha=\alpha_n$ of the BH procedure: 
\begin{equation}\label{equcondalpha}
\frac{3n\overline{F}_0(a_n^*)}{s_n(1-\Lambda_n(\theta))} 
\le \alpha_n\leq \min\braces[\Big]{1,\frac{n\overline{F}_0(a_n^*-\delta_n)}{s_n}}\,\mbox{ for $n$ large enough}.
\end{equation}
\begin{theorem} \label{thm:BHpointwise}
	Consider the sparse sequence model \eqref{eqn:generalnoisemodel}--\eqref{sparse} and grant Assumption~\ref{ass:generalnoise}\ref{ass:location}. 
Then the BH procedure taken at a level $\alpha=\alpha_n$ obeying \eqref{equcondalpha} satisfies, for all $\theta\in \R^n$ with $\abs{S_\theta}=s_n$, for $n$ large enough,
\begin{equation}\label{eqn:BHrisk} \fr(\te,\vphi) \leq \La_n(\theta) + \alpha_n + \exp\brackets[\Big]{-\frac{(1-\Lambda_n(\theta))^2}{32}s_n}.\end{equation}
It follows that the BH procedure achieves the bound of Theorem~\ref{th:minimax-multilevel}, 
\[ \sup_{\theta \in \Theta(\ba,s_n)} \mathfrak{R}(\theta,\vphi) \leq \Lambda_n(\ba) + o(1).\]
In addition, if $\Lambda_n(\ba)$ is bounded away from 1 we have the concrete expression $\alpha_n + \exp\brackets[\big]{-(1-\Lambda_n(\ba))^2 s_n/32}$ for the $o(1)$ terms.
Finally, if we instead grant Assumption~\ref{ass:generalnoise}\ref{ass:scale}, the same bounds hold with $\Lambda_n(\cdot)$ replaced by $2\Lambda_n(\cdot)-1$.

\end{theorem}

Theorem~\ref{thm:BHpointwise} is proved in Section~\ref{sec:proofBHfirsth}. It implies that the risk bound of Theorem~\ref{th:minimax-multilevel} can be attained adaptively to $s_n$ in any case where a valid level $\alpha_n$ can be chosen. 
Typically, the condition \eqref{equcondalpha} is achieved for $\Lambda_n(\theta)$ bounded away from 1 by choosing $\alpha_n\to 0$ with $(n/s_n)\overline{F}_0(a_n^*)/\alpha_n\to 0$ (and noting that $(n/s_n)\overline{F}_0(a_n^*-\delta_n)\geq 1$ for $n$ large). 
	For instance, 
 in the Subbotin case with parameter $\zeta$, under the assumption of polynomial sparsity \eqref{polspa}, condition \eqref{equcondalpha} can be achieved if $\alpha_n=o(1)$ satisfies $\log(1/\alpha_n) = o (( \log n )^{1-1/\zeta})$, see Lemma~\ref{lem:assumptionok} and Remark~\ref{rem:polysparse}. 
 In addition, choosing $\alpha_n\asymp (\log n)^{-(1-1/\zeta)}$ leads to the convergence rate
 $
 \brackets[\big]{\fr(\te,\vphi) - s_n^{-1}\sum_{i\in S_\theta} F_{\abs{\theta_i}}((\zeta \log(n/s_n))^{1/\zeta})} \lesssim\frac{\log\log n}{(\log n)^{1-1/\zeta}} + e^{-cs_n},
 $
 for some constant $c>0$, see again Remark~\ref{rem:polysparse}. Note that this shares similarities with the rate obtained in \cite{neuvialroquain12}  for the classification risk of BH procedure.

\begin{remark}
	The inequality in \eqref{eqn:BHrisk} is in fact an \emph{equality}, up to $o(1)$ terms, so that the result quantifies fairly precisely the risk actually obtained by the Benjamini--Hochberg procedure in a sparse sequence model setting: see Remark~\ref{rem:BH-pointwise-lowerbound} after Theorem~\ref{thm:pointwise-risk}. 
\end{remark}

%

Let us remark that the $\ell$-value procedure also achieves the bound in the Gaussian case with multiple signal levels: see Section~\ref{sec:proof-lvals}. We conjecture that the $\ell$-value procedure also achieves the bound in the Subbotin model, and possibly even in the general noise model under Assumption~\ref{ass:generalnoise}, but proving this would require substantial extra technical work.


The following result, proved in Section~\ref{thm-notrade-Subbotin}, shows that the FNR is the dominating term at the boundary when focusing on sparsity preserving procedures even in the general framework of this section. 
Recall that the BH procedure is sparsity preserving (proved in Section~\ref{sec:sparsitypreserving}), hence achieves the bounds of Theorem~\ref{thm-notrade-Subbotin} adaptively  under the conditions of Theorem~\ref{thm:BHpointwise}.


\begin{theorem}\label{thm-notrade-Subbotin}
Consider the settings of Theorem~\ref{th:minimax-multilevel}. 
 For any sequence $B=(B_n)_n$ with $B_n^2\leq \tfrac{1}{3}(n/ s_n)\overline{F}_0(a_n^*-\delta_n)$ and $\varliminf_n B_n >1$, %
 let $\mathcal{S}_{B}$ denote the set of sparsity preserving procedures over the set $\Theta=\Theta(\ba,s_n)$, as in Definition~\ref{def:spapres}.
 Then the conclusions of Theorem~\ref{th:minimax-multilevel} hold with $\mathfrak{R}(\theta,\vphi)$ replaced by $\FNR(\theta,\vphi)$ if the infimum is taken only over $\vphi\in \mathcal{S}_{B}$.
\end{theorem}
In the Subbotin case with parameter $\zeta>1$, under the assumption of polynomial sparsity \eqref{polspa}, the condition $B_n^2\leq \tfrac{1}{3}(n/ s_n)\overline{F}_0(a_n^*-\delta_n)$ is satisfied if $B_n^2 \leq \exp((\log n)^\upsilon)$ for some $\upsilon\in (0,1-1/\zeta)$, see Remark~\ref{rem:polysparse}.

Let us finally mention that this result entails the sub-optimality of any sparsity preserving procedure with FDR asymptotically above some $\alpha>0$. That is, Corollary~\ref{cor-subopt} also extends to the more general framework of this section.

\subsection{~~Testability}\label{sec:testability}
While Theorem~\ref{th:minimax-multilevel} is stated as a minimax result over the set $\Theta$, it is possible to formulate results for a given collection of signal strengths only, in a `pointwise in $\theta$' sense. This completely solves the question of {\em characterising} the difficulty of the testing problem for an arbitrary sparse vector $\theta$. We postpone the rigorous statements to the Section~\ref{sec:local}, see Theorem ~\ref{thm:pointwise-risk} therein. Given its importance we formulate a corollary thereof.
	 
\begin{corollary}[Difficulty of testing for arbitrary sparse $\theta$] \label{cortest}	
Fix $\alpha\in(0,1)$.	Suppose we know $\theta\in\ell_0[s_n]$ has $s_n$ nonzero coordinates with {\em arbitrary} absolute values $\abs{\theta_j}, j\in S_\theta$, whose \emph{values} may be known or unknown but whose \emph{positions and signs} are unknown. Then there exists a multiple testing procedure with $\fr$--risk asymptotically less than $\alpha$  for this problem  if and only if, for  $a_n^*$ as in Assumption~\ref{ass:generalnoise}, and under~\ref{ass:generalnoise}\ref{ass:location},
	\[ \varlimsup_{n}\ \frac{1}{s_n} \sum_{j\in S_{\theta}} F_{\abs{\theta_j}}(a_n^*) \le \alpha,\]
and respectively smaller than or equal to $(1+\alpha)/2$ under~\ref{ass:generalnoise}\ref{ass:scale}. 
\end{corollary}
	Hence, given different arbitrary shapes of signals (again as in Figure~\ref{fig:mt_pb}), to compare the difficulty of the corresponding multiple testing problems, it suffices to compare the corresponding limsups in the last display: the value of the limsup characterises the (asymptotic) difficulty of the multiple testing problem for any \emph{arbitrary} $\theta$ with $s_n$ nonzero coordinates.

\section{~~Illustrations} \label{sec:simu}

In this section, we present numerical illustrations for the results stated above, as well as provide some intuition behind these.

\subsection{~~Illustrating FDR/FNR tradeoff} \label{sec:illustrating-tradeoff}

Theorems~\ref{thm-notrade}~and~\ref{thm-notrade-Subbotin} establish that the FNR is asymptotically the driving force for optimal procedures: any optimal (sparsity preserving) procedure has an FDR vanishing when $n$ goes to infinity and only the FNR matters in the minimax risk.

\begin{figure}[h!]
\begin{center}
\begin{tabular}{cc}
$\zeta=1.5$ & $\zeta=2$\vspace{-1cm}\\
\includegraphics[scale=0.5]{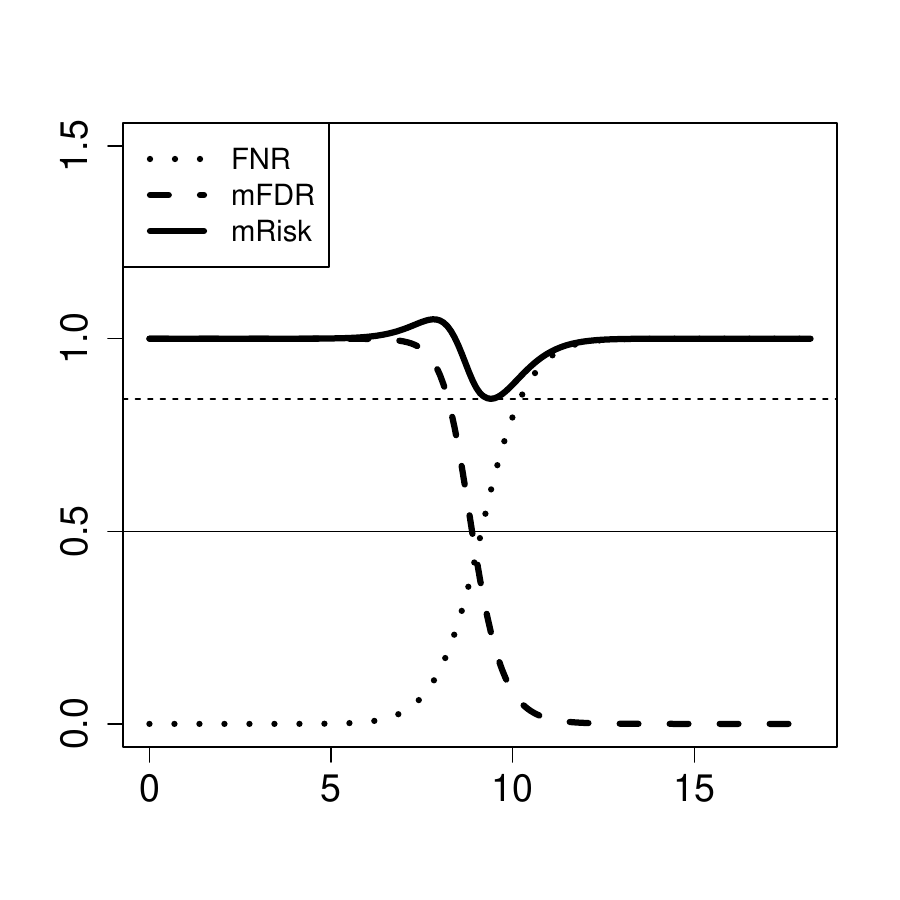}&
\includegraphics[scale=0.5]{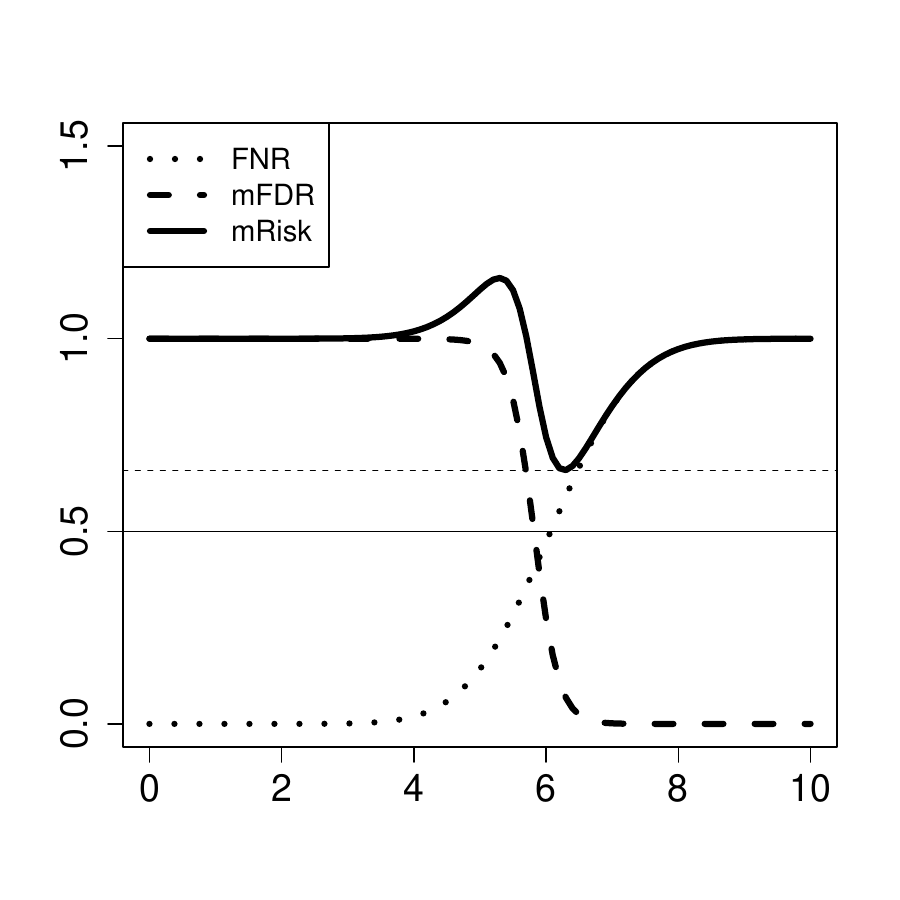}\vspace{-5mm}\\
$\zeta=3$ & $\zeta=5$\vspace{-1cm}\\
\includegraphics[scale=0.5]{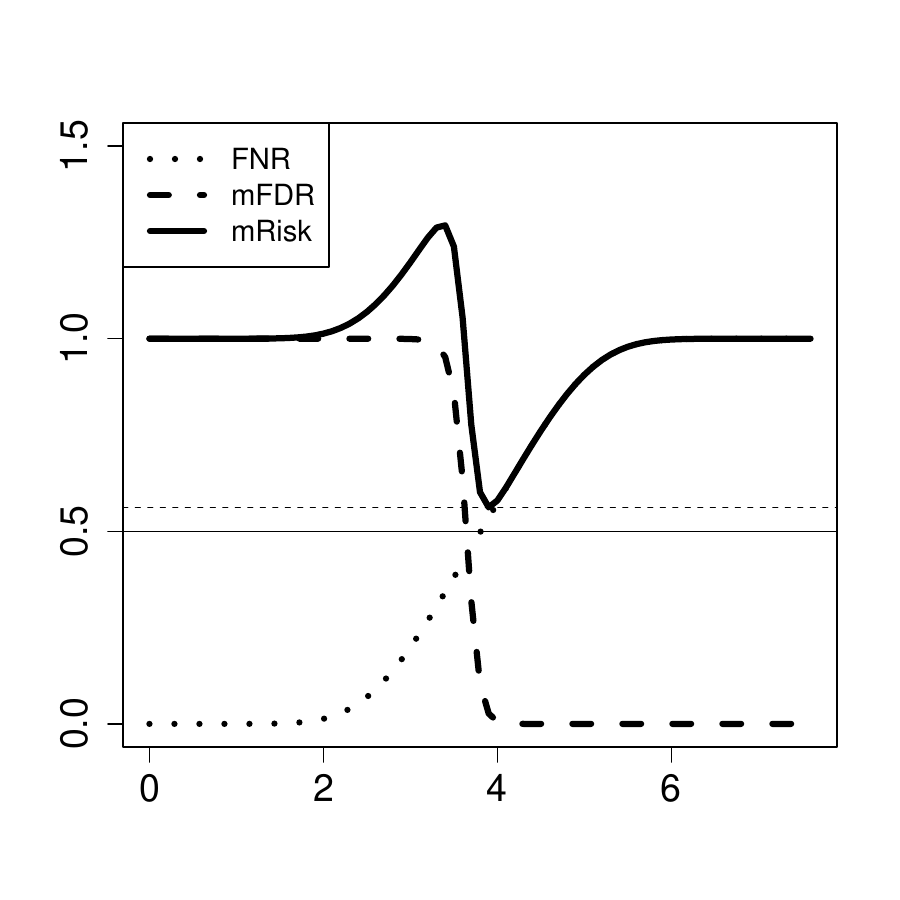}&
\includegraphics[scale=0.5]{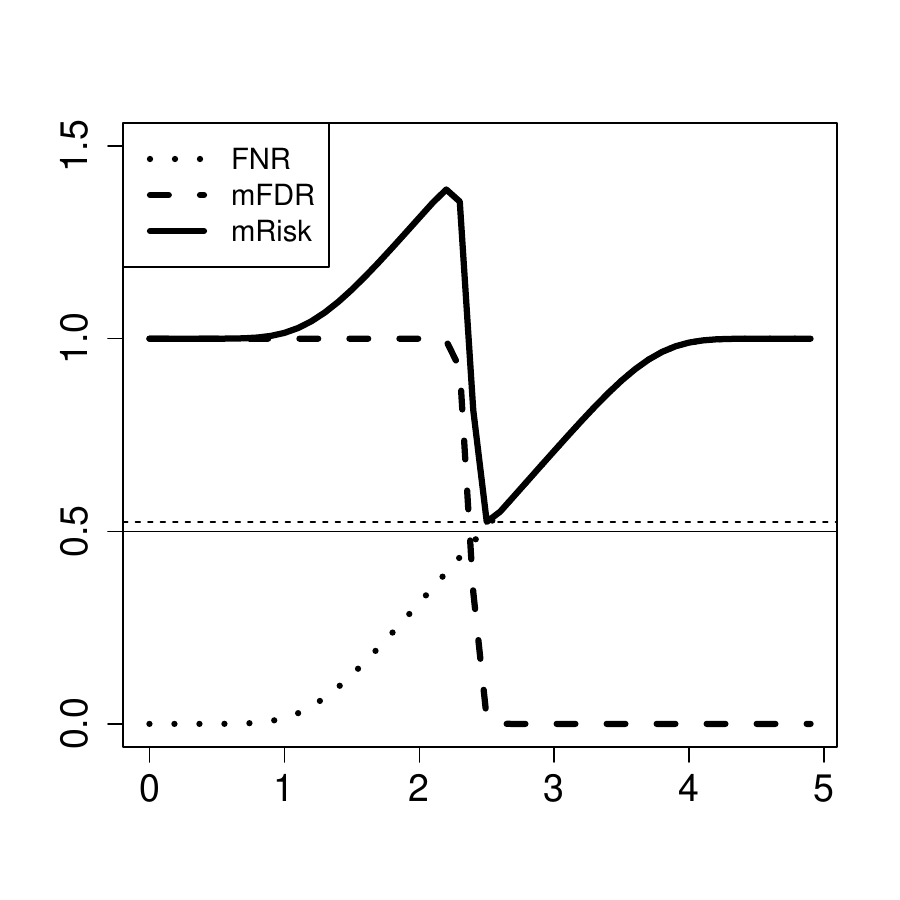}\vspace{-5mm}
\end{tabular}
\end{center}
\caption{Finite-sample mFDR \eqref{equ:mFDRt}, FNR \eqref{equ:FNRt} and $\mR$ \eqref{equ:mRt} for a thresholding procedure $\varphi_t=\II_{\abs{X_i}\geq t}$, as a function of the threshold $t$ ($x$-axis). Here, $n=10^{10}$, $s_n=10^2$. 
 The solid horizontal line marks the asymptotic minimax risk of $1/2$, and is approached by the dashed horizontal denoting $\mR(\vphi)$ for $\vphi$ an optimal thresholding procedure.
 \label{fig:mFDRvsFNR}}
\end{figure}

This phenomenon is illustrated in Figure~\ref{fig:mFDRvsFNR} for thresholding based procedures $\varphi_t=\II_{\abs{X_i}\geq t}$, $t\in \RR$.
For ease of computations, instead of the FDR, we consider here the closely related  marginal FDR (corresponding to the ratio of expectations in \eqref{deffdr}), denoted by mFDR. 
Hence, in the Subbotin setting of Example~\ref{example:SubbotinLocation}, we consider the marginal  risk 
\begin{align}
\mR(\varphi_t)&=\mFDR(\varphi_t)+\FNR(\varphi_t);\label{equ:mRt}\\
\FNR(\varphi_t)&= \overline{\Phi}_\zeta(t-a_n^*) + \overline{\Phi}_\zeta(t+a_n^*);\label{equ:FNRt}\\
\mFDR(\varphi_t)&=\frac{2(n-s_n)\overline{\Phi}_\zeta(t)}{2(n-s_n)\overline{\Phi}_\zeta(t)+s_n (1-\FNR(\varphi_t))},\label{equ:mFDRt} 
\end{align}
for which the parameter $\theta$ is chosen on the border of $\Theta_b=\Theta(a_b,s_n)$ (for $b=0$) as follows: $\theta_i=a_n^*=\{\zeta\log{n/s_n}\}^{1/\zeta}$ if $1\leq i\leq s_n$ and $\theta_i=0$ otherwise. (It can be shown that our main theorems also hold with the marginal risk $\mR$ instead of the original risk $\fr$.)
 
 We can make the following general comments: first, as the threshold $t$ increases (along the $x$-axes in Figure~\ref{fig:mFDRvsFNR}), fewer rejections are made, so that $\FNR(\varphi_t)$ increases with $t$ and $\mFDR(\varphi_t)$ decreases, hence there is a tradeoff for the finite sample (marginal) risk $\mR(\varphi_t)$: the minimum is displayed with the light dashed horizontal line. 
 This finite sample risk can be compared to its asymptotic counterpart (thin solid horizontal line).  We note that the asymptotic regime is not reached for the current choice of $n$, $s_n$ if $\zeta=1.5$, but is approached as $\zeta$ increases, with almost a  perfect matching when $\zeta=5$. This is well expected from the convergences rates, for example obtainable in Theorem~\ref{thm:BHpointwise}, that are faster for $\zeta$ larger
 so that the case $\zeta=5$ well approximates the  asymptotic picture.  
 Second, when the asymptotics are (close to being) reached, we see that the dominating part in the risk at the point where the risk is minimum is the FNR, as Theorems~\ref{thm-notrade}~and~\ref{thm-notrade-Subbotin} establish. The pictures above give an  interpretation of these results: 
 since the transition from $0$ to $1$ is much more abrupt for the (m)FDR,  we should make the (m)FDR close to $0$ to make a good tradeoff; the mFDR is close to a step function, so achieving $\mFDR < 1-\eps$ requires essentially the same threshold as achieving $\mFDR< \eps$, hence to make the optimal tradeoff will require mFDR close to zero.

This sharp transition can be seen theoretically from Assumption~\ref{ass:generalnoise} (see also Remark~\ref{rem:ass1}); indeed it follows from \eqref{equ:FNRt}--\eqref{equ:mFDRt} that the mFDR of the $\vphi_t$ is, with $\overline{F}_0 = \overline{\Phi}_\zeta$,
\[ \frac{1}{1+ \left[1-\overline{F}_0(t-a_n^*)+o(1)\right] \left[2 ((n-s_n)/s_n)\overline{F}_0(t)\right]^{-1} }.\] 
By \eqref{eqn:modelassumption2}, this tends to zero for $t=a_n^*$, and by \eqref{eqn:modelassumption3}, it tends to 1 for $t=a_n^*-\delta_n$.


\subsection{~~Illustrating minimax risks for two signal strengths}
%
The asymptotic minimax risk 
 found in Section~\ref{sec:mainmult} is a function of the multiple signal strengths that we propose to illustrate in this section. For simplicity, we focus on the case of the Subbotin location model with two signal strengths, see Example~\ref{ex:twosignalstrengths}, which corresponds to considering the functional
\begin{equation}\label{Lambdainftyex}
 \Lambda_{\infty}:(x,y)\in \R^2 \mapsto   \Lambda_{\infty}(x,y)=\beta \overline{\Phi}_\zeta(\max(x,y))+(1-\beta) \overline{\Phi}_\zeta(\min(x,y)).
\end{equation}



\begin{figure}[h!]
\begin{center}
\begin{tabular}{ccc}
\vspace{-1cm}
\rotatebox{90}{\hspace{-3cm}$\zeta=2$}&$\beta=1/2$ &$\beta=1/4$\\
\vspace{-2cm}
&\hspace{-5mm}\includegraphics[scale=0.6]{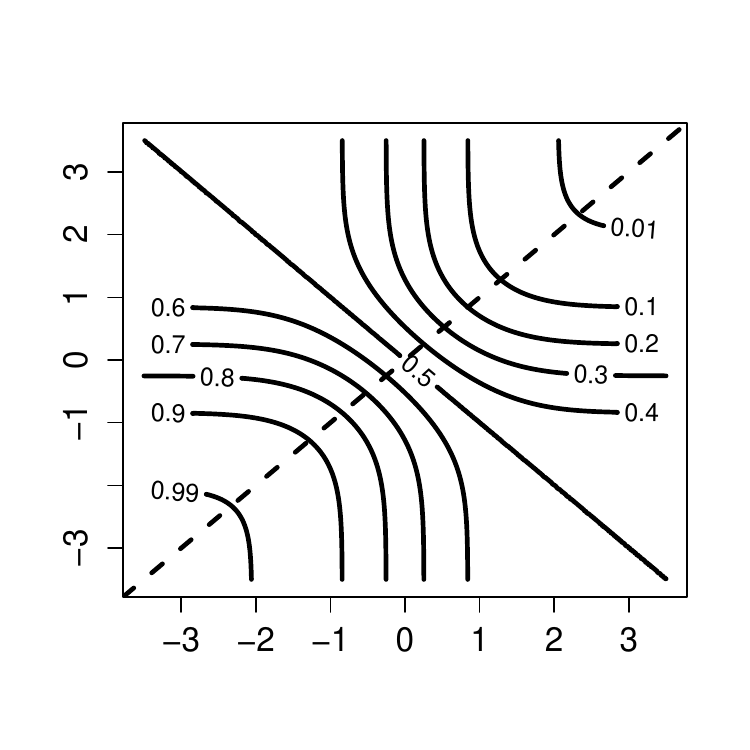}&\hspace{-1cm}\includegraphics[scale=0.6]{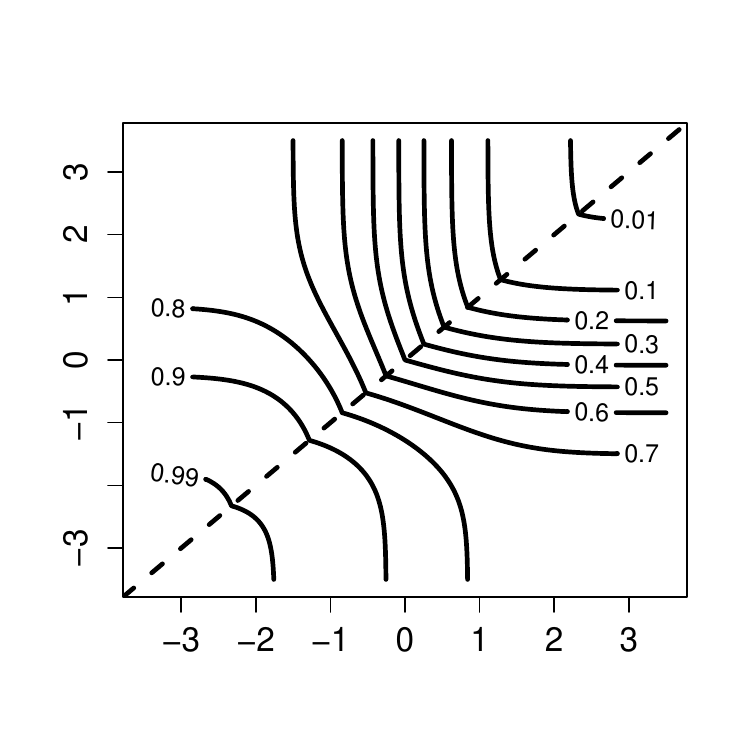}\\
\vspace{-1cm}
\rotatebox{90}{\hspace{3.6cm}$\zeta=4$}&\hspace{-5mm}\includegraphics[scale=0.6]{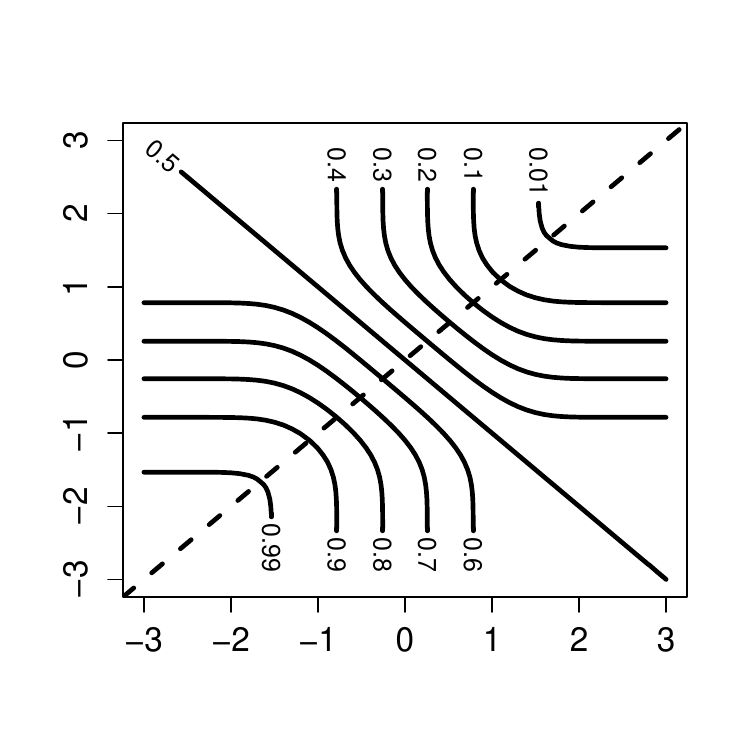}&\hspace{-1cm}\includegraphics[scale=0.6]{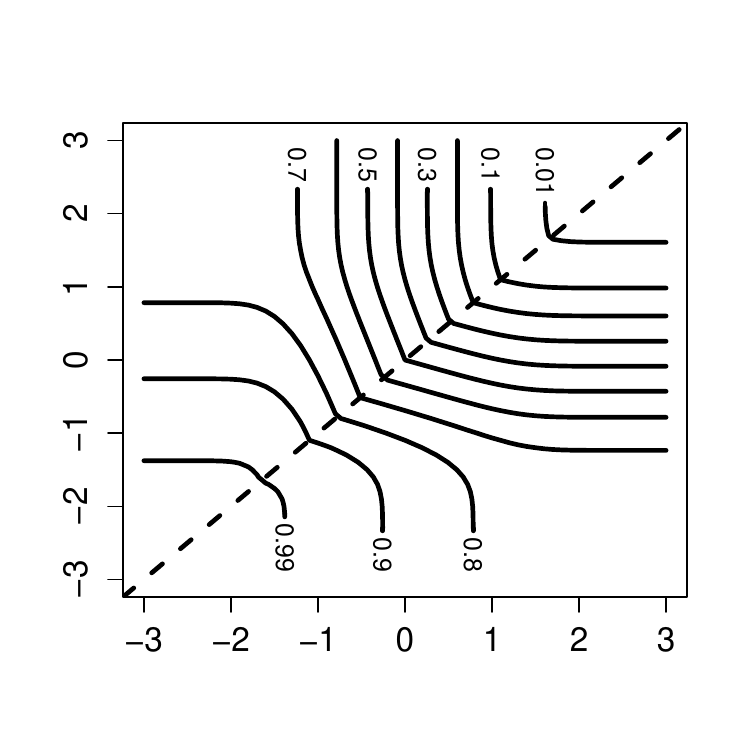}
\end{tabular}
\caption{Level sets of the asymptotic boundary, that is, of the function $\Lambda_\infty(x,y)$ as defined by \eqref{Lambdainftyex} in the case of two signal strengths, with $\lfloor s_n \beta\rfloor$ nonzero means equal to $\{\zeta\log{n/s_n}\}^{1/\zeta}+\max(x,y)$ and $s_n-\lfloor s_n \beta\rfloor$ nonzero means equal to $\{\zeta\log{n/s_n}\}^{1/\zeta}+\min(x,y)$. Left: $\beta=1/2$; right: $\beta=1/4$. Top: Gaussian noise; Bottom: Subbotin noise with $\zeta=4$. \label{fig:generalboundary}}
\end{center}
\end{figure}
Level sets of this function are displayed  in Figure~\ref{fig:generalboundary} for $\beta\in \{1/4,1/2\}$ and $\zeta\in \{2,4\}$. 
As we can see, while the behaviour on the diagonal $x=y$ matches that of the single signal strength case, the asymptotic minimax risk has various behaviours off this diagonal, when there are two different signal strengths. 
 For instance, when $\beta=1/2$, we see that the signal strengths $(x,-x)$ and $(0,0)$ are equally risky ($\La_\infty(0,0)=\La_\infty(x,-x)$ for $x>0$) while $(1,3)$ is more risky than $(2,2)$ (this can be shown by explicit computation or convexity).
In general, the shape of the level sets depends on the proportion $\beta$ and on $\zeta$. 
First, $\beta$ strongly affects the level sets: $\Lambda_{\infty}$ is larger for $\beta=1/4$ than for $\beta=1/2$ off the diagonal $x=y$. This is to be expected because when $\beta=1/4$, only a proportion $1/4$ of the nonzero means are equal to the larger value $\{\zeta\log{n/s_n}\}^{1/\zeta}+\max(x,y)$ while the other nonzero values are equal to $\{\zeta\log{n/s_n}\}^{1/\zeta}+\min(x,y)$. This is clearly a less favorable situation compared to the case where the proportions are balanced ($\beta=1/2$). 
Second, $\zeta$ also affects the level sets (although less severely): increasing $\zeta$ makes the level sets flatter in the center of the $(x,y)$-picture. This difference comes from the fact that the tails of the $\zeta$-Subbotin distribution are lighter for larger values of $\zeta$.


\begin{figure}[h!]
\begin{center}
\hspace{-10mm}
\includegraphics[scale=0.45]{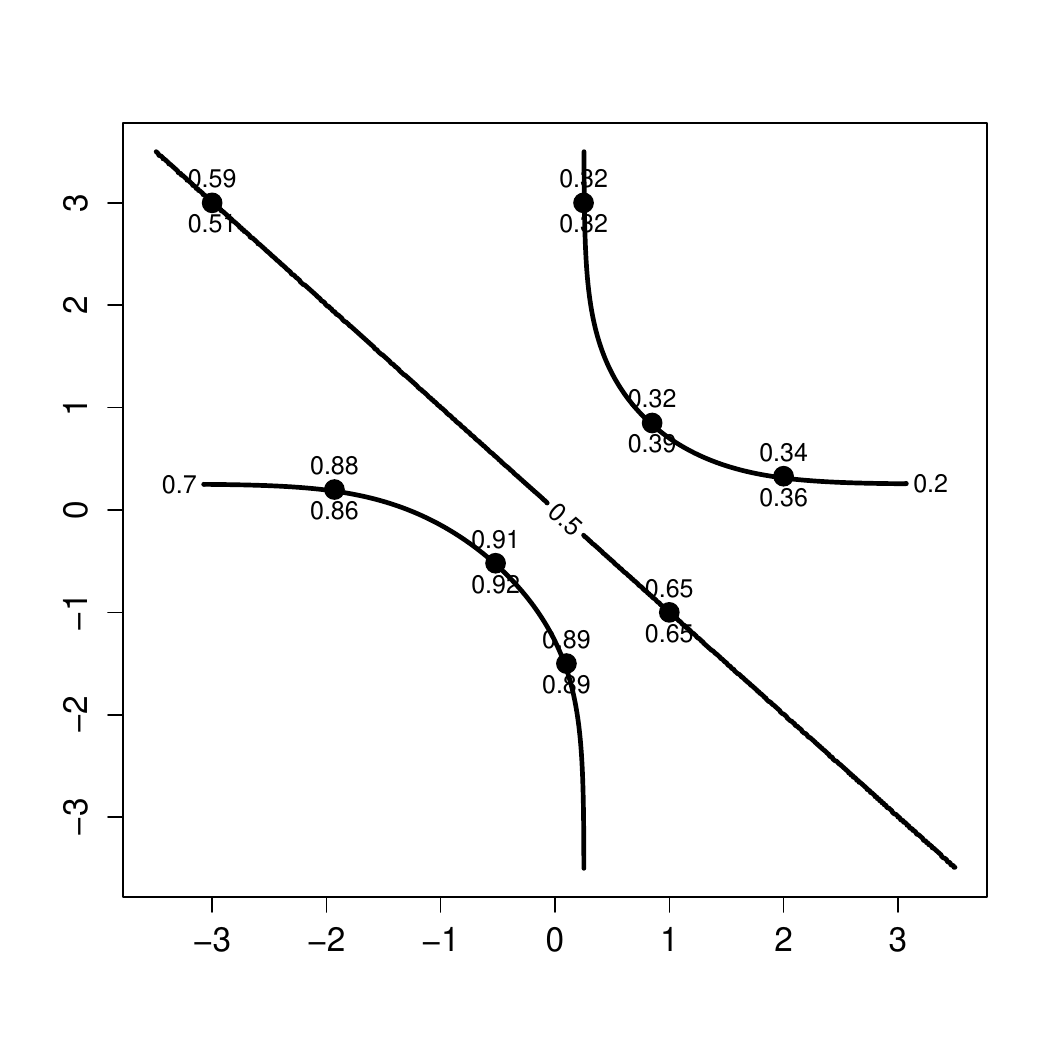}
\hspace{-10mm}
\vspace{-1cm}
\includegraphics[scale=0.45]{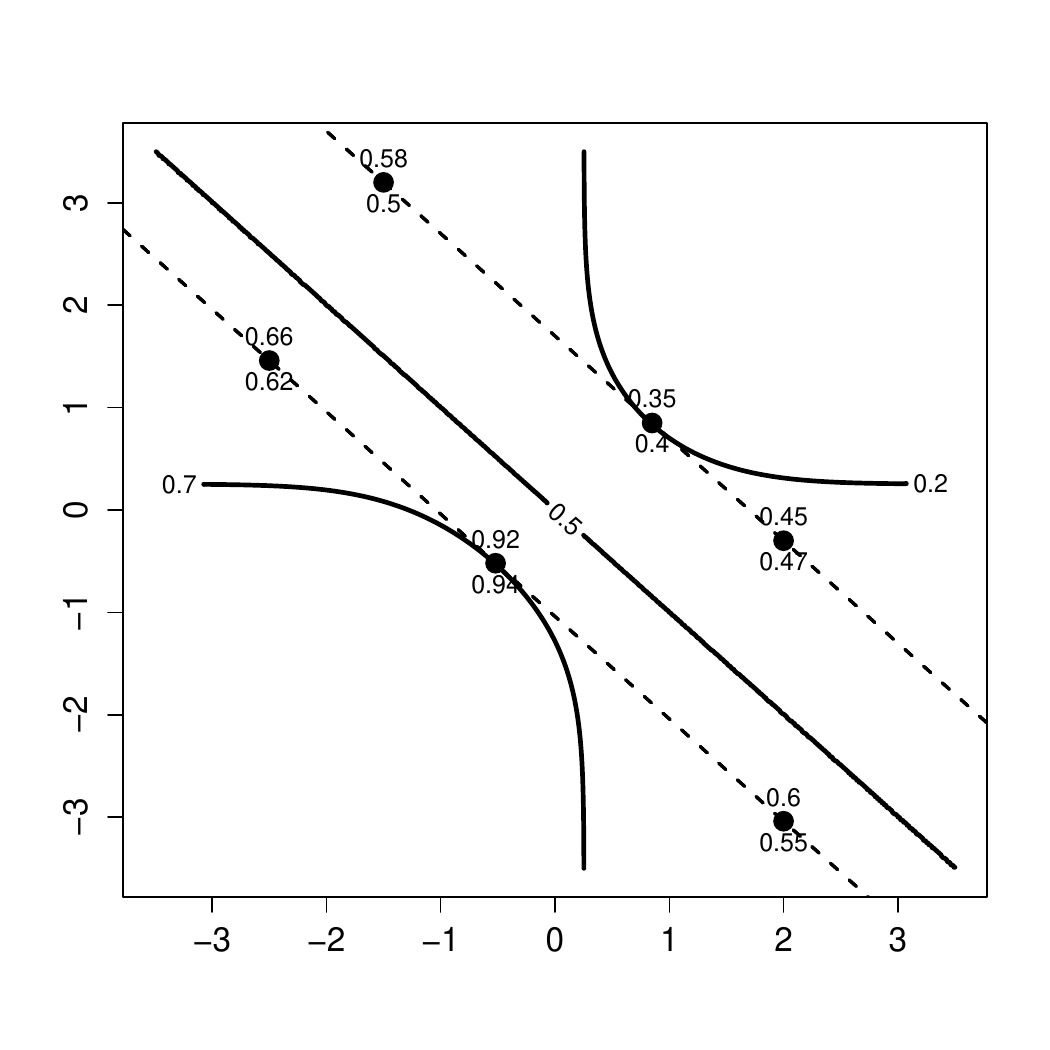}
\caption{Same as Figure~\ref{fig:generalboundary} (top-left, Gaussian case with $\beta=1/2$) with finite sample risk ($n=10^6$, $s_n=20$) at some particular configurations (displayed by black dots). The risk is computed via $100$ Monte-Carlo simulations. Two (asymptotically) minimax procedures are implemented:  the $\ell$-value procedure ($t=0.3$, risk displayed below each dot) and the BH procedure ($\alpha=0.1$, risk displayed above each dot). Left: dots are located on the $\Lambda_\infty$-level sets of values in $\{0.7,0.5,0.2\}$ of the boundary function. Right: dots are located on lines such that the average of  $x$ and $y$ is kept constant (not $\Lambda_\infty$-level sets). \label{fig:generalboundarywithpoints}}
\end{center}
\end{figure}

Figure~\ref{fig:generalboundarywithpoints} further illustrates how the (asymptotic) level sets of $\Lambda_{\infty}$ are approached for large finite samples: it displays some finite-sample risks  $\mathfrak{R} (\theta,\vphi)$ ($n=10^6$, $s_n=20$) of the minimax procedures described in Section~\ref{sec:popular}: $\ell$-value for $t=0.3$ and BH procedure at level $\alpha=0.1$, for some parameter $\theta$ corresponding to configurations  $(x,y)$ taken from the Gaussian and $\beta=1/2$ plot (top-left panel of Figure~\ref{fig:generalboundary}, see again Example~\ref{ex:twosignalstrengths}). 
%
These pointwise risks are estimated via $100$ Monte-Carlo simulations.
As expected from the slow convergence already discussed above, the finite sample risk has not yet converged to the minimax risk. However, we can see that the global monotonicity of the risk is maintained. More precisely, the left panel of Figure~\ref{fig:generalboundarywithpoints} displays the finite sample risks at points taken along asymptotic level sets. We can see that the values reported for the finite sample risks are  near-constant along asymptotic level sets (up to the Monte-Carlo errors), suggesting that the shapes of the finite sample level sets are close to those of their asymptotic counterparts. One exception for which the convergence appears to be slower is for configuration $(x,y)=(-3,3)$: non-asymptotically this configuration seems to be easier than for example $(x,y)=(1,-1)$ (although asymptotically equivalent). One possible explanation is that an extreme value of $y$ makes it easily detectable for finite $n$. 
On the other hand, the right panel of Figure~\ref{fig:generalboundarywithpoints} displays the finite sample risks for points taken along lines which are not asymptotic level sets. The results are markedly different from the left panel: the values vary much more along these lines. Hence, these simple lines, which are not asymptotic level sets,  are also far from being finite-sample levels sets. This suggests that the functional $\Lambda_{\infty}$ can be used for comparing testing difficulties of given collections of signal strengths  for $n$ large.

\section{~~Extensions to other risks: FDR-controlling procedures and classification} \label{sec:orisk}
Results in this section are for simplicity stated in the Gaussian location model under beta-min conditions of Section~\ref{sec:main}, but versions also exist in the general  setting of Section~\ref{sec:mainmult}.
\subsection{~~Combined testing risk with given FDR control}
The following proposition is easily obtained from our results, but particularly relevant for multiple testing since common intuition says that FDR can be traded for FNR and vice-versa.  
From this perspective, an interesting alternative risk to $\fr$ is, with $z_+=z\vee 0=\max(z,0)$,
\[ \mathfrak{R}_\alpha(\theta,\vphi) = (\FDR(\theta,\vphi) - \alpha)_+ +\FNR(\theta,\vphi).\]
Proposition~\ref{prop-riskt} shows that allowing for some false discoveries by targeting $\mathfrak{R}_\alpha\to 0$ rather than $\mathfrak{R}=\mathfrak{R}_0\to 0$ does not allow for weaker boundary conditions.
\begin{prop} \label{prop-riskt}
For $\alpha\in[0,1)$ and a fixed real $b$, consider the set of sparsity preserving procedures $\mathcal{S}_{B}(\Theta_b)$ as in Definition~\ref{def:spapres}. Under the conditions of Theorem~\ref{thm-notrade},
\[ \inf_{\vphi\in \mathcal{S}_{B}(\Theta_b)}\sup_{\te\in\Theta_b}\, \mathfrak{R}_\alpha(\theta,\vphi)= \overline\Phi(b)+o(1).\]
\end{prop}

Proposition~\ref{prop-riskt} is an immediate consequence of Theorems~\ref{thmgeneric} and~\ref{thm-notrade}  and of the bounds $\FNR\leq \mathfrak{R}_\alpha\leq \mathfrak{R}$.
From the multiple testing literature perspective, this result is relatively counter-intuitive as allowing a looser FDR control is generally considered to be beneficial for the power of the procedure. This is of course true, but as was shown in Figure~\ref{fig:mFDRvsFNR} and the discussion thereof, for large $n$ the FDR curve is much steeper than the FNR curve, so that the two cannot be traded efficiently (at least for the noise distributions considered here).

\subsection{~~Classification:  sharp adaptive minimaxity}
\label{sec:classexp} 
The classification (Hamming) loss in terms of classes $\{0\}$ and $\RR\setminus \{0\}$ is defined by
\begin{equation}\label{equLC}
 \lc(\te,\vphi) = \sum_{i=1}^n \left(\ind{\te_i=0}\ind{\vphi_i\neq 0}
+ \ind{\te_i\neq 0}\ind{\vphi_i = 0}\right). 
 \end{equation}
This risk is studied for sparse vectors in \cite{butucea18} (see also \cite{butucea21}), with a focus on different regimes of large signals.  A testing procedure $\vphi\in\{0,1\}^n$  is said to {\em achieve almost full recovery} \cite{butucea18} with respect to the Hamming loss over a subset $\Theta_{s_n}\subset \ell_0[s_n]$ if, as $n\to\infty$, 
\[ \sup_{\te\in\Theta_{s_n}} E_\te \lc(\te,\vphi)/s_n =o(1). \]
This notion is quite close (although not equivalent) to the notion of conservative testing as considered in Corollary~\ref{cor_ct}, see Section \ref{sec:almostfullrecovery}
  for a precise comparison. A complementary notion from \cite{butucea18} is that of {\em exact} recovery, which will be relevant in Section \ref{sec:fast} below.

To allow for a direct comparison with \cite{butucea18}, we introduce the class considered therein, defined, for $\Theta(a_b,s)$ as in \eqref{signalb}, by
\begin{equation}\label{thetaprime}
 \Theta_b'=\Theta'_b(s_n)= \bigcup_{0\le s\le s_n} \Theta(a_b,s).
\end{equation} 
Note that all our lower bound results imply the same bounds for this larger class $ \Theta_b'\supset \Theta_b=\Theta(a_b,s_n)$, and the following result gives a corresponding upper bound on this class. We restrict to Gaussian noise to fix ideas, but as before, similar results hold more generally. 

\begin{theorem} \label{thm-adapt-cl}
Consider the Gaussian location model of Example \ref{ex:GaussianSequence}. 
If one sets $a_b=\sqrt{2\log(n/s_n)}+b$ for an arbitrary real $b$ or a sequence $b=b_n\to \pm\infty$,  then for $\Theta_b'=\Theta_b'(s_n)$ as in \eqref{thetaprime}, the sharp asymptotic minimax risk for classification is 
\[ 
\inf_{\vphi} \sup_{\te\in\Theta_b'} E_\te \lc(\te,\vphi)/s_n = \overline\Phi(b)+o(1). \]

For $\vphi=\vphi^{\hat{\ell}}$ the empirical Bayes $\ell$-value procedure \eqref{ellvalp} or  $\vphi=\vphi^{BH}$ the BH procedure \eqref{defBH} at a level $\alpha=\alpha_n = o(1)$ with $-\log \alpha = o((\log n)^{1/2})$ (additionally assuming the polynomial sparsity \eqref{polspa}), the bound is achieved: for any real $b$, or for $b=b_n\to\pm\infty$,
\[\sup_{\te\in\Theta_b'} E_\te \lc(\te,\vphi)/s_n = \overline\Phi(b)+o(1).\]
\end{theorem}

Theorem~\ref{thm-adapt-cl} is proved in Section~\ref{sec:proof-of-thm-adapt-cl}. These results complement some of the results of Sections 4-5 in \cite{butucea18} in the almost sure recovery regime.  
Theorem 5.2 therein states that there exists an adaptive procedure that achieves almost full recovery if $b_n\geqa \log\log{n}$. Theorem~\ref{thm-adapt-cl} shows that the $\ell$-value procedure achieves it under the weakest possible condition $b_n\to +\infty$, not requiring a particular rate. Further, Theorem~\ref{thm-adapt-cl} 
 investigates the setting of a finite $b$. Over the class $\Theta_b'$, Theorem  4.2(ii) in \cite{butucea18} provides an in-expectation lower bound that is asymptotically similar to that from Theorem~\ref{thm-adapt-cl}, but only for the case of $b\ge 0$ (which essentially amounts to $W>0$  in the notation from \cite{butucea18}); Theorem~\ref{thm-adapt-cl} provides the sharp asymptotic constant for any real $b$, and asserts that it can be achieved by an adaptive (i.e. not dependent on $s_n$ or $b$) procedure. 
 
In-probability results complementing Theorem \ref{thm-adapt-cl}, covering also the case of general signals, are provided in Section \ref{sec:classification-in-prob}. 
Let us remark that in the above result, in contrast to Theorem~\ref{thm-adapt-r} (and Remark \ref{rem:lvaluenotphibar}), we are able to include the case $b=b_n\to-\infty$ for the $\ell$-value procedure without extra assumptions.

\section{~~Large signal regime and faster minimax rates} \label{sec:fast}

In this section, we work in the Subbotin location model presented in Example~\ref{example:SubbotinLocation} (sometimes restricted to the Gaussian case for simplicity). 
An anonymous referee suggested to investigate a `large signal' regime  for which nonzero signals have an amplitude larger than $M=M(r)=(\zeta r\log{n})^{1/\zeta}$ while the sparsity satisfies, say, $s_n\asymp n^{1-\be}$, for some $r>\beta$. 
In Section~\ref{sec:fastmr}, we identify the minimax risk in this large signal regime.
In Section~\ref{sec:fastmr_adapt}, we consider the possibility of adaptation to both the unknown sparsity and signal strength.

%
Let us first formally introduce an appropriate parameter space of large signals. For some $\be\in(0,1)$, $\zeta>1$, and $r>\be$, for fixed $0<a<b$, define
\begin{align}
  \Theta(r,\be) &= \bigcup_{s_n\in[an^{1-\be},bn^{1-\be}]}\{ \te \in\ell_0[s_n]:\ |S_\te|=s_n,\ \ |\te_i|\ge M(r)  \text{ for all } i\in S_\te \},  \label{classab}\\
   M & =M(r)=(\zeta r\log{n})^{1/\zeta}.                   \label{emer}
\end{align} 
 A main motivation for considering classes \eqref{classab} is to investigate the rate at which the maximum risk goes to $0$ in Corollary \ref{cor_ct} (i.e. Theorem \ref{thmgeneric} for $b\to+\infty$ fast). The recent work \cite{millerstep23} considers a related problem for the classification loss, but without looking at precise convergence rates. Let us mention that all our results below are also valid for the (normalised) classification loss.

\subsection{~~Fast minimax rate} \label{sec:fastmr}

The following result holds and resembles  Theorem~2 of \cite{rabinovich20} for Subbotin noise. There are important differences: first, the latter work studied so--called ``generalised Gaussian'' noises which do not contain the Gaussian or Subbotin noises considered here, and the obtained rates then differ by logarithmic factors; second, Theorem~\ref{thmRabinovic} below provides the minimax rate over {\em all} possible estimators, not only thresholding estimators; finally, the nonzero signals in class \eqref{classab} are not necessarily equal (if they were be, averaging procedures could be considered that would allow one to estimate signal strength more easily in some cases).

%

\begin{theorem} \label{thmRabinovic}
Consider the Subbotin location model (Example~\ref{example:SubbotinLocation}) for some $\zeta>1$ and the parameter set  $\Theta(r,\be)$ defined by \eqref{classab}--\eqref{emer}. 
Let $\kappa=\kappa(r,\beta,\zeta)$ be the unique  element of $(0,r/2^\zeta)$ such that $(r^{1/\zeta} -\kappa^{1/\zeta})^\zeta-\kappa= \beta$. 
 Then the thresholding procedure
$ \vphi^*_i=\II_{|X_i|\ge  t_n^*} $ based upon the threshold
\begin{equation}\label{optthresholdRab}
t_n^* =   (\zeta  \log n)^{1/\zeta} (r^{1/\zeta} - \kappa^{1/\zeta}) = (\zeta (\beta + \kappa) \log n)^{1/\zeta} 
\end{equation}
is asymptotically rate minimax: 
\[  \sup_{\te\in\Theta(r,\be) }  \mathfrak{R} (\theta,\vphi^*)\asymp \inf_\vphi \sup_{\te\in\Theta(r,\be) }  \mathfrak{R} (\theta,\vphi) \asymp n^{-\kappa}/(\log n)^{1-1/\zeta}. \]
The same result holds for classification upon replacing $\fr$ by $E_\te\lc/n^{1-\be}$.  
\end{theorem} 


Theorem~\ref{thmRabinovic} is proved in Section~\ref{proof:thmRabinovic}.
Let us first notice that the threshold $t_n^*$ in \eqref{optthresholdRab} is markedly larger than the threshold  $(\zeta \beta \log n)^{1/\zeta}$ optimal  in the boundary case (see Theorem~\ref{th:minimax-multilevel}). In particular, if $r$ is much larger than $\beta$, we have $\kappa\approx r/2^\zeta$ and $t_n^*\approx 0.5 (\zeta  r\log n)^{1/\zeta}$, and the minimax risk is converging to $0$ roughly at the rate $n^{-r/2^\zeta}$. 
Second, an important observation is that the FNR and FDR are of the same order in this regime (see the upper bound in the proof). In particular, the FNR part is not dominating for the class of signals \eqref{classab}, which is markedly different from the results we obtained in the regime where the signal is at the boundary (see Theorems~\ref{thm-notrade}, \ref{thm-notrade-Subbotin} and \ref{thm-notrade-pointwise}).


For the classification risk $E_\te \lc$, we note that the bound achieved using Theorem \ref{thmRabinovic} goes to $0$ if $\ka\ge 1-\be$, in which case one has exact recovery in the sense of \cite{butucea18}. We also note that using similar arguments, we can obtain minimax optimal rates also in probability for $\lc$; also, for large signals, one can show that the support of $\te$ is recovered with probability going to $1$ and study rates for the probability of making at least one error; see Theorem \ref{thm:inprlarge}. 

\subsection{~~Adaptation to fast minimax rates}\label{sec:fastmr_adapt}

We now study possible adaptation in the context of the minimax result of Theorem~\ref{thmRabinovic}. The class  $\Theta(r,\be)$ depends on two parameters and so adaptation can be considered with respect to both. We focus on the Gaussian case $\zeta=2$ for simplicity: this case already captures the main phenomena at stake (the first two cases of Theorem \ref{thmlargeadapt} below will in fact be obtained also for Subbotin noise). The `top--$s_n$ procedure' will refer to a procedure that selects exactly the $s_n$ largest observations in absolute value.  
The following theorem summarises our results on the possibility of adaptation. Precise statements can be found in Section~\ref{sec:prooffastrate}. 
 
 \begin{theorem}[Adaptation to large signals, summary in Gaussian case]\label{thmlargeadapt} 
 Consider the setting of Theorem \ref{thmRabinovic} with $\zeta=2$ (Gaussian noise).  The following points consider the adaptation problem over the class $\Theta(r,\be)$ for some $\be\in(0,1)$ and $\be<r$. The results hold for the $\fr$--risk as well as for the normalised classification risk $E_\te \lc(\te,\vphi)/n^{1-\be}$. 
 \begin{enumerate}[(i)] 
  \item Simultaneous adaptation to $(r,\be)$ is impossible over the full range of parameters.
 \item Known exact sparsity $s_n$, unknown $r$: the top--$s_n$ procedure provides adaptation to $r$.
 \item Unknown $\be$, known $r$: a plug-in procedure provides adaptation to $\be$.
 \item Simultaneous adaptation to $(r,\be)$ is possible in the regime $\ka(r,\be,2)<1-\be$.
 \end{enumerate}  
Moreover, for point (i) the loss due to adaptation is polynomial in $n$.  
\end{theorem} 

Theorem \ref{thmlargeadapt} can be interpreted as follows: in the case of large signals in the class \eqref{classab}, when $r$ becomes too large, the rate goes very quickly to zero and adaptation becomes impossible. The intuitive reason is that for very large signals even only one error in the recovery of the support can make the rate drop. In fact, the region in which adaptation is not possible is the one where the optimal (nonnormalised) classification risk is $o(1)$, which corresponds to the region of  exact recovery considered in \cite{butucea18}.  

The results also provide a simple non-artificial example of a setting where there is a polynomial--in--$n$ loss for adaptation (as opposed to more commonly observed logarithmic--in--$n$ losses in nonparametrics).


\begin{remark}[optimalities for top--$s_n$ procedure]\label{rem:topKsuboptimal}
The top--$s_n$ procedure, which is adaptive to $r$ here for large signals, is by contrast not sharp minimax (even though it is oracle) when less signal strength is present:  the no-trade-off result Theorem \ref{thm-notrade} actually enables one to derive its precise risk for the class $\Theta_b$, see Corollary \ref{cor-topsub}. 
\end{remark}

\section{~~Discussion} \label{sec:disc}
\if TT

{\em Overview of the results. }

This work derived new results for multiple testing from a minimax point of view, in particular deriving the sharp minimax constant for the sum risk  $\fr=\FDR+\FNR$.  
Allowing for a variety of possible noise distributions including standard Gaussian noise, we first considered the beta-min condition on signals, and next allowed for arbitrary sparse signals. This enables one to qualitatively compare difficulties of multiple testing problems, such as the ones depicted on Figure~\ref{fig:mt_pb}. For such general signals, a notable finding is that the overall testing difficulty can be expressed in terms of a necessary and sufficient condition on a certain average, $\Lambda_n$ in \eqref{eqn:LambdaN}, where the strength of a signal is formulated by comparison to the `oracle threshold' (e.g., $\sqrt{2\log(n/s_n)}$ for Gaussian noise).

As a special case, it follows from this work that the ``boundary'' for the testing problem corresponds to slightly weaker signals {(i.e., classes  $\Theta_b$ with fixed $b=\overline{\Phi}^{-1}(t)$ for a target risk $t$ in the Gaussian case)} compared to classification with almost sure support recovery, for which one needs to be well above the oracle threshold ($b=b_n\to +\infty$).

The reason is that, as is common in (multiple) testing, one allows for a certain percentage of error in the overall testing risk. An important message regarding this tolerance level is that asymptotically the two types of errors, FDR and FNR, are not symmetric in terms of their contribution to this level: for any optimal procedure, as long as it is sparsity preserving, the FNR dominates, at least for the noise distributions considered here. This main finding of the paper has consequences for (sub-)optimality of some popular procedures, and also intuitively explains why we are able to obtain similar results for multiple testing and classification losses, even at level of sharp constants. 

Although originally one main goal was to investigate the phase transition in terms of constants from ``easy'' to ``impossible'' multiple testing, the techniques developed are useful for the case of strong signals too. In this setting, very different results are derived: while adaptation to the sharp constant is possible for the worst-case risk, full adaptation to the rate of decrease of the multiple testing risk is impossible for large signals, despite the setting being ``easier''. Adaptation becomes possible when restricting to certain subregions of signal/sparsity or assuming one of the two is known. 


{\em Adaptive procedures: comparing BH and $\ell$-values.} We have shown that two popular  procedures reach the optimal (local minimax) bounds in an adaptive manner: the $\ell$-value procedure and a properly tuned BH procedure. Both have optimal behaviour under relatively similar conditions, but each has its own merits in certain settings. One important message is that, even if the overall risk  $\fr$ is allowed to be at least $\alpha$, taking the standard BH procedure with fixed $\alpha$ parameter does not lead to an optimal risk: taking $\alpha=\alpha_n$ to go to zero is really needed in order to achieve optimality (see Corollary~\ref{cor-subopt}).

The $\ell$-value procedure does not need a tuning parameter going to zero: it can be applied, for example, for $t=1/2$. One can interpret this as the fact that it is already on the correct ``scale'' for the $\fr$-risk (this is perhaps unsurprising, as this procedure relates to the Bayes classifier for the Hamming risk when the prior is correct), while BH is scaled for the FDR, i.e. it essentially prescribes the FDR value. Since the FDR has to be negligible for certain parameters for optimal procedures with respect to $\fr$ (for sparsity preserving ones, as directly follows from combining  Theorems~\ref{th:minimax-multilevel} and \ref{thm-notrade-Subbotin}), this  explains why one needs to take $\alpha=\alpha_n\to 0$ to achieve optimality for BH. 

On the other hand, since $\ell$-values are in principle not immediately designed for a (frequentist) FDR-control, proving that the $\ell$-value procedure controls the FDR can be non-trivial; in \cite{cr20} it was shown to be the case for arbitrary sparse signals, but required significant technical work (a result that is invoked here only to handle the somewhat degenerate case $b_n\to -\infty$ in Remark~\ref{rem:lvaluenotphibar}). 

From a more practical point of view, choosing between these two competitors depends on the pursued aim:  a user interested solely in the combined risk may use the $\ell$-value procedure because it is more intrinsic and  asymptotically optimal  for a fixed  value of the parameter $t$ (in the case of the BH-$\alpha$ procedure the level $\alpha$ parameter needs to be tuned appropriately), while if the FDR value should be also controlled or known, they could use the BH procedure with a level $\alpha$ satisfying the convergence requirements. Also note that we have been able to prove optimality of the BH procedure under broader noise assumptions than for the $\ell$-value procedure, though we conjecture the latter will work in all models considered here.

{\em Future directions.}  
This work paves the way for several future investigations. First, we expect many ideas relevant here for the Gaussian sequence model to transport to multiple testing for more complex models, such as high dimensional linear regression. Second, since the combined risk involves two parts with markedly different behaviors at the boundary, it could be valuable to find another notion of risk making a better balance between these two terms. Deriving such a  risk notion would have interesting consequences for building semi-supervised machine learning algorithms achieving an appropriate FDR/FNR tradeoff. 
Finally, the multiple signal framework introduced here is worth investigating for estimation-type risks, where the constant $\La_\infty=\lim_n \Lambda_n $ is expected to still play a key role.

\section{~~Proof of Theorems~\ref{thmgeneric} and \ref{th:minimax-multilevel}} \label{sec:proofs}

In this section, we give the proof of Theorems~\ref{thmgeneric} and \ref{th:minimax-multilevel}.
In fact we shall restrict our attention to proving Theorem~\ref{th:minimax-multilevel}, because Theorem~\ref{thmgeneric} is a direct consequence of it. Indeed, Lemma~\ref{lem:assumptionok} verifies the conditions of Theorem~\ref{th:minimax-multilevel} (first bullet point) in the setting of Theorem~\ref{thmgeneric} for $f_a(x)=\phi(x-a)$ and $a_n^*=\sqrt{2\log(n/s_n)}$; note that for $\ba=(a_n^*+b,\dots,a_n^*+b)$ we have $\Theta(\ba,s_n)=\Theta_b$ and $\Lambda_n(\ba)=\overline{\Phi}(b)$. [In the cases $b\in \RR$ fixed or $b=b_n\to+\infty$, we apply Theorem~\ref{th:minimax-multilevel} for $n$ large enough that $b>-(2\log(n/s_n))^{1/2}$; in the case $b=b_n\to -\infty$ we replace $b_n$ by a sequence $b_n'\to-\infty$ satisfying $b_n'>\max(-(2\log(n/s_n))^{1/2},b_n)$ before applying Theorem~\ref{th:minimax-multilevel}, obtaining in this way
$\inf_{\vphi}\sup_{\theta\in \Theta_{b_n'}} \frak{R}(\theta,\vphi)=\overline{\Phi}(b_n')+o(1)=1+o(1)$. Since $\Theta_{b_n'}\subset \Theta_{b_n}$ we deduce $\inf_{\vphi}\sup_{\theta\in \Theta_{b_n}}\frak{R}(\theta,\vphi) \geq 1+o(1)$; since the upper bound in this case is obtained using the trivial test $\vphi\equiv 0$ we deduce Theorem~\ref{thmgeneric} in all cases.] 


%

\subsection{~~Lower bound}\label{sec:nonasymptotic-lower-bounds}
Let us start by proving the lower bound. For this, we apply the general lower bound of Theorem~\ref{thm:bayesLB} in our particular model: for any $\rho>0$, any $s_n\geq 1$, any $\eta \in (0,1)$ and any prior $\pi$ on $\R^n$, with $P_\pi$ denoting the joint law of $(X,\theta)$ under $\pi$, 
	$$
	 \inf_\vphi \sup_{\te\in\Theta(\ba,s_n) }  \mathfrak{R} (\theta,\vphi) \geq   \left(\lambda \wedge \frac{\rho \lambda}{1+\rho \lambda}\right)  (1-  e^{- c\eta^2 M_\rho})- n (1\vee \rho) P_\pi(\norm{\theta}_0> s_n) - 2 P_\pi(\theta\notin \Theta(\ba,s_n)),
		$$
	for some universal constant $c>0$, where $\lambda= M_\rho (1-\eta) / s_n$ and 
	$$M_\rho=\sum_{i=1}^n P_\pi\left[ \theta_i\neq 0,\ell_i(X)> \rho/(1+\rho)\right] ,$$ for $ \ell_i(X)=P_\pi(\theta_i=0\:|\: X)$, $1\leq i\leq n$.
We will apply this with $\rho=\rho_n$ and $\eta=\eta_n$ particular positive sequences converging slowly to infinity and $0$ respectively to be specified later on.

Let us consider a specific prior $\pi$ as a product prior over $s_n$ blocks of consecutive coordinates $Q_1=\{1,2,\ldots, q\}, Q_2=\{q+1,\ldots,2q\},\ldots, Q_{s_n}=\{(s_n-1)q+1,\ldots,n'\}$, where $q= \floor{n/s_n}$ and $n'=qs_n$. We write $Q_\infty$ for the (possibly empty) set $\{n'+1,\dots,n\}$. 
Over each block $Q_j$, $1\leq j\leq s_n$, one takes the following prior:  
first draw an integer $I_j$ from the  uniform distribution $\cU(Q_j)$ over the block $Q_j$ and next for each $i\in Q_j$ set $\theta_{i}=a_j$ if $i=I_j$ and $\theta_{i}=0$ otherwise.  
For $i\in Q_\infty$, set $\theta_{i}=0$. With this prior, we clearly have 
$P_\pi(\norm{\theta}_0> s_n)=0$ and $P_\pi(\theta\notin \Theta(\ba,s_n))=0$. Moreover,	
for all $1\leq j\leq s_n$ and $i\in Q_j$,
\begin{align*}
	\ell_i(X)&=P_\pi(i\neq I_j\:|\: X)=1- w_i^j(X) \\
	w_i^j(X)&= \frac{f_{a_j}(X_i)/f_0(X_i)}{\sum_{k\in Q_j} f_{a_j}(X_k)/f_0(X_k)}= \frac{h(X_i,a_j)}{\sum_{k\in Q_j} h(X_k,a_j)},\\ 
	h(x,a)&= f_a(x)/f_0(x).
\end{align*}

In addition,
\begin{align}
	M_\rho&=  \sum_{i=1}^n P_\pi[ \theta_i\neq 0\,,\,\ell_i(X)> \rho/(1+\rho)]=\sum_{j=1}^{s_n} \sum_{i\in Q_j} P_\pi[ i=I_j\,,\,\ell_i(X)> \rho/(1+\rho)]\nonumber\\
	&=\sum_{j=1}^{s_n} \sum_{i\in Q_j} P_\pi\sqbrackets[\Big]{ i=I_j\,,\, (\rho+1) h(X_i,a_j) < \sum_{k\in Q_j} h(X_k,a_j) }\nonumber\\
	&\geq \sum_{j=1}^{s_n} \sum_{i\in Q_j} P_\pi\sqbrackets[\big]{ i=I_j\,,\, \#\{ k\in Q_j\backslash\{i\}\::\:  h(X_k,a_j)> h(X_i,a_j)  \}\geq \rho}\nonumber\\
	&=  \sum_{j=1}^{s_n}  P_{X_i\sim f_{a_j}}\left[ \#\{ k\in Q_j\backslash\{i\}\::\:  h(\varepsilon_k,a_j)> h(X_i,a_j)  \}\geq \rho\right],\label{equMrho}
\end{align}
where  the $\varepsilon_k$'s are i.i.d. $\sim f_0$, and in the last line $i=i_j\in Q_j$ is chosen arbitrarily by symmetry. 

Consider the first case, when Assumption~\ref{ass:generalnoise}\ref{ass:location} 
holds. 
We see that $\varepsilon_k> X_i$ implies $h(\varepsilon_k,a_j)> h(X_i,a_j)$, so that the last display can be further lower bounded by
\begin{equation*}
	\sum_{j=1}^{s_n}  P_{X_i\sim f_{a_j}}\left[ \#\{ k\in Q_j\backslash\{i\}\::\:  \varepsilon_k> X_i \}\geq \rho\right]. 
\end{equation*}
By Lemma~\ref{lem:ControlOfpn},  provided $\rho$ satisfies the condition of the lemma, we have
\[P_{X_i\sim f_{a_j}}\left[ \#\{ k\in Q_j\backslash\{i\}\::\:  \varepsilon_k> X_i \}\geq  \rho\right]=F_{a_j}(a_n^*)+o(1),\] so that continuing the inequalities, 
\begin{align}
	M_\rho/s_n
	&\geq s_n^{-1} \sum_{j =1 }^{s_n} F_{a_j} \left(a_n^* \right) +o(1)= \Lambda_n(\ba) +o(1) ,\label{equMrhoboundedA}
\end{align}
where $\Lambda_n(\ba)$ is defined by \eqref{eqn:LambdaN}.
This gives the lower bound 
\begin{align*}
\inf_\vphi \sup_{\theta\in \Theta(\ba,s_n)} \fr(\te,\vphi) \geq &  \left([\Lambda_n({\bf a})+o(1)](1-\eta) \wedge \frac{\rho [\Lambda_n({\bf a})+o(1)](1-\eta)}{1+\rho [\Lambda_n({\bf a})+o(1)](1-\eta)}\right) \\&\times (1-  e^{- c\eta^2 [\Lambda_n({\bf a})+o(1)](1-\eta)s_n}) .
\end{align*}
Now using that for all $x\in [0,1]$, $A,y>0$, we have 
$$[x\wedge (y x/(1+y x))](1-e^{-A x})\geq x+0\wedge(1-1/(1+yx)-x) -xe^{-Ax}\geq x- 1/y - 1/(Ae),$$
we deduce
\begin{align*}
\inf_\vphi \sup_{\theta\in \Theta(\ba,s_n)} \fr(\te,\vphi) \geq & [\Lambda_n({\bf a})+o(1)](1-\eta) - \rho^{-1}-1/(c\eta^2 s_n e).
\end{align*}
Let us set $\rho=\floor{(n/(3s_n))\overline{F}_0(a_n^*-\delta_n)}\to \infty$, an admissible choice for applying Lemma~\ref{lem:ControlOfpn}. Further setting $\eta=s_n^{-1/4}$,  one gets
$$
\varliminf_n \inf_\vphi \sup_{\theta\in \Theta(\ba,s_n)} \fr(\te,\vphi)\geq  \Lambda_n({\bf a})+o(1).
$$
This proves the lower bound.

When instead Assumption~\ref{ass:generalnoise}\ref{ass:scale} holds, we note that $h(x,a)=h(-x,a)$ is increasing in $\abs{x}$ for $a\neq 0$, hence returning to \eqref{equMrho} we see
\begin{align*}
	M_\rho&\geq \sum_{j=1}^{s_n}  P_{X_i\sim f_{a_j}}\sqbrackets[\big]{ \#\braces{ k\in Q_j\backslash\{i\}\::\:  \abs{\varepsilon_k}> \abs{X_i} }\geq  \rho}.
\end{align*}
%

Lemma~\ref{lem:ControlOfpn} tells us that in this case we have 
\[P_{X_i\sim f_{a_j}}\left[ \#\{ k\in Q_j\backslash\{i\}\::\:  \abs{\varepsilon_k}> \abs{X_i} \}\geq \rho\right]=1-2\overline{F}_{a_j}(a_n^*)+o(1)= 2F_{a_j}(a_n^*)-1 +o(1),\]
yielding 
	\begin{equation} \label{equMrhoboundedB} M_\rho/s_n \geq 2\Lambda_n(\ba) - 1 + o(1),
\end{equation}
and we finish the proof as under the other assumption.

\subsection{~~Upper bound}\label{sec:ub}

If a subsequence $n_j$ is such that $\Lambda_{n_j}\to 1$, then the trivial test $\vphi=0$ has $\mathfrak{R}(\theta,\vphi) = 1 = \Lambda_{n_j}+o(1)=2\Lambda_{n_j}-1+o(1)$. We may therefore limit our attention to a subsequence on which $\Lambda_{n_j}$ is bounded away from 1. Relabel $n_j$ as $n$.

Define 
\[ \vphi_i(X) = \II\braces{\abs{X_i}>a_n^*}.\]
Write $V$ and $S$ for the number of false discoveries and true discoveries, respectively, made by $\vphi$.

By Markov's inequality, for any $c_n>0$ the number of false discoveries satisfies
\[ P_\te (V>c_n s_n) \leq c_n^{-1}s_n^{-1} E[V] = c_n^{-1} (2(n-s_n)/s_n)\overline{F}_0(a_n^*).\] This latter expression tends to zero by Assumption~\ref{ass:generalnoise} if $c_n$ tends to zero slowly enough, yielding that $V=o_P(s_n)$.
Similarly, using that $\Var(S)\leq  E[S]\leq s_n$ and applying Chebyshev's inequality, the number of true discoveries satisfies
\[ P_\te[\abs{S-E_\te S}\ge s_n^{3/4} ]\le s_n^{-1/2}. \] Let $\mathcal{A}$ be an event of probability tending to one on which $V$ and $S$ are suitably bounded.

Under Assumption~\ref{ass:generalnoise}\ref{ass:location}, for $a\neq 0$ we have by symmetry 
\[ P_{\theta_i=a}(\abs{X_i}>a_n^*) = \overline{F}_a(a_n^*)+F_a(-a_n^*)=\overline{F}_a(a_n^*)+\overline{F}_{-a}(a_n^*).\] Noting that by the assumed monotonicity $\overline{F}_{-\abs{a}}(a_n^*)\leq \overline{F}_0(a_n^*)\to 0$ because $a_n^*\to \infty$, we deduce that 
\[ P_{\theta_i=a}(\abs{X_i}>a_n^*) = \overline{F}_{\abs{a}}(a_n^*)+o(1)\] where the errors $o(1)$ tend to zero uniformly in $a$. We thus calculate
\begin{align*}
	E_\theta [S] &= \sum_{i\in S_\theta} P_{\theta_i}(\abs{X_i}>a_n^*) = \sum_{i\in S_\theta} \left( \overline{F}_{\abs{\theta_i}}(a_n^*) + o(1) \right) \\ &\ge  \sum_{j\leq s_n} \left(\overline{F}_{a_j}(a_n^*) + o(1)\right) = s_n [1-\Lambda_n(\ba) +o(1)],\
\end{align*}
where we have used that $S_\theta$ can be enumerated as $i_1,\dots,i_{s_n}$ with $\abs{\theta_{i_j}}\geq a_j>0$.

The combined risk of $\vphi$ is then 
\begin{align*} \mathfrak{R}&(\theta,\vphi) = E_\theta \frac{V}{V + S} +E_\theta \frac{s_n-S}{s_n} \\ &\leq  P(\mathcal{A}^c) + \frac{ o(s_n)}{o(s_n)+ s_n(1-\Lambda_n(\ba)+o(1))-s_n^{3/4}} + \frac{s_n(\Lambda_n(\ba)+o(1))}{s_n} = \Lambda_n(\ba) + o(1),
\end{align*} 
yielding the desired upper bound. 

Under Assumption~\ref{ass:generalnoise}\ref{ass:scale}, the same calculations hold, except that  now $P_{\theta_i}(\abs{X_i}>a_n^*)=2\overline{F}_{\theta_i}(a_n^*)$, hence (again using the assumed monotonicity) $s_n^{-1}\sum_{i\in S_\theta} P_{\theta_i}(\abs{X_i}>a_n^*)\geq s_n^{-1}\sum_{j\leq s_n} 2\overline{F}_{a_j}(a_n^*) = 1-(2\Lambda_n(\ba)-1+o(1))$, which gives the desired upper bound in this case as well.

\section*{Funding} 
ER has been supported by ANR-16-CE40-0019 (SansSouci), ANR-21-CE23-0035 (ASCAI) and  the GDR ISIS through the "projets exploratoires" program (project TASTY). ER and IC have been supported by ANR-17-CE40-0001 (BASICS). 
KA is supported by the EPSRC Programme Grant on the Mathematics of Deep Learning under the project EP/V026259/1. Initial work on this project was completed while he was at Universit{\'e} Paris-Saclay, supported by a public grant as part of the Investissement d'avenir project, reference ANR-11-LABX-0056-LMH, LabEx LMH.

\section*{Acknowledgements} 
The authors are grateful to Aad van der Vaart for suggesting  investigating the case of multiple signals, which led us to the results in Section~\ref{sec:genresults}.   We also thank two anonymous referees and an associate editor for suggesting investigating more general noise models and other comments, which led to substantial generalisations.

\bibliographystyle{imsart-number} 
\bibliography{bibliography}

\newpage
	This supplementary material includes the remaining proofs for the main paper, and some further results and discussion. References to the main paper are included without prefixes, while references within this supplement have the prefix S-.

\setcounter{lemma}{0}
\setcounter{equation}{0}
\setcounter{section}{0}
\setcounter{corollary}{0}
\setcounter{theorem}{0}
\setcounter{definition}{0}
\renewcommand{\thelemma}{S-\arabic{lemma}}
\renewcommand{\theequation}{S-\arabic{equation}}
\renewcommand{\thesection}{S-\arabic{section}}
\renewcommand{\thecorollary}{S-\arabic{corollary}}
\renewcommand{\thetheorem}{S-\arabic{theorem}}
\renewcommand{\thedefinition}{S-\arabic{definition}}
\renewcommand{\theremark}{S-\arabic{remark}}
\renewcommand{\theexample}{S-\arabic{example}}

\section{~~Interpretation and verification of Assumption~\ref{ass:generalnoise}}\label{sec:VerificationOfAssumption} 


\subsection{~~Verification of Assumption~\ref{ass:generalnoise} in Subbotin case}

Assumption~\ref{ass:generalnoise} relies on a choice of a pair $(a_n^*,\delta_n)$. The next remark provides an example of such a choice, which also allows us to approximate $\Lambda_n(\ba)$. 

\begin{remark}\label{rem:bn*}
	If \eqref{eqn:modelassumption2} and \eqref{eqn:modelassumption3} hold for some $a_n^*\to \infty$ and $\delta_n\to 0$, then for $b_n^*=\overline{F}_0^{-1}(s_n/n)$ we necessarily have
	\begin{align}
		(n/s_n) \overline{F}_0(b_n^*-\delta_n)&\to \infty, \\
		(n/s_n) \overline{F}_0(b_n^*+\delta_n)&\to 0,
	\end{align}
	since $\overline{F}_0$ is continuous and monotone, and since by the intermediate value theorem we must have $a_n^*-\delta_n<b_n^*<a_n^*$ for $n$ large enough.
	Moreover, since the $F_a$ are uniformly Lipschitz,
	\begin{equation}\label{equ:lambdanrem}
		\frac{1}{s_n} \sum_{j=1}^{s_n} F_{a_j}(b_n^*) = \Lambda_n(\ba) + o(1).\end{equation}
	We may therefore always choose $a_n^*=b_n^*+o(1)$, and use \eqref{equ:lambdanrem} to approximate $\Lambda_n$. 
\end{remark}
	Similar reasoning indicates that we may replace any valid $a_n^*$ by $a_n^*+\kappa_n$ if $\kappa_n=o(1)$, lending some flexibility which aids in condition \eqref{equcondalpha}.
In the following result we locally introduce the notation $\azeta=(\zeta\log(n/s_n))^{1/\zeta}$ since choosing $a_n^*$ slightly larger than $\azeta$ is helpful for example in Theorem~\ref{thm:BHpointwise}.

\begin{lemma}\label{lem:assumptionok}
	Assumption~\ref{ass:generalnoise}\ref{ass:location} holds in the Subbotin location model $X_i=a +\varepsilon_i$, $1\leq i\leq n$, and Assumption~\ref{ass:generalnoise}\ref{ass:scale} holds in the Subbotin scale model $X_i=(1+|a|)^{1/2} \varepsilon_i$, $1\leq i\leq n$, 
	for a noise $\varepsilon_i$ i.i.d. distributed as the Subbotin density $\phi_\zeta$ defined by \eqref{eqn:def:Subbotin-density}, with $\zeta>1$. 
	In both of these cases, Assumption~\ref{ass:generalnoise} holds for the pair $(a_n^*,\delta_n)$ for  $a_n^*= \azeta :=(\zeta \log(n/s_n))^{1/\zeta}$ and $\delta_n = (\log(n/s_n))^{-\upsilon}$ for any $\upsilon\in (0,1-1/\zeta)$ but also for all the pairs $(\azeta+\kappa_n,\delta_n+\kappa_n)
	$  with any positive sequence $\kappa_n\to 0$. In addition, 
	we have
\begin{align}
	\frac{n}{s_n}\overline{F}_0(\azeta-\delta_n) &\gtrsim (\azeta)^{-(\zeta-1)} \exp\brackets[\big]{c (\log (n/s_n))^{1-1/\zeta -\upsilon}}\to \infty\label{condF0anstarmoinsdelta}\\
	\frac{n}{s_n}\overline{F}_0(\azeta+\kappa_n) &\lesssim (\azeta)^{-(\zeta-1)} \exp(-\kappa_n (\azeta)^{\zeta-1})\to 0.\label{condF0anstar}
\end{align}
\end{lemma}

\begin{remark}\label{rem:polysparse}
In the Subbotin models, some of our  conditions can be made more explicit using  Lemma~\ref{lem:assumptionok} (recall that $\zeta=2$ is the Gaussian case), defining  $\azeta$ as in the lemma:
\begin{itemize} 
\item Theorem~\ref{thm-notrade-Subbotin}, condition on $B=(B_n)_n$:  the inequality $B_n^2\leq (n/ s_n)\overline{F}_0(a_n^*-\delta_n)/3$ holds if $B_n^2 \leq \exp((\log (n/s_n))^\upsilon)$ for some $\upsilon\in (0,1-1/\zeta)$, by applying \eqref{condF0anstarmoinsdelta} and choosing the pair $(a_n^*,\delta_n)=(\azeta,(\log(n/s_n))^{-\upsilon'})$ for $\upsilon'<1-1/\zeta-\upsilon$. Under polynomial sparsity \eqref{polspa}, the condition $B_n^2 \leq \exp((\log n)^\upsilon)$ suffices.
\item 
Theorem ~\ref{thm:BHpointwise}, condition \eqref{equcondalpha} for a suitable choice of $a_n^*,\delta_n$: 
%
Recall that condition~\eqref{equcondalpha} is satisfied if for $n$ large $\Lambda_n$ is bounded away from 1 and $(n/s_n)\overline{F}_0(a_n^*-\delta_n)\geq 1$, and if both $\alpha_n$ and $(n/s_n)\overline{F}_0(a_n^*)/\alpha_n$ tend to zero. The latter can be achieved with any $\alpha_n\to 0$ satisfying $\log(1/\alpha_n) = o((\log(n/s_n))^{1-1/\zeta})$: to see this, apply \eqref{condF0anstar} for the pair 
$(\azeta+\kappa_n,\delta_n+\kappa_n)=((\zeta \log(n/s_n))^{1/\zeta} + \kappa_n,\delta_n+\kappa_n)$ with $\kappa_n\asymp \log(1/\alpha_n)/\log(n/s_n)^{1-1/\zeta}$. Under polynomial sparsity \eqref{polspa}, the condition reduces to $\log(1/\alpha_n)=o( (\log n)^{1-1/\zeta})$, so that no further knowledge of $s_n$ is required to define the BH procedure. Other regimes, such as $s_n\asymp n/(\log n)^d$, may also be permitted for suitably chosen $\alpha_n$, slightly relaxing the polynomial sparsity assumption. 
\item Theorem ~\ref{thm:BHpointwise}, conclusion:
Combining the previous point with the conclusion \eqref{eqn:BHrisk} of Theorem~\ref{thm:BHpointwise} (and recalling that the $\overline{F}_a$ are uniformly Lipschitz) leads to the bound
\begin{align*}\label{eqn:BHriskcvrate} 
&\fr(\te,\vphi) - s_n^{-1}\sum_{i\in S_\theta} F_{\abs{\theta_i}}(\azeta)  \\ \lesssim& s_n^{-1}\sum_{i\in S_\theta} F_{\abs{\theta_i}}(\azeta+\kappa_n) - s_n^{-1}\sum_{i\in S_\theta} F_{\abs{\theta_i}}(\azeta)  + \alpha_n + \exp\brackets[\Big]{-\frac{(1-\Lambda_n(\theta))^2}{32}s_n}\\
&\lesssim 
\frac{\log(1/\alpha_n)}{\log(n/s_n)^{1-1/\zeta}} + \alpha_n + \exp\brackets[\Big]{-\frac{(1-\Lambda_n(\theta))^2}{32}s_n},
\end{align*}  
with a (close to) optimal choice of $\alpha_n=(\log(n/s_n))^{-(1-1/\zeta)}$, giving a convergence rate to $s_n^{-1}\sum_{i\in S_\theta} F_{\abs{\theta_i}}(\azeta)$ of order at most $\frac{\log\log (n/s_n)}{\log(n/s_n)^{1-1/\zeta}} + e^{-cs_n}$ for some constant $c>0$.
\end{itemize}

\end{remark}

\begin{proof} In view of  
Remark~\ref{rem:ass1}, it suffices to verify \eqref{eqn:modelassumption2}, \eqref{eqn:modelassumption3} and either \eqref{eqn:monotonicity-location} or \eqref{eqn:monotonicity-scale}. 
	Let us begin by verifying the conditions on the null model, which for both the location and scale Subbotin  models has $f_0,F_0$ and $\ol{F}_0$ equal to $\phi_\zeta,\Phi_\zeta$ and $\ol{\Phi}_\zeta$ respectively. By Lemma~\ref{lem:sub}, 
	\[ \overline{\Phi}_\zeta(x)\asymp \frac{{\phi}_\zeta(x)}{x^{\zeta-1}}, \quad \text{for}\quad x\geq 1.\] 
	Noting that ${\phi}_\zeta(\azeta)=L_\zeta^{-1} s_n/n$
	, we see that $(n/s_n)\overline{\Phi}_\zeta(\azeta)\asymp (\azeta)^{-(\zeta-1)}\to 0$. 
	It remains to lower bound $\overline{\Phi}_\zeta(\azeta-\delta_n)$. For any $\delta>0$, we have
	\[ \frac{n}{s_n}\overline{\Phi}_\zeta(\azeta-\delta) \asymp \frac{{\phi}_\zeta(\azeta-\delta)}{(\azeta-\delta)^{\zeta-1}{\phi}_\zeta(\azeta)},\]
	provided that $\azeta-\delta\geq 1$. 
	Lemma~\ref{lem:dConvex} tells us that for $\delta\geq 0$ we have (uniformly in $\delta<1$)
	\[ {\phi}_\zeta(\azeta-\delta)/{\phi}_\zeta(\azeta) \geq \exp\brackets[\big]{\delta \abs{\azeta-\delta}^{\zeta-1} \sign(\azeta-\delta)}.\]
	Thus (uniformly in $\delta<1$),
	\[ \frac{n}{s_n}\overline{\Phi}_\zeta(\azeta-\delta) \gtrsim (\azeta-\delta)^{-(\zeta-1)} \exp(\delta \abs{\azeta-\delta}^{\zeta-1}) \geq (\azeta)^{-(\zeta-1)}\exp(\delta (\azeta/2)^{\zeta-1}),\] the inequality holding since $\azeta/2 >\delta$ for $n$ large because $\azeta \to \infty$. If a sequence $\delta=\delta_n\in(0,1)$ satisfies $\delta_n \geq (\log(n/s_n))^{-\upsilon}$ for some $\upsilon\in(0,1-1/\zeta)$ then the last expression is lower bounded by, for a constant $c=c(\zeta)$, 
	\[ (\azeta)^{-(\zeta-1)} \exp\brackets[\big]{c (\log n/s_n)^{1-1/\zeta -\upsilon}}\to \infty.\]
This shows \eqref{condF0anstarmoinsdelta} and \eqref{eqn:modelassumption2}, \eqref{eqn:modelassumption3} for the pair $(\azeta,\delta_n)$. For the pairs $(\azeta+\kappa_n,\delta_n+\kappa_n)$ with $\kappa_n\to 0$, we observe that, again using Lemmas~\ref{lem:sub} and \ref{lem:dConvex}, 
	\[ \frac{n}{s_n} \overline{\Phi}_\zeta(\azeta+\kappa_n) \leq L_\zeta^{-1} (\azeta+\kappa_n)^{-(\zeta-1)} \exp(-\kappa_n \abs{\azeta}^{\zeta-1})\lesssim (\azeta)^{-(\zeta-1)} \exp(-\kappa_n (\azeta)^{\zeta-1}),\]
	which leads to \eqref{condF0anstar} (and \eqref{eqn:modelassumption3} with $a_n^* = \azeta +\kappa_n$).
	
	Next we move to verify the monotonicity conditions, first in the Subbotin location model $f_a(x)=\phi_\zeta(x-a)$, $x,a\in \R$. 
	To verify \eqref{eqn:monotonicity-location}, that is, that $f_{a}(x)/f_0(x)=\exp( \zeta^{-1} [\abs{x}^\zeta - \abs{x-a}^\zeta])$ is increasing in $x$ for $a>0$, it suffices to differentiate the map $x\in \RR\mapsto \abs{x}^{\zeta} - \abs{x-a}^{\zeta}$ on the regions $x\leq 0$, $0\leq x\leq a$ and $x\geq a$. 
	%
	
	In the Subbotin scale model $f_a(x)=\phi_\zeta(x/(1+\abs{a})^{1/2}) /(1+\abs{a})^{1/2}$, $x,a\in\RR$. 
	For \eqref{eqn:monotonicity-scale}, we have $f_{a}(x)/f_0(x)=
	\exp( \zeta^{-1} \abs{x}^\zeta (1-(1+|a|)^{-\zeta/2}))/(1+|a|)^{1/2}
	$ which is increasing in $x\geq 0$. 
%
\end{proof}


\begin{lemma} \label{lem:assumption-implies-symmetry}
	Under Assumption~\ref{ass:generalnoise}\ref{ass:location}, $\overline{F}_a(x)=F_{-a}(-x)$ for all $x,a\in\RR$. Under Assumption~\ref{ass:generalnoise}\ref{ass:scale}, $\overline{F}_a(x)=F_a(-x)$ for all $x,a\in\RR$.
\end{lemma}
\begin{proof}
	Immediate from the conditions on $f_a$ by substituting in the integrals defining $F_a$.
\end{proof}

\begin{lemma}\label{lem:monotoneLR}
Under Assumption~\ref{ass:generalnoise}, 
for all $a \in \R$ (Assumption~\ref{ass:generalnoise}\ref{ass:location}) or  $a >0$ (Assumption~\ref{ass:generalnoise}\ref{ass:scale}), we have that
$t\in [0,\infty)\mapsto \ol{F}_a(t)/\ol{F}_0(t)$ is continuous increasing. 
\end{lemma}

\begin{proof}
First observe that by \eqref{eqn:monotonicity-location} or \eqref{eqn:monotonicity-scale}, we have $\ol{F}_a(t)/\ol{F}_0(t)> f_a(t)/f_0(t)$ for all $t\geq 0$. Then the proof follows from a simple derivative computation.
\end{proof}

\subsection{~~Interpretation of Assumption~\ref{ass:generalnoise}}

The conditions \eqref{eqn:modelassumption2} and \eqref{eqn:modelassumption3} of Assumption~\ref{ass:generalnoise} can be thought of as light tail conditions. The following shows that Laplace tails are too heavy. In view of the proof, notice that the conditions roughly reduce to requiring the conditional probability
$P_{X\sim f_0}(\abs{X}>a_n^* \mid \abs{X}>a_n^*-\delta_n)$ to tend to zero for some $a_n^*\to\infty$ and for $\delta_n$ tending to zero slowly enough.
\begin{lemma}\label{lem:assumption-fails-for-laplace}
	Consider the Laplace tail function $\Phi_\zeta$, $\zeta=1$, defined as in Example~\ref{example:SubbotinLocation}. There do not exist positive numbers $a_n^*\to\infty$ and $\delta_n\to 0$ such that $(n/s_n)\overline{\Phi}_1(a_n^*)\to 0,~(n/s_n)\overline{\Phi}_1(a_n^*-\delta_n)\to \infty$.
\end{lemma}
\begin{proof}
	For any $a_n^*$ and $\delta_n$ with $a_n^*-\delta_n\ge 0$ we have by memoryless of the exponential distribution (or direct computation) 
	\[ \overline{\Phi}_1(a_n^*) 
	=\overline{\Phi}_1(a_n^*-\delta_n)e^{-\delta_n}.\] 
	The result follows since $e^{-\delta_n}\to 1$ if $\delta_n\to 0$. 
\end{proof}

Not only does the assumption not hold for Laplace noise, but also the conclusions are not true. In particular, it is relatively straightforward to show, using similar memorylessness arguments, that the optimal thresholding procedure will have $\FDR$ and $\FNR$ of the same order, in contrast to the conclusion of Theorem~\ref{thm-notrade-Subbotin} which said that the $\FDR$ always contributes negligibly to the combined risk at the boundary.
\\

Finally, note that a related assumption was considered in \cite{rabinovichpreprint}, as we discuss in Section~\ref{sec:rab}.

\section{~~Sparsity preserving procedures}\label{sec:sparsitypreserving}
 
Here we show that a large class of procedures are sparsity preserving in the sense of Definition~\ref{def:spapres}, namely: the oracle thresholding procedure, BH procedures with either fixed or vanishing level, and empirical Bayes $\ell$-value procedures with fixed level.

To start with, let us check the sparsity preserving property for the oracle thresholding procedure $\vphi_i=\II\braces{\abs{X_i}>a_n^*}$. Its number of true positives is at most $s_n$, and if $V$ denotes its number of its false positives, we have $E_\te V\le (n-s_n)2\ol{F}_0(a_n^*)$ so Markov's inequality combined with condition \eqref{eqn:modelassumption3} give that it is sparsity preserving up to a constant multiplicative factor $A$ (e.g. $A=1+\veps$, for arbitrary $\veps>0$).
 
The next lemma is useful for procedures for which a control of the FDR is already known. 
\begin{lemma}\label{lemmaBHpreservsparsity2}
	For any $\Theta\subset \ell_0[s_n]$ and any multiple testing procedure $\vphi$, we have for any $u>1$ and $\theta\in \Theta$,
	\begin{equation}\label{equBHpreservsparsity2}
		P_{\te}\left[\sum_{i=1}^n \vphi_i(X) > u s_n \right]\leq P_{\te}\left[ \FDP(\theta,\vphi) > 1-u^{-1},\sum_{i=1}^n \vphi_i(X) \geq s_n \right] \leq \frac{u}{u-1} \FDR(\theta,\vphi).
	\end{equation}
	In particular, we have 
	\begin{itemize}
		\item[(i)]
		for any constant $c>1$ and any sequence $A=(A_n)_n$ for which $A_n\geq c$, any sequence of procedures $\vphi$ with vanishing FDR in the sense of
		\[
		\sup_{\theta\in \Theta}\FDR(\te,\vphi)=o(1)
		\]
		is sparsity preserving up to the multiplicative factor $A=(A_n)_n$ over $\Theta$, that is, $\vphi\in \mathcal{S}_{A}(\Theta)$ with Definition~\ref{def:spapres}.
		\item[(ii)]
		for any sequence $A=(A_n)_n$ for which $A_n\to \infty$, any sequence of procedures $\vphi$ with an FDP bounded uniformly in probability away from $1$ when at least $s_n$ rejections are made, in the sense that for some fixed $t<1$,
		\[
		\sup_{\theta\in \Theta} P_\theta\left(\FDP(\theta,\vphi) > t, \sum_{i=1}^n \vphi_i(X) \geq s_n\right)=o(1)
		\]
		is sparsity preserving up to the multiplicative factor $A=(A_n)_n$  over $\Theta$, that is, $\vphi\in \mathcal{S}_{A}(\Theta)$ with Definition~\ref{def:spapres}.
	\end{itemize}
\end{lemma}

An easy consequence of Lemma~\ref{lemmaBHpreservsparsity2}(i) is that the BH procedure taken at a level $\alpha_n\to 0$ (see Section~\ref{secBH} for a formal definition) satisfies the sparsity preserving condition on $\Theta=\ell_0[s_n]$ for any sequence $A=(A_n)_n$ with $A_n\geq c>1$ for all $n$ (recall that the FDR of this procedure is at most $\alpha_n$, see \eqref{FDRBH}). Using Theorem~1 of \cite{cr20}, the $\ell$-value procedure also has a vanishing FDR (under polynomial sparsity), so that Lemma~\ref{lemmaBHpreservsparsity2}(i) also ensures that the $\ell$-value procedure is sparsity preserving on $\Theta=\ell_0[s_n]$. We directly obtain in Section~\ref{sec:proof:thm-adapt-cl} sparsity-preservingness of the $\ell$-value procedure on $\Theta=\ell_0[s_n]$ without requiring polynomial sparsity,  see  Remark~\ref{rem:l-val-sparsity-preserving}. Finally Corollary~\ref{cor-BHSP} below  uses Lemma~\ref{lemmaBHpreservsparsity2}(ii) to prove that the BH procedure at a fixed level $\alpha<1$ is sparsity preserving.

\begin{proof}
	Letting $V=\sum_{i=1}^n\ind{\theta_i = 0} \vphi_i(X)$, we have for all $u>1$,
	\begin{align*}
		\left\{\sum_{i=1}^n \vphi_i(X) >  u s_n \right\}&\subset\left\{V > (u-1) s_n \right\}\cap  \left\{\sum_{i=1}^n \vphi_i(X) \geq s_n \right\}\\
		&\subset \left\{\frac{V}{V+s_n} > \frac{(u-1) s_n}{(u-1) s_n+s_n} \right\}\cap  \left\{\sum_{i=1}^n \vphi_i(X) \geq s_n \right\}\\
		&\subset \left\{ \FDP(\theta,\vphi) > 1-u^{-1} \right\}\cap  \left\{\sum_{i=1}^n \vphi_i(X) \geq s_n \right\} , 
	\end{align*}
	because the function $x\in [0,\infty)\mapsto \frac{x}{x+s_n}$ is non-decreasing. We conclude the proof of \eqref{equBHpreservsparsity2} by taking probabilities and using Markov's inequality.
	Parts $(i)$ and $(ii)$ are immediate consequences.
\end{proof}

\begin{corollary}\label{cor-BHSP}
	The BH procedure taken at a fixed level $\alpha<1$ is sparsity preserving  up to any multiplicative factor $A=(A_n)_n$ with  $A_n\to \infty$ over any parameter set $\Theta\subset \ell_0[s_n]$.
\end{corollary}

\begin{proof}
	From \eqref{defBH}, letting $\hat{k}=\sum_{i=1}^n \vphi^{BH}_i$, classically one deduces  that $2\ol{F}_0(\hat{t})=\alpha\hat{k}/n$. Thus, for all $\epsilon\in (0,1-\alpha)$,
	\begin{align*}
		&P_\theta\left(\FDP(\theta,\vphi^{BH}) > \alpha +\epsilon, \sum_{i=1}^n \vphi_i^{BH} \geq s_n\right)\\
		&= P_\theta\left(\sum_{i=1}^n\ind{\theta_i = 0} \ind{2\ol{F}_0(\abs{X_i})\leq \alpha \hat{k}/n} > (\alpha +\epsilon) \hat{k}, \hat{k} \geq s_n\right)\\
		&\leq \sum_{k\geq s_n} P_\theta\left(\sum_{i=1}^n\ind{\theta_i = 0} \ind{2\ol{F}_0(\abs{X_i})\leq \alpha k/n} > (\alpha +\epsilon) k\right)\\
		&\leq \sum_{k\geq s_n} P_\theta\left(\sum_{i=1}^{n-s_n}  \ind{U_i\leq \alpha k/n} > (\alpha +\epsilon) k\right) \leq \sum_{k\geq s_n} P_\theta\left(\sum_{i=1}^{n}  \ind{U_i\leq \alpha k/n} > (\alpha +\epsilon) k\right),
	\end{align*}
	for $U_i$, $1\leq i\leq n$, i.i.d. variables uniformly distributed on $(0,1)$. Now, we have by Bernstein's inequality (Lemma~\ref{th:bernstein}), for any $k\geq s_n$,
	\begin{align*}
		P_\theta\left(\sum_{i=1}^n  \ind{U_i\leq \alpha k/n} > (\alpha +\epsilon) k\right)&\leq P_\theta\left(\sum_{i=1}^n  \ind{U_i\leq \alpha k/n} - \alpha k> \epsilon k\right)\\
		&\leq \exp\left(-0.5 \frac{(\epsilon k)^2}{\alpha k+\epsilon k/3}\right) \leq \exp(-c \epsilon^2 k),
	\end{align*}
	for some universal constant $c>0$ (which can be chosen independent of $\alpha$ by bounding $\alpha$ by $1$). 
	It follows that
	\begin{align*}
		&P_\theta\left(\FDP(\theta,\vphi^{BH}) > \alpha +\epsilon, \sum_{i=1}^n \vphi_i^{BH} \geq s_n\right)\\
		&\leq \sum_{k\geq s_n} \exp(-c \epsilon^2 k)\leq (1-e^{-c \epsilon^2})^{-1} e^{-c \epsilon^2  s_n}=o(1).
	\end{align*}
	The result follows from Lemma~\ref{lemmaBHpreservsparsity2}(ii).
\end{proof}

Let us now give an example  of a (randomized) non-sparsity-preserving procedure for which some FDR/FNR tradeoff is possible.

\begin{example}\label{tradepossible}
	Given a real $b$ and for $\vphi^*$ a procedure satisfying \eqref{err-oracle}, define a (randomized) procedure $\tilde\vphi$ as follows: given a Bernoulli variable $Z$ with success probability $\pi\in[0,1]$, let 
	\[ \tilde\vphi =\tilde\vphi(X,Z)= (1-Z) \vphi^*(X) + Z.\]
	With probability $\pi$, the test $\tilde\vphi$ equals the trivial test $1$, so rejects all null hypotheses. Then 
	\[  \sup_{\te\in\Theta_b} \FDR(\te,\tilde\vphi) = \pi \frac{n-s_n}{n}+(1-\pi) \sup_{\te\in\Theta_b} \FDR(\te,\vphi^*) =\pi (1+o(1)),
	\]
	while the FNR is controlled at level
	\[ \sup_{\te\in\Theta_b} \FNR(\te,\tilde\vphi) = \pi\cdot 0+ (1-\pi)\sup_{\te\in\Theta_b} \FNR(\te,\vphi^*)=
	(1-\pi)\overline{\Phi}(b)(1+o(1)). \]
	In particular, for $\pi>0$, the procedure $\tilde\vphi$ 
	trades a gain of $\pi\overline{\Phi}(b)$ in terms of the FNR  with a loss of $\pi$ in terms of the FDR. The procedure $\vphi^*$ is not sparsity preserving, as it makes $n\gg s_n$ rejections with probability $\pi>0$. 
\end{example}


\section{~~Link between conservative testing and almost full recovery}\label{sec:almostfullrecovery}

We state a lemma to illustrate the link between conservative testing and almost full recovery.

\begin{lemma} \label{lemlink}
	If a test $\vphi$ achieves almost full recovery with respect to the Hamming loss over a subset $\Theta_{s_n}\subset\ell_0[s_n]\backslash \ell_0[s_n-1]$, that is, $\sup_{\te\in\Theta_{s_n}} E_\te \lc(\te,\vphi)/s_n = o(1)$, then it also allows for conservative testing over $\Theta_{s_n}$, that is,  $\sup_{\te\in\Theta_{s_n}}  \mathfrak{R} (\theta,\vphi) \to 0$. 
	
	Conversely, suppose $\vphi$ allows for conservative testing over $\Theta_{s_n}$, and that the number of false discoveries $V(\vphi)=\sum_{i:\te_i=0} \vphi_i$ of $\vphi$ concentrates around its mean in that, 
	\begin{equation} \label{vconc}
		\sup_{\te\in\Theta_{s_n}} P_\te[\abs{V(\vphi)-E_\te V(\vphi)}>E_\te V(\vphi)/2] = o(1).
	\end{equation}
	Then $\vphi$ achieves almost full recovery with respect to the Hamming loss over $\Theta_{s_n}$.
\end{lemma}

\begin{proof}[Proof of Lemma~\ref{lemlink}]
	Recall that the FDP of $\vphi$ can be written $V/\braces{\max(V+S,1)}$, where $V=V(\vphi)$ is the number of false discoveries of $\vphi$ and $S=S(\vphi)$ the number of true discoveries. 
	
	First suppose $\vphi$ achieves almost full recovery with respect to the Hamming loss. This can be written $E[V]+E[s_n-S]=o(s_n)$ uniformly over $\Theta$. Let $\cA=\{V+S> s_n/2\}$, then using Markov's inequality,
	\[ P_\te[\cA^c] \le P_\te[s_n-S\geq
	s_n/2]\le \frac2{s_n}E_\te[s_n-S]\]
	so that $P_\te[\cA^c]=o(1)$. On the other hand,
	\begin{equation*}
		\FDR(\te,\vphi)  = E_\te\left[ \frac{V}{(V+S)\vee 1} \right] 
		\le \frac{2}{s_n}E_\te\left[ V \II_{\cA}\right] + P_\te[\cA^c],
	\end{equation*}
	which implies FDR$(\te,\vphi)=o(1)$ over $\Theta$, while FNR$(\te,\vphi)=o(1)$  follows from $E_\te[s_n-S]=o(s_n)$. 
	
	Conversely, suppose $\vphi$ allows for conservative testing over $\Theta$. Then 
	FNR$(\te,\vphi)=o(1)$ implies $E_\te[s_n-S]=o(s_n)$. Set $\mathcal{B}=\{|V(\vphi)-E_\te V(\vphi)|\le E_\te V(\vphi)/2\}$. By assumption $P_\te[\mathcal{B}^c]=o(1)$ over $\Theta$. Then
	\[ \FDR(\te,\vphi)=E_\te\left[ \frac{V}{(V+S)\vee 1} \right]
	\ge  \frac{E_\te V/2}{3E_\te V/2+s_n}P_{\te}[\mathcal{B}]=\psi(E_\te V/s_n)(1+o(1)),
	\]
	where $\psi$ is the bijective continuous map $x\to x/(3x+2)$ from $(0,\infty)$ to $(0,1/3)$. By assumption $\FDR(\te,\vphi)=o(1)$. As $\psi^{-1}$ is continuous at $0$ with limit $0$, one gets $EV_\te/s_n=o(1)$ which implies combined with the {control of $E[s_n-S]$} that $\vphi$ achieves almost full recovery with respect to the Hamming loss.
\end{proof}

\section{~~Pointwise and local minimax results}

\label{sec:local}
Here we give the claimed ``pointwise'' version of Theorem~\ref{th:minimax-multilevel}. Uniformity will be demanded only over permutations,
\begin{equation}\label{eqn:theta-sigma}
	\theta_\sigma = (\theta_{\sigma(i)} : i\leq n).
\end{equation}
In particular, we allow the possibility of procedures knowing the exact nonzero values of the $\theta_i$, with only their locations unknown. 
Recall from \eqref{eqn:def:Lambda_n(theta)} the definition
\[\Lambda_n(\theta) = s_n^{-1}\sum_{i\in S_\theta} F_{\abs{\theta_i}}(a_n^*).\]

\begin{theorem} \label{thm:pointwise-risk}	Consider the sparse sequence model \eqref{eqn:generalnoisemodel}--\eqref{sparse} and grant Assumption~\ref{ass:generalnoise}\ref{ass:location}. The $\mathfrak{R}$-risk satisfies
	\[  \inf_\vphi \sup_{\sigma}  \mathfrak{R} (\theta_\sigma ,\vphi)= \La_n(\theta) + o(1), \]
	where the supremum is over all permutations $\sigma$ and the infimum over all testing procedures $\vphi$.
	The test $\vphi$ may be chosen independently of $\theta$. Under the assumption that $\Lambda_n(\theta)>c$ for some $c>0$, the remainder $\abs{\inf_\vphi \sup_\sigma \fr(\theta_\sigma, \vphi)-\Lambda_n(\theta)}$ is $o(1)$ \emph{uniformly} in $\theta$.
	
	Under Assumption~\ref{ass:generalnoise}\ref{ass:scale}, the same conclusion holds with $2\Lambda_n(\theta)-1$ in place of $\Lambda_n(\theta)$.
\end{theorem} 
Note that we do not make the $o(1)$ terms explicit here: this will be investigated in detail in Section \ref{sec:prooffastrate} for certain classes of large signals.
\begin{remark}\label{rem:minimaxrates}
This gives an alternative proof of the first part of Theorem~\ref{th:minimax-multilevel} (for $\Lambda_n(\ba)$ bounded away from zero): one applies the above result, noting simply that $\Theta(\ba,s_n)$ is closed under permutations and $\sup_{\theta\in \Theta(\ba,s_n)}\Lambda_n(\theta)=\Lambda_n(\ba)$ in the setting of Theorem~\ref{th:minimax-multilevel}. 
\end{remark}

\begin{remark}\label{rem:BH-pointwise-lowerbound}
	Since the BH procedure $\vphi_\alpha^{BH}$ is invariant under permutations ($\vphi_{\alpha,\sigma(i)}^{BH}(X_\sigma)=\vphi_{\alpha,i}(X)$) we deduce that the risk of the BH procedure at some $\theta$ is lower bounded by $\Lambda_n(\theta)+o(1)$. The upper bound in Theorem~\ref{thm:BHpointwise} is therefore an \emph{equality}, up to $o(1)$ terms.
\end{remark}

\begin{remark}\label{rem:one-tailed}
	The proof of Theorem~\ref{thm:pointwise-risk} does not require the full strength of Assumption~\ref{ass:generalnoise}. In particular, if we omit the assumption that $\overline{F}_a(x)$ is increasing in $a$ from Assumption~\ref{ass:generalnoise}\ref{ass:scale}, and replace the assumption $f_a(x)=f_a(-x)$ and $f_a(x)/f_0(x)$ is increasing in $x>0$ with the single assumption that $f_a(x)/f_0(x)$ increases in $\abs{x}$, the pointwise result still holds. However, these extra assumptions remain necessary for Theorem~\ref{th:minimax-multilevel}, since they ensure that the ``hardest'' $\theta$ lies on the boundary of $\Theta(\ba,s_n)$.
	
	In contrast, we \emph{do} make full use of Assumption~\ref{ass:generalnoise}\ref{ass:location} in the proof, but if one were to perform one-tailed testing instead (i.e.\ $H_{1,i}:\theta_i>0$), then both the symmetry assumption and the monotonicity of $\overline{F}_a(x)$ in $a$ are no longer needed. Proofs are almost identical, with thresholding procedures replacing absolute value thresholding procedures.
\end{remark}
Theorem~\ref{thm-notrade-Subbotin} also has a ``pointwise'' version.

\begin{theorem}\label{thm-notrade-pointwise}
	Consider the setting of Theorem~\ref{thm:pointwise-risk}. 
	For any sequence $B=(B_n)_n$ with $B_n^2\leq \tfrac{1}{3}(n/ s_n)\overline{F}_0(a_n^*-\delta_n)$ and $\varliminf_n B_n >1$, %
	and for $\mathcal{S}_{B}(\mathcal{S}_\theta)$ as in Definition~\ref{def:spapres} the set of sparsity preserving procedures over the set $\mathcal{S}_\theta = \braces{\theta_\sigma : \sigma \text{ a permutation}}$,
	\[ \inf_{\vphi\in \mathcal{S}_{B}(\mathcal{S}_\theta)}\sup_{\sigma} \, \FNR(\te_\sigma,\vphi) = \La_n(\theta)+ o(1).\]
\end{theorem}
Note that $\mathcal{S}_B(\mathcal{S}_\theta)\supseteq \mathcal{S}_B(\Theta)$ if $\Theta \supseteq \mathcal{S}_\theta$, so that this result strengthens the conclusion of Theorem~\ref{thm-notrade-Subbotin}.
 
We omit the proof; it is similar to the proof of Theorem~\ref{thm:pointwise-risk}.
This result in particular implies that asymptotically,   the $\fr$--risk of the BH$(\alpha)$ procedure for fixed $\alpha>0$ must be suboptimal at {\em any} signal $\te$, since it must incur an FNR of at least $\La_n(\te)$ (using invariance by permutation as above to get the result for a specific $\te$ from the above display) and therefore an $\fr$--risk of at least $ \La_n(\te)+\alpha>\La_n(\te)$.

\begin{proof} [Proof of Theorem~\ref{thm:pointwise-risk}]
	The proof is virtually identical to that of Theorem~\ref{th:minimax-multilevel}, particularly in the case of granting Assumption~\ref{ass:generalnoise}\ref{ass:scale}. We outline the adjustments required under Assumption~\ref{ass:generalnoise}\ref{ass:location}.
	
For the upper bound, no change is required, other than skipping the step where we bound $\overline{F}_{\abs{\theta_i}}(a_n^*)\geq \overline{F}_{a_j}(a_n^*)$.
	
For the lower bound, one follows the proof of Theorem~\ref{th:minimax-multilevel} almost exactly, simply noting, with $a_j$ taking the nonzero values of $\theta_i$, that the prior chosen therein gives mass 1 to the set $\braces{\theta_\sigma: \sigma \text{ a permutation}}$. The only step requiring care is that previously we took $a_j>0$ for all $j$, and now some $\theta_i$ may take negative values. This is easily accommodated: in bounding $\sum_{j=1}^{s_n}  P_{X_i\sim f_{a_j}}\left[ \#\{ k\in Q_j\backslash\{i\}\::\:  h(\varepsilon_k,a_j)> h(X_i,a_j) \}\geq \rho\right]$ we now split the sum into $\braces{j: a_j>0}$ and $\braces{j: a_j<0}$. For the former we argue as before that $P_{X_i\sim f_{a_j}}\sqbrackets*{ \#\braces{ k\in Q_j\backslash\{i\}\::\:  h(\varepsilon_k,a_j)> h(X_i,a_j) }\geq \rho}=F_{\abs{a_j}}(a_n^*)+o(1)$ by Lemma~\ref{lem:ControlOfpn}. For the latter, we achieve the same bound by noting that the distribution of $X_i$ under $f_a$ is the same as the distribution of $-X_i$ under $f_{-a}$ (and the distribution of $\eps_k\sim f_0$ is symmetric), so that using $h(x,a):=f_a(x)/f_0(x)=h(-x,-a)$ we have
\begin{align*} &P_{X_i\sim f_{a_j}}\sqbrackets*{ \#\braces{ k\in Q_j\backslash\{i\}\::\:  h(\varepsilon_k,a_j)> h(X_i,a_j) }\geq \rho}\\ =& P_{X_i\sim f_{-a_j}}\sqbrackets*{ \#\braces{ k\in Q_j\backslash\{i\}\::\:  h(-\varepsilon_k,a_j)> h(-X_i,a_j) }\geq \rho} \\ =& P_{X_i\sim f_{\abs{a_j}}}\sqbrackets*{ \#\braces{ k\in Q_j\backslash\{i\}\::\:  h(\varepsilon_k,\abs{a_j})> h(X_i,\abs{a_j}) }\geq \rho},\end{align*} 
which is equal to $F_{\abs{a_j}}(a_n^*)+o(1)$ by the previous case. 

\end{proof}

\section{~~General lower bounds}\label{sec:general-lower-bounds}

We present in this section general lower bounds that can be applied in any model where we observe $X\sim P_\theta$, $\theta\in \R^n$. 
For any prior $\pi$ on $\R^n$, we denote by $P_\pi$ the distribution of $(X,\te)$ in the Bayesian model where $\te\sim \pi$ and $X\given \te\sim P_\te$.

The first result is as follows.

\begin{theorem} \label{thm:bayesLB}
	For any prior $\pi$ on $\R^n$, 
	let $ \ell_i(X)=P_\pi(\theta_i=0\:|\: X)$, $1\leq i\leq n$. Then for all  $\rho>0$, $s_n\geq 1$, and all measurable $\Theta\subset \R^n$, 
	 the minimax risk $\inf_\vphi \sup_{\te\in\Theta }  \mathfrak{R} (\theta,\vphi)$ is lower bounded by the two following quantities:  
	\begin{itemize}
\item[(i)]	For any $\eta\in (0,1)$,
$$
	\inf_\vphi \sup_{\te\in\Theta }  \mathfrak{R} (\theta,\vphi)  \geq   \left(\lambda \wedge \frac{\rho \lambda}{1+\rho \lambda}\right)  (1-  e^{- c\eta^2 M_\rho})- n (1\vee \rho) P_\pi(\norm{\theta}_0> s_n) - 2 P_\pi(\theta\notin \Theta),
		$$
	for some universal constant $c>0$, where $\lambda=\frac{ M_\rho (1-\eta)}{s_n}$ and 
	$$M_\rho= \sum_{i=1}^n P_\pi[ \theta_i\neq 0,\ell_i(X)> \rho/(1+\rho)] .$$ 
	\item[(ii)] For $m_{\pi,\rho}=P_\pi(\exists i\in \{1,\dots,n\}\::\: \theta_i\neq 0,\ell_i(X)> \rho/(1+\rho))$,
	\begin{align*}
\inf_\vphi \sup_{\te\in\Theta }  \mathfrak{R} (\theta,\vphi)\geq  \left( \frac{\rho}{s_n + \rho} \wedge \frac{1  }{s_n}  \right) m_{\pi,\rho}- n (1\vee \rho)P_\pi(\norm{\theta}_0> s_n)- 2 P_\pi(\theta\notin \Theta).
	\end{align*}
	\end{itemize}
	\end{theorem}
		
		The lower bound (i) says roughly that the Bayes risk for $ \mathfrak{R}$ is lower bounded by the type-two error of the Bayes procedure for the $\rho$-weighted classification risk problem. 	
		Note that the bounds above are true for any prior on $\R^n$. Nevertheless, we should choose the sparsity and the signal strength wisely to both make $M_\rho$ large and make $P_\pi(\norm{\theta}_0> s_n)$ small in the lower bound. 
		
	\begin{proof}
	For all $\theta\in \R^n$ and $\vphi$, writing $D_n(X) = \sum_{i\leq n} \vphi_i(X)$ and let
	$$
	Q(\theta,\vphi,X)= \sum_{i=1}^n \left\{ \ind{\te_i=0} \frac{\vphi_i(X)}{1\vee D_n(X)}
	+ \ind{\te_i\neq 0} \frac{1-\vphi_i(X)}{1\vee \norm{\theta}_0} \right\}.
	$$
	so that $ \mathfrak{R} (\theta,\vphi) =E_\theta Q(\theta,\vphi,X)$. Classically, since $Q(\theta,\vphi,X)\leq 2$, we have 
	$$
	\fr_\pi^*:=E_\pi \mathfrak{R} (\theta,\vphi) \leq \sup_{\te\in\Theta }  \mathfrak{R} (\theta,\vphi) + 2 P_\pi(\theta\notin \Theta).
	$$
	Hence, we only have to prove a lower bound for $\fr_\pi$.
	For this, let
	\begin{equation}\label{equLrho}
	L_\rho(\te,\varphi)= \sum_{i=1}^n \left\{ \ind{\te_i=0} \vphi_i(X)+ \rho \ind{\te_i\neq 0} (1-\vphi_i(X)) \right\}.
	\end{equation}
For any $\delta_n>0$, whenever $D_n(X)\leq s_n(1+\delta_n)$ and $\norm{\theta}_0\leq s_n$, we have that $Q(\theta,\vphi,X)$ is at least 
	\[
	\left(\frac{1}{1+\delta_n}\wedge \rho^{-1} \right)	\frac{L_\rho(\te,\varphi)}{s_n}
	\]
	
	Also, if $D_n(X)\geq s_n(1+\delta_n)$ (hence $D_n(X)\geq 1$) and $\norm{\theta}_0\leq s_n$, we have 
	\[ Q(\theta,\vphi,X)
\geq  \frac{D_n(X)-\norm{\theta}_0}{ D_n(X)}\geq  \frac{D_n(X)-s_n}{   D_n(X)}\geq  \frac{\delta_n}{1+\delta_n}.\]
	Hence, we have for all $\delta_n>0$, 
	\begin{align*}
	&E_\pi[Q(\theta,\vphi,X)] = E_\pi[E_\pi[Q(\theta,\vphi,X)\:|\:X]] \\
	&\geq  E_\pi\left(\ind{D_n(X)\leq s_n(1+\delta_n)} 
	 \left(\frac{1}{1+\delta_n}\wedge \rho^{-1}\right)\frac{E_\pi[L_\rho(\te,\varphi) \ind{\norm{\theta}_0\leq  s_n}\:|\:X]}{s_n}\right.\\
	&+\left.\ind{D_n(X)> s_n(1+\delta_n)} \frac{\delta_n}{1+\delta_n}E_\pi[ \ind{\norm{\theta}_0\leq  s_n}\:|\:X] \right).
	\end{align*}
	Since
	\begin{align*}
	E_\pi[L_\rho(\te,\varphi) \ind{\norm{\theta}_0\leq  s_n}\:|\:X] & =E_\pi[L_\rho(\te,\varphi)\:|\:X]-E_\pi[L_\rho(\te,\varphi) \ind{\norm{\theta}_0>  s_n}\:|\:X] \\
	& \ge E_\pi[L_\rho(\te,\varphi)\:|\:X]-n(1\vee \rho) P_\pi[ \norm{\theta}_0>  s_n\:|\:X],
	\end{align*}
	 we have almost surely in $X$,
	\begin{align*}
	&\ind{D_n(X)\leq s_n(1+\delta_n)} 
	 \left(\frac{1}{1+\delta_n}\wedge \rho^{-1}\right)\frac{E_\pi[L_\rho(\te,\varphi) \ind{\norm{\theta}_0\leq  s_n}\:|\:X]}{s_n}\\
	&+\ind{D_n(X)> s_n(1+\delta_n)} \frac{\delta_n}{1+\delta_n}E_\pi[ \ind{\norm{\theta}_0\leq  s_n}\:|\:X] \\
	&\geq \frac{\delta_n}{1+\delta_n} \wedge \left\{ \left(\frac{1}{1+\delta_n}\wedge \rho^{-1}\right)\frac{\inf_\vphi E_\pi (L_\rho(\te,\vphi)|X)}{s_n}\right\} - n(1\vee \rho) P_\pi[ \norm{\theta}_0>  s_n\:|\:X].
		\end{align*}
%
	By integration, this gives 
	\begin{align*}
	 \fr_\pi^* &\geq E_\pi\left( \frac{\delta_n}{1+\delta_n} \wedge \left\{ \left(\frac{1}{1+\delta_n}\wedge \rho^{-1}\right)\frac{\inf_\vphi E_\pi (L_\rho(\te,\vphi)|X)}{s_n}\right\}  \right)- n(1\vee \rho)  P_\pi(\norm{\theta}_0> s_n)
	 \end{align*}
	 Now solving the Bayes problem for the (weighted) classification loss, we have 
	 \begin{align}
	 \inf_\vphi E_\pi (L_\rho(\te,\vphi)|X)& = \sum_{i=1}^n \left\{ \ell_i(X) \ind{\ell_i(X)\leq \rho/(1+\rho)}
	+ \rho(1-\ell_i(X))  \ind{\ell_i(X)> \rho/(1+\rho)} \right\} \nonumber\\
	&\geq \rho L'_\rho,\label{equLprimerholb}
	 \end{align}
	 by letting 
	 \begin{equation}\label{equLrhoprime}
	 L'_\rho=\sum_{i=1}^n 	 (1-\ell_i(X))  \ind{\ell_i(X)> \rho/(1+\rho)} .
	 \end{equation}
	  This entails 
	 \begin{align}
	\fr_\pi^* &\geq E_\pi\left( \frac{\delta_n}{1+\delta_n} \wedge \left\{ \left(\frac{1}{1+\delta_n}\wedge \rho^{-1}\right)\frac{\rho L'_\rho}{s_n}\right\}  \right)- n (1\vee \rho)P_\pi(\norm{\theta}_0> s_n)\label{equinterm}
	\end{align}
	For proving (i), we observe that the right-hand side of \eqref{equinterm} is at least
	 \begin{align*}
	 \frac{\delta_n}{1+\delta_n} \wedge \left\{ \left(\frac{\rho}{1+\delta_n}\wedge 1\right)\frac{M_\rho (1-\eta) }{s_n}\right\}  (1-  e^{- c \eta^2 M_\rho})- n (1\vee \rho) P_\pi(\norm{\theta}_0> s_n)\nonumber
	,
	\end{align*}
	 by recalling $M_\rho=E_\pi L'_\rho=\sum_{i=1}^n E_\pi[(1-\ell_i(X))  \ind{\ell_i(X)> \rho/(1+\rho)}]$ and because by Bernstein's inequality,
	$$
P_\pi\left(L'_\rho < M_\rho (1-\eta)\right)\leq e^{- c \eta^2 M_\rho},
	$$
	for $c>0$ some constant. The result (i) then follows by letting $\delta_n=\rho M_\rho (1-\eta)/s_n$.
	To prove (ii), we observe that the right-hand side of \eqref{equinterm} is at least
	\begin{align*}
 \left( \frac{\delta_n}{1+\delta_n} \wedge \left\{ \left(\frac{1}{1+\delta_n}\wedge \rho^{-1}\right)\frac{\rho  }{s_n}\right\}  \right) P_\pi(L'_\rho\geq 1)- n (1\vee \rho)P_\pi(\norm{\theta}_0> s_n)
	\end{align*}
	The result (ii) follows by letting $\delta_n=\rho /s_n$.
%
%

	\end{proof}
	
The following lower bound is similar to the one of Theorem~\ref{thm:bayesLB} (i), but for the weighted risk
\begin{equation}\label{riskW} \mathfrak{R}_W(\theta,\vphi) =W. \FDR(\theta,\vphi)+\FNR(\theta,\vphi), 
\end{equation}
where $W>0$ is some weight.  
In addition, while it is true for any procedure, it is particularly suitable for sparsity preserving procedures (see Definition~\ref{def:spapres}).

\begin{theorem} \label{thm:bayesLB2}
	For any prior $\pi$ on $\R^n$, let $ \ell_i(X)=P_\pi(\theta_i=0\:|\: X)$, $1\leq i\leq n$. Then for all  $\rho,W,B>0$, $s_n\geq 1$, all measurable $\Theta\subset \R^n$, and all multiple testing procedures $\vphi$, we have for any $\eta\in (0,1)$,
\begin{align*}
	 \sup_{\te\in\Theta }  \mathfrak{R}_W (\theta,\vphi)  \geq &  \left(1 \wedge \frac{\rho W}{B}\right) \left(1 \wedge \frac{M_\rho(1-\eta)}{s_n}\right)(1-e^{- c\eta^2 M_\rho})  -  \sup_{\te\in\Theta}P_\theta\left(D_n(X)>Bs_n\right)   \\&- \left(\frac{W}{B}\wedge \rho^{-1} \right) (B+\rho) P_\pi(\norm{\theta}_0> s_n)- (W+2) P_\pi(\theta\notin \Theta),
		\end{align*}
	for some universal constant $c>0$, where $ \mathfrak{R}_W$ denotes the weighted risk \eqref{riskW} and where 
	$M_\rho= \sum_{i=1}^n P_\pi[ \theta_i\neq 0,\ell_i(X)> \rho/(1+\rho)] .$ 
	\end{theorem}

	\begin{proof}
	We use the same method as in the proof of Theorem~\ref{thm:bayesLB} (i), although we provide the full proof here for the sake of completeness.
	For all $\theta\in \R^n$ and $\vphi$, writing $D_n(X) = \sum_{i\leq n} \vphi_i(X)$ and let
	$$
	Q_W(\theta,\vphi,X)= \sum_{i=1}^n \left\{ \ind{\te_i=0} W\frac{\vphi_i(X)}{1\vee D_n(X)}
	+ \ind{\te_i\neq 0} \frac{1-\vphi_i(X)}{\norm{\theta}_0} \right\}.
	$$
	so that $ \mathfrak{R}_W (\theta,\vphi) =E_\theta Q_W(\theta,\vphi,X)$. Classically, since $Q_W(\theta,\vphi,X)\leq W+1$, we have 
	$$
E_\pi \mathfrak{R}_W (\theta,\vphi) \leq \sup_{\te\in\Theta }  \mathfrak{R}_W (\theta,\vphi) + (W+1) P_\pi(\theta\notin \Theta).
	$$
	Hence, we only have to prove
	\begin{align}
E_\pi \mathfrak{R}_W (\theta,\vphi)&\geq \left(1 \wedge \frac{\rho W}{B}\right) \left(1\wedge \frac{M_\rho(1-\eta)}{s_n}\right)(1-e^{- c\eta^2 M_\rho})  \nonumber\\
&-  \sup_{\te\in\Theta}P_\theta\left(D_n(X)>Bs_n\right) - P_\pi(\theta\notin\Theta) - \left(\frac{W}{B}\wedge \rho^{-1} \right) (B+\rho) P_\pi(\norm{\theta}_0> s_n).
\label{toprovelb2}
	\end{align}
	Recall $L_\rho(\te,\varphi)$ in \eqref{equLrho}.
Whenever $D_n(X)\leq B s_n$ and $\norm{\theta}_0\leq s_n$, we have that $Q_W(\theta,\vphi,X)$ is at least 
	\[
	\left(\frac{W}{B}\wedge \rho^{-1} \right)	\frac{L_\rho(\te,\varphi)}{s_n}
	\]
	Hence, we have
	\begin{align*}
	&E_\pi[Q_W(\theta,\vphi,X)] \ge E_\pi[E_\pi[Q_W(\theta,\vphi,X)\:|\:X]] \\
	&\geq  E_\pi\left(\ind{D_n(X)\leq B s_n} 
	 \left(\frac{W}{B}\wedge \rho^{-1} \right)
	  \frac{E_\pi[L_\rho(\te,\varphi) \ind{\norm{\theta}_0\leq  s_n}\:|\:X]}{s_n}\right)\\
	  &\geq  \left(\frac{\rho W}{B}\wedge 1 \right)  E_\pi\left(\ind{D_n(X)\leq B s_n} 
	  \frac{L'_\rho}{s_n}\right) - \left(\frac{W}{B}\wedge \rho^{-1} \right) (B+\rho) P_\pi(\norm{\theta}_0> s_n),
	\end{align*}
	because 	$L_\rho(\te,\varphi)/s_n \leq B+\rho$ when $D_n(X)\leq B s_n$ 
	and by using \eqref{equLprimerholb} and defining $L'_\rho$ as in \eqref{equLrhoprime}.
	Next, we have
	\begin{align*}
	 E_\pi\left(\frac{L'_\rho}{s_n}\ind{D_n(X)\leq B s_n} 
	  \right) 
	  &\geq E_\pi\left(1\wedge \frac{L'_\rho}{s_n}\right) - P_\pi\left(D_n(X)> B s_n \right)\\
	  &\geq \left(1\wedge \frac{M_\rho(1-\eta)}{s_n}\right)P_\pi(L'_\rho \geq M_\rho (1-\eta)) - P_\pi\left(D_n(X)> B s_n \right).
	  \end{align*}
	We obtain \eqref{toprovelb2} because
	$
	P_\pi\left(D_n(X)> B s_n \right)\leq  \sup_{\theta\in \Theta} P_\theta\left(D_n(X)> B s_n \right) + P_\pi\left(\theta\notin \Theta\right)
	$
	and
	$
P_\pi\left(L'_\rho < M_\rho (1-\eta)\right)\leq e^{- c \eta^2 M_\rho},
	$
	for $c>0$ some constant (as shown in the proof of Theorem~\ref{thm:bayesLB} (i)). 
	\end{proof}

The following lower bound is similar to the one of Theorem~\ref{thm:bayesLB} (i), but somewhat more classical, because it is for the classification risk $E_\theta  \lc(\te,\vphi)/s_n$, where $ \lc(\te,\vphi)$ is the classification loss given by \eqref{equLC}. 

\begin{theorem} \label{thm:bayesLB-classif}
	For any prior $\pi$ on $\R^n$, 
	let $ \ell_i(X)=P_\pi(\theta_i=0\:|\: X)$, $1\leq i\leq n$. Then for all  $s_n\geq 1$, and all measurable $\Theta\subset \R^n$, we have
$$
	\inf_\vphi \sup_{\te\in\Theta }  E_\theta  \lc(\te,\vphi)/s_n \geq  M/s_n- n P_\pi(\theta\notin \Theta),
		$$
where $M=\sum_{i=1}^n P_\pi[ \theta_i\neq 0,\ell_i(X)> 1/2]$ (that is, $M$ is $M_\rho$ of Theorem~\ref{thm:bayesLB} (i) for $\rho=1$).
	\end{theorem}

\begin{proof}
The proof is analogous to the one of Theorem~\ref{thm:bayesLB} (i), with important simplifications. Since $\lc(\te,\vphi)\leq n$ pointwise, we first have for all $\vphi$,
$$\sup_{\te\in\Theta }  E_\theta  \lc(\te,\vphi)/s_n \geq E_{\pi} \lc(\te,\vphi)/s_n - n P_\pi(\theta\notin \Theta).$$
Then the result follows using arguments as before because 
$$
\inf_\vphi E_{\pi} \lc(\te,\vphi) \geq E_{\pi} \inf_\vphi [E_{\pi} [\lc(\te,\vphi)\:|\:X]] \geq M,
$$
because the Bayes rule is given by $\vphi_i=\ind{\ell_i(X)>1/2}$.
\end{proof}

The above generic lower bound will be used in combination with the following quantitive result, which is closely related to an observation made in Section~\ref{sec:VerificationOfAssumption}. There it was noted that $o_P(s_n)$ of the nulls exceeded $a_n^*$ in absolute value, while many more than $s_n$ exceed $a_n^*-\delta_n$. This formed the basis of the sketch proof that the optimal thresholding procedure thresholds at $a_n^*$. Here we instead prove that in any collection of roughly $n/s_n$ nulls, there will be at least $\rho$ of these taking absolute values between $a_n^*-\delta_n$ and $a_n^*$. This is key in proving our lower bound over {\em all} procedures, using that the classification risk can be decomposed into the sum of risks over blocks (see the proof of Theorem~\ref{th:minimax-multilevel}).

\begin{lemma}\label{lem:ControlOfpn}
	
	Under Assumption~\ref{ass:generalnoise}\ref{ass:location}, for any integer sequence $\rho=\rho_n$ satisfying \[1\leq \rho \leq (n/2s_n)\overline{F}_0(a_n^*-\delta_n)-1\] we have 
	\[ P_{X_1\sim f_a}(\# \braces{2\leq i \leq n/s_n : \eps_i>X_1}\leq \rho-1) = \overline{F}_{a}(a_n^*)+o(1).\]
	In particular, we  may choose $\rho$ tending to infinity or $\rho=1$. The $o(1)$ term is uniform in $a$. 
	
	Under Assumption~\ref{ass:generalnoise}\ref{ass:scale}, for the same condition on $\rho$ we instead have
	\[ P_{X_1\sim f_a}(\#\braces{2 \leq i \leq n/s_n : \abs{\eps_i}>\abs{X_1}}\leq \rho-1)= 2\overline{F}_a(a_n^*)+o(1).\]
\end{lemma}
\begin{proof}
Under Assumption~\ref{ass:generalnoise}\ref{ass:location}, with $a_n^*,$ $\delta_n$ as in the assumption, write $A_n$ for the event $A_n = \braces[\big]{ \#\braces{2\leq i \leq n/s_n : \eps_i>X_1\sim f_a}\leq \rho-1}$ and set $U_n = \#\braces{2 \leq i \leq n/s_n : \eps_i > a_n^* - \delta_n}$; note that
\[U_n \sim \operatorname{Bin}(\floor{n/s_n}-1, \overline{F}_0 (a_n^*-\delta_n)).\] Define $V_n$ correspondingly without the $\delta_n$.

Observe that	\[ E[U_n] \to \infty,\]
and similarly we deduce
\[ E[V_n] = (\floor{n/s_n}-1)\overline{F}_0(a_n^*)\to 0.\]
For any $\rho=\rho_n\geq 1$ such that $\rho-1<E[U_n]/2$ (which is true under the specified condition on $\rho$), we may apply Chebyshev's inequality with the bound $\Var(U_n)\leq E[U_n]$ to obtain
\[ P(U_n \leq \rho-1) \leq  P( \abs{U_n-E[U_n]} \geq E[U_n]-(\rho-1)) \leq P\brackets[\Big]{\abs{U_n-E[U_n]}>\frac{E U_n}{2}} \leq \frac{4}{E U_n}\to 0.\]
Similarly, applying Markov's inequality, we have for any $\rho\geq 1$,
\[ P(V_n>\rho-1) = P(V_n\geq  \rho)\leq  \frac{E V_n}{\rho}\to 0.\]

We have thus shown that on an event $B_n$ of probability tending to 1, 
$U_n\geq \rho$ and $V_n\leq \rho-1$.  In words, on $B_n$, at most $\rho-1$ of the numbers $(\eps_i,~2\leq i\leq n/s_n)$ are larger than $a_n^*$ and at least $\rho$ of them are larger than $a_n^*-\delta_n$. It follows that, on $B_n$, the event $A_n$ holds if $X_1>a_n^*$ and fails if $X_1\leq a_n^*-\delta_n$. Thus, 
\begin{align*} P(A_n) \geq P_a(X_1 > a_n^*) - P(B_n^c),\\
	P(A_n^c) \geq P_a(X_1 \leq a_n^*-\delta_n) -P(B_n^c).
\end{align*}
Since $F_a$ is Lipschitz 
we deduce the result.  [Note we eventually apply this result simultaneously for multiple values of $a$, hence the demand that the $F_a$ be \emph{uniformly} Lipschitz.] 

Under Assumption~\ref{ass:generalnoise}\ref{ass:scale} we instead define $U_n = \#\braces{2 \leq i \leq n/s_n:\abs{\eps_i}>a_n^*-\delta_n}$ and similarly for $V_n$. The same arguments as before then imply that $U_n\geq \rho$ and $V_n\leq \rho-1$ with probability tending to 1, and hence the probability in question is 
\[ P(\abs{X_1}>a_n^*)+o(1)= 2\overline{F}_a(a_n^*)+o(1).\qedhere\]
\end{proof}

\section{~~Some remaining proofs}\label{sec:remaining-proofs}

\subsection{~~Proof of Theorem~\ref{thm-adapt-r}}

Given the lower bound provided by Theorem~\ref{thmgeneric}, Theorem~\ref{thm-adapt-r} is a consequence of Theorems~\ref{thm:lval-multilevel} and \ref{thm:BHpointwise} and Remark~\ref{rem:polysparse}. Indeed, for the BH procedure, in Theorem~\ref{thm:BHpointwise} one chooses $f_a(x)=\phi(x-a)$ to be the Gaussian density; by Lemma~\ref{lem:assumptionok} and Remark~\ref{rem:polysparse}, Assumption~\ref{ass:generalnoise}\ref{ass:location} is satisfied with $a_n^*=\sqrt{2\log(n/s_n)}+\kappa_n$ for $\kappa_n=\log(1/\alpha_n)/\sqrt{\log(n/s_n)}$ and under polynomial sparsity \eqref{polspa} the condition on $\alpha_n$ of Theorem~\ref{thm:BHpointwise} is implied by that of Theorem~\ref{thm-adapt-r}
;
choosing $\ba=(\sqrt{2\log(n/s_n)}+b,\sqrt{2\log(n/s_n)}+b,\dots,\sqrt{2\log(n/s_n)}+b)\in \RR^{s_n}$, the set $\Theta(\ba,s_n)$ exactly equals $\Theta_b$, and $\Lambda_n(\ba)=\overline{\Phi}(b)+o(1)$. Note that $\sqrt{2\log(n/s_n)}+b>0$ for $n$ large enough, so that this choice of $\ba$ is permitted.

For the $\ell$-value procedure, given the multilevel bound Theorem~\ref{thm:lval-multilevel}, the exact same reasoning applies.
\subsection{~~Proof of Theorems~\ref{thm-notrade} and \ref{thm-notrade-Subbotin}}

Note from Lemma~\ref{lem:assumptionok} (or see Remark~\ref{rem:polysparse}) that $a_n^*=\sqrt{2\log(n/s_n)}$, $\delta_n= (\log(n/s_n))^{-1/4}$ is valid in Assumption~\ref{ass:generalnoise} in the Gaussian setting of Theorem~\ref{thm-notrade}. One calculates (similarly to  \eqref{condF0anstarmoinsdelta})
that $(n/3s_n)2\overline{\Phi}(\sqrt{2\log(n/s_n)}-(\log (n/s_n))^{-1/4})\geq \exp((\log(n/s_n))^{1/4})$ for $n$ large.  It then follows, up to minor technical details as given in the proof of Theorem~\ref{thmgeneric}, that Theorem~\ref{thm-notrade} is implied by Theorem~\ref{thm-notrade-Subbotin}, 
hence we restrict our attention to proving the latter result. 

It follows from computations in Section~\ref{sec:ub} that the oracle procedure $\vphi_{a_n^*}$ considered in the proof of Theorem~\ref{th:minimax-multilevel} makes $o_P(s_n)$ false discoveries and hence belongs to the class $\mathcal{S}_{B}$ (see also  Section~\ref{sec:sparsitypreserving}) and that its FNR is controlled at level $\Lambda_n(\ba)$+o(1), proving that 
\begin{equation*}
	\inf_{\vphi \in \mathcal{S}_B} \sup_{\theta\in \Theta(\ba,s_n)} \FNR(\theta,\vphi) \leq \Lambda_n(\ba)+o(1).
\end{equation*}
 [Note the trivial procedure $\vphi=0$, also used in the proof for the case $\Lambda_n=1+o(1)$, is sparsity preserving.]

To show a corresponding lower bound, the idea is to work with the weighted risk. 
More precisely, assume one can show, for a suitable sequence $W_n\to 0$, that 
\begin{equation} \label{trlbnew}
	\inf_{\vphi\in \mathcal{S}_{B}} \sup_{\theta\in \Theta(\ba,s_n)} \left[
	W_n \FDR(\te,\vphi) + \FNR(\te,\vphi)\right] \ge \Lambda_n(\ba) + o(1).
\end{equation}
Then by bounding the FDR from above by $1$ one gets
\[ W_n + \inf_{\vphi\in \mathcal{S}_{B}} \sup_{\theta\in \Theta(\ba,s_n)} \FNR(\te,\vphi) \ge \Lambda_n(\ba)+o(1),\]
which gives the desired lower bound since $W_n\to 0$.
To prove \eqref{trlbnew}, we follow a method similar to the lower bound part of the proof of Theorem~\ref{th:minimax-multilevel} (see Section~\ref{sec:nonasymptotic-lower-bounds}), the only difference being that we consider the weighted risk and use the general lower bound of Theorem~\ref{thm:bayesLB2} instead of Theorem~\ref{thm:bayesLB}. More precisely, applying Theorem~\ref{thm:bayesLB2} with the same prior as in Section~\ref{sec:nonasymptotic-lower-bounds}, with $B=B_n$ as in the statement of Theorem~\ref{thm-notrade-Subbotin},  $\rho_n=\floor{n/(3s_n) \overline{F}_0(a_n^*-\delta_n)}$ and $\eta_n=s_n^{-1/4}$, and noting that $P_\pi(\norm{\theta}_0>s_n)=P_\pi(\theta\not\in \Theta)=0$ by construction and $\sup_{\theta\in\Theta}P_\theta(\#\braces{i : \vphi_i =1}>B_ns_n)=o(1)$ by assumption, we obtain 
\begin{align*}
	&\inf_{\vphi\in \mathcal{S}_{B}} \sup_{\theta\in \Theta(\ba,s_n)} \left[
	W_n \FDR(\te,\vphi) + \FNR(\te,\vphi)\right]  \\
	\geq &  \left(1 \wedge \frac{\rho W_n}{B_n}\right) \left(1 \wedge \frac{M_\rho(1-\eta)}{s_n}\right)(1-e^{- c\eta^2 M_\rho})  + o(1) ,
		\end{align*}
	for some universal constant $c>0$ and where 
	$M_\rho= \sum_{i=1}^n P_\pi[ \theta_i\neq 0,\ell_i(X)> \rho/(1+\rho)] .$ 
As proved in Section~\ref{sec:nonasymptotic-lower-bounds}, we have $M_\rho/s_n\geq \Lambda_n(\ba)+o(1)$ under Assumption~\ref{ass:generalnoise}\ref{ass:location} and $M_\rho/s_n\geq 2\Lambda_n(\ba)-1+o(1)$
under Assumption~\ref{ass:generalnoise}\ref{ass:scale}, see \eqref{equMrhoboundedA}--\eqref{equMrhoboundedB}. This implies $e^{- c\eta^2 M_\rho}=o(1)$. By assumption $B_n\leq \sqrt{\rho}$, so taking $W_n=1/\sqrt{\rho}$ leads to \eqref{trlbnew}.

\subsection{~~Proof of Theorem~\ref{thm-adapt-cl}}\label{sec:proof-of-thm-adapt-cl} 

For the lower bound, we use Theorem~\ref{thm:bayesLB-classif} for the prior given in Section~\ref{sec:nonasymptotic-lower-bounds} and $\Theta=\Theta(a_b,s_n)$ (clearly, this will give the same lower bound over the larger class $\Theta_b'$). We then have, for $M=M_1$,
$$
	\inf_\vphi \sup_{\te\in\Theta(a_b,s_n) }  E_\theta  \lc(\te,\vphi)/s_n \geq  M/s_n,
		$$
because $P_\pi(\theta\notin\Theta(a_b,s_n))=0$ and for $M/s_n\geq \Lambda_n(\ba) +o(1) =  \overline\Phi(b)+o(1)$ by using \eqref{equMrhoboundedA} for $\rho=1$ and in the Gaussian case.


For the upper bound, given $\te \in \Theta_b'(s_n)$ write $S_\theta=\braces{i: \theta_i\neq 0}$ for the support of $\theta$ and write $s_\theta=\abs{S_\theta}$. 
We can decompose the classification loss $\lc(\te,\vphi)$ as the sum $V+(s_\theta-S)$, where $V=V(\vphi,\theta)=\sum_{i \not \in S_\theta} \vphi_i$ and $S=S(\vphi,\theta)=\sum_{i\in S_\theta}\vphi_i$.
It suffices to show that 
\begin{align} \label{eqn:Es-S-upper-bound} \sup_{\theta\in \Theta_b'(s_n)} E_\theta [s_\theta- S] &\leq s_n\overline{\Phi}(b)+o(s_n), \\
	\label{eqn:EV-upper-bound} \sup_{\theta \in \Theta_b'(s_n)} E_\theta V &=o(s_n).
\end{align}
\\
We begin with \eqref{eqn:Es-S-upper-bound}. This holds trivially when $b=b_n\to -\infty$ since $0\leq s_\theta-S\leq s_\theta\leq s_n$ for $\theta\in\Theta_b'(s_n)$, so we restrict to  the cases $b\in\RR$ fixed or $b=b_n\to\infty$. It suffices to show, for a sequence $s_n'$ satisfying $1\ll s_n'\ll s_n$, that 
\[
\max_{s_n'\leq s\leq s_n} \sup_{\theta\in \Theta_b(s)} E_\theta [s- S] 
\le s_n\overline{\Phi}(b)+o(s_n),
\]
since for $\theta\in\Theta_b(s)$ with $0\leq s\leq s_n'$, we have $0\leq s_\theta-S\leq s_\theta=s \leq s_n' =o(s_n)$.
Let $s_n''$ denote a sequence such that
\[ \max_{s_n'\leq s\leq s_n} \sup_{\theta\in\Theta_b(s)} E_\theta[s-S] = \sup_{\theta\in\Theta_b(s_n'')} E_\theta[s_n''-S].\]
\\
We now apply Theorem~\ref{thm-adapt-r} with $s_n''$ in place of $s_n$ (note the asymptotics \eqref{sparse}
hold). We have
\[ (s_n'')^{-1} E_\theta [s_n''-S] \leq \frak{R}(\theta,\vphi) \leq \overline{\Phi}(b)+o(1)\]
and we deduce that $E_{\theta}[s_n''-S] \leq s_n''(\overline{\Phi}(b)+o(1)) \leq s_n(\overline{\Phi}(b)+o(1))$ as required.
\\
For \eqref{eqn:EV-upper-bound}, the proof differs depending on whether $\vphi$ is the $\ell$-value procedure or the BH procedure. For the former one notes that the proof of \eqref{eqn:EVbound-classification} in Section~\ref{sec:proof:thm-adapt-clpr} holds for $\theta\in\Theta_b'(s_n)$ not just $\theta\in\Theta_b(s_n)$. Taking the supremum yields the claim.
For the latter one uses Lemma~\ref{lem:BHNR} in Section~\ref{sec:BHuseful-lemmas} and the fact that $\alpha=\alpha_n=o(1)$. 
An alternative proof of this upper bound in the $\ell$-value case is provided in Remark~\ref{sec:proof:thm-adapt-cl}.

\section{~~Materials for BH procedure}\label{secBH}

\subsection{~~Definition}

Recall that the BH procedure of level $\alpha$ is defined as follows (see, e.g., \cite{Roq2011}):
\begin{align}
	\vphi^{BH}_\alpha&=(\II\braces{\abs{X_i}\geq \hat{t} })_{1\leq i\leq n}\label{defBH};\\
	\hat{t} &= \min\braces{t\in \R\cup\braces{-\infty }\::\: \hat{G}_n(t)\geq 2\ol{F}_0(t)/\alpha};\nonumber\\
	\hat{G}_n(t) &= n^{-1}\sum_{i=1}^n \II\braces{\abs{X_i}\geq t}\nonumber.
\end{align}

\subsection{~~Proof of Theorem~\ref{thm:BHpointwise}}\label{sec:proofBHfirsth}

First, by \cite{BH1995,BY2001} we have any $\theta$,
\begin{align}\label{FDRBH}
	\FDR(\theta,\vphi^{BH}_\alpha) = \alpha(n-\abs{S_\theta})/n\leq \alpha,
\end{align}
where $S_\theta$ denotes the support of $\theta$.
It thus suffices to prove
\[ \FNR(\theta,\vphi_\alpha^{BH})\leq \Lambda_n(\theta) + e^{-cs_n},\quad c=(1-\Lambda_n(\theta))^2/32.\] 

We first give the proof under Assumption~\ref{ass:generalnoise}\ref{ass:location}.
We rely on the following inequalities: for all $t\geq 0$,
\begin{align*}
	P_{\theta}(\hat{t}>t) &\leq  P_{\theta}(\hat{G}_n(t)< 2\ol{F}_0(t)/\alpha)=P_{\theta}(\hat{G}_n(t)-G_n(t)< 2\ol{F}_0(t)/\alpha-G_n(t)),
\end{align*} 
where using the symmetry (see Lemma~\ref{lem:assumption-implies-symmetry}) \[G_n(t)=E_\theta \hat{G}_n(t) = (1-s_n/n) 2\ol{F}_0(t) + n^{-1}\sum_{i\in S_\theta}  \brackets[\big]{\ol{F}_{\theta_i}(t) + \ol{F}_{-\theta_i}(t)}.\]

Let $a_n^*,\delta_n$ be as in Assumption~\ref{ass:generalnoise}. 
Note that $G_n(t)\geq (1-s_n/n)2\overline{F}_0(t)+n^{-1}\sum_{i\in S_\theta} \overline{F}_{\abs{\theta_i}}(t)$ and observe that, defining $\Psi_{\theta}:t\mapsto s_n^{-1} \sum_{i\in S_\theta}\ol{F}_{\abs{\theta_i}}(t)/\ol{F}_0(t)$, the restriction of $\Psi_\theta$ to $[0,\infty)$ is a continuous increasing bijection to $[\Psi_\theta(0),\infty)$ under \eqref{equcondalpha} by Lemma~\ref{lem:tnstar_general}. 
Hence, denoting by 
$t^*_n$ the only point $t>0$ (which exists for $n$ large by the lemma) such that  $(1-s_n/n) 2\ol{F}_0(t) + (s_n/n) s_n^{-1} \sum_{i \in S_\theta}\ol{F}_{\theta_i}(t)=3\ol{F}_0(t)/\alpha$, 
we obtain
\begin{align*}
	P_{\theta}(\hat{t}>t^*_n) &\leq  P_{\theta}(\hat{G}_n(t^*_n)-G_n(t^*_n)< -\ol{F}_0(t^*_n)/\alpha)\\
	&\leq \exp\left(-0.5 \frac{n^2 \ol{F}_0^2(t^*_n)/\alpha^2}{n G_n(t^*_n) + (n/3) \ol{F}_0(t^*_n)/\alpha  }\right)\\
	&\leq \exp\left(-0.5 \frac{n \ol{F}_0^2(t^*_n)/\alpha^2}{ s_n/n +  (2\alpha+1/3)\ol{F}_0(t^*_n)/\alpha  }\right),
\end{align*} 
by using Bernstein's inequality (Lemma~\ref{th:bernstein}) and the fact that $G_n(t)\leq  2\ol{F}_0(t) + s_n/n$ for all $t$. Next, Lemma~\ref{lem:tnstar_general} further tells us that $a_n^*-\delta_n\leq t_n^* \leq a_n^*$ for $n$ large, and $\overline{F}_0(t_n^*)\asymp \alpha s_n/n$ (up to a dependence on $\Lambda_n$).
The latter implies that, for $c>0$,
\begin{equation}\label{eqn:hat_t<t_n^*}
	P_\theta(\hat{t}>t_n^*)\leq e^{-cs_n}.
\end{equation} 
Indeed, inserting the lower bound $\overline{F}_0(t_n^*)\geq \tfrac{1}{3}\alpha (1-\Lambda_n) s_n/n$ in the numerator and denominator, bounding $2\alpha+1/3$ by $7/3$, and bounding $1+(7/9) (1-\Lambda_n)$ by $16/9$, we obtain this with $c=(1/32)(1-\Lambda_n)^2$.

Using symmetry (as in Lemma~\ref{lem:assumption-implies-symmetry}) we thus have 
\begin{align}\label{eqn:FNRofBH}
	\FNR(\theta,\vphi^{BH}_\alpha) &\leq  s_n^{-1} E_\theta\brackets[\bigg]{\sum_{i\in S_\theta} \II\braces{\abs{X_i}< t^*_n} }+ e^{- c s_n } \\ \nonumber &\leq s_n^{-1}\sum_{i\in S_\theta} \brackets[\big]{F_{\abs{\theta_i}}(t_n^*)-F_{\abs{\theta_i}}(-t_n^*)} + e^{-cs_n} \leq  s_n^{-1}\sum_{i\in S_\theta} F_{\abs{\theta_i}}(t_n^*) + e^{-cs_n}.
\end{align} 
By monotonicity of ${F}_{\abs{\theta_i}}$, using that $t_n^*\leq a_n^*$, we have
\[ \FNR(\theta,\vphi^{BH}_\alpha) \leq \Lambda_n(\theta) + e^{-cs_n},\]
which concludes the proof of the pointwise bound. For the uniform bound, simply note that for $\theta\in\Theta(\ba,s_n)$ we have $\Lambda_n(\theta)\leq \Lambda_n(\ba)$ to obtain the result in the case where $\Lambda_n(\ba)$ is bounded away from 1. If there is a subsequence $n_j$ along which $\Lambda_{n_j}(\ba)\to 1$, the bound \eqref{FDRBH} together with $\FNR(\theta,\vphi)\leq 1$ for any $\theta,\vphi$ yields $\mathfrak{R}(\theta,\vphi^{BH}_\alpha)\leq 1+\alpha_n= \Lambda_{n_j}(\ba)+o(1)$. 

The proof under Assumption~\ref{ass:generalnoise}\ref{ass:scale} is identical up to the bound \eqref{eqn:FNRofBH}. [It is perhaps worth noting that in this setting is is more natural to define $t_n^*$ relative to $G_n(t)= (1-s_n/n)2\overline{F}_0(t)+n^{-1}\sum_{i\in S_\theta}2\overline{F}_{\theta_i}(t)$ itself rather than relative to the lower bound $(1-s_n/n)2\overline{F}_0(t)+n^{-1}\sum_{i\in S_\theta}\overline{F}_{\abs{\theta_i}}(t)$, but the latter bound remains valid and yields valid deductions.] The first inequality of \eqref{eqn:FNRofBH} holds in this setting, and the result follows from 
$E_\theta \II\braces{\abs{X_i}<t_n^*}= 1-2\overline{F}_{\theta_i}(t_n^*)$, monotonicity of this expression and the fact that $t_n^*\leq a_n^*$.

\subsection{~~A classification risk bound for BH}\label{proofBHclassif}
Consider the  BH procedure $\vphi^{BH}_\alpha$ \eqref{defBH}, for $\alpha=\alpha_n = o(1)$ with $-\log \alpha = o((\log (n/s_n))^{1/2})$. 
Recalling the definition \eqref{equLC} of the classification loss $L_C$, let us prove that for any fixed real $b$ and any $\eta>0$,
\begin{equation}\label{equ:classifproba}
	\sup_{\te\in\Theta_b} P_\te\left[\lc(\te,\vphi^{BH}_\alpha) \ge (\overline\Phi(b)+\eta)s_n
	\right]
	=o(1), 
\end{equation}
and that the same holds with $\overline{\Phi}(b)$ replaced by 0 when $b=b_n\to+\infty$ or by 1 when $b=b_n\to-\infty$.
This entails the BH part of  Theorem~\ref{thm-adapt-clpr} (to come in Section~\ref{sec:classification-in-prob}) under polynomial sparsity \eqref{polspa}, because 
$-\log \alpha = o((\log n)^{1/2})$ implies 
$-\log \alpha = o((\log (n/s_n))^{1/2})$ in that case.

Fix $\te\in\ell_0[s_n]$ and
denote by $V=\sum_{i=1}^n \ind{\te_i=0}\ind{\vphi_i^{BH}\neq 0}$ the number of false discoveries of $\vphi^{BH}$ and by $W=\sum_{i=1}^n \ind{\te_i\neq 0}\ind{\vphi_i^{BH} = 0}$ the number of false non-discoveries of $\vphi^{BH}$.
Assume first $b\in \R$. We have
\begin{align*}
	P_\te\left[\lc(\te,\vphi^{BH}_\alpha) \ge (\overline\Phi(b)+\eta)s_n
	\right]
	\leq P_\te\left[V \ge \eta s_n/2
	\right] + P_\te\left[W \ge s_n \overline\Phi(b)+\eta s_n/2
	\right].
\end{align*}
We have by Markov's inequality and Lemma~\ref{lem:BHNR},
\begin{align*}
	P_\te\left[V \ge \eta s_n/2\right] \leq (\eta/2)^{-1}\frac{E_\te V}{s_n} \leq (\eta/2)^{-1} \left(\frac{\alpha }{1-\alpha} +  \frac{\alpha}{s_n(1-\alpha)^2}\right),
\end{align*}
hence $\sup_{\te\in\Theta_b} P_\te\left[V \ge \eta s_n/2
\right] = o(1)$.
On the other hand, we have by using \eqref{eqn:hat_t<t_n^*} (note that we use $\Lambda_n=\overline{\Phi}(b)<1$ here)
\begin{align*}
	P_\te\left[W \ge s_n \overline\Phi(b)+\eta s_n/2\right] &= P_\te \left[\sum_{i\in S_\theta}\ind{\abs{X_i}< \hat{t}}\geq s_n \overline\Phi(b)+\eta s_n/2\right]\\
	&\leq P_\te \left[\sum_{i\in S_\theta}\ind{|X_i|< t_n^*}\geq s_n \overline\Phi(b)+\eta s_n/2\right] + e^{-c s_n},
\end{align*}
where $t_n^*>0$ is defined as in Lemma~\ref{lem:tnstar_general}. Using also Lemma~\ref{lem:assumptionok} we have in the Gaussian setting that $t_n^*=a_n^*+o(1)$ with $a_n^*=\sqrt{2\log(n/s_n)}$.  
Using the bound \eqref{eqn:FNRofBH},  the random variable $\sum_{i\in S_\theta}\ind{\abs{X_i}< t_n^*}$ is stochastically upper-bounded by a Binomial distribution with parameter $s_n$ and $q_n:=1-\ol{\Phi}(t_n^*-a_n^*-b)=\ol{\Phi}(b)+o(1)$. Applying Bernstein's inequality (Lemma~\ref{th:bernstein}), we obtain
\begin{align*}
	&P_\te \left[\sum_{i\in S_\theta}\ind{\abs{X_i}< t_n^*}\geq s_n \overline\Phi(b)+\eta s_n/2\right] \\
	&\leq  P_\te \left[\mathcal{B}(s_n, q_n) \geq s_n \overline\Phi(b)+\eta s_n/2\right] \\
	&=P_\te \left[\mathcal{B}(s_n, q_n) - s_nq_n\geq s_n (\overline\Phi(b)-q_n) +\eta s_n /2\right] \\
	&\leq \exp\left( -0.5s_n \frac{(\overline\Phi(b)-q_n +\eta  /2)^2}{q_n+( \overline\Phi(b)-q_n +\eta /2)/3 }\right) \leq e^{-c' s_n},
\end{align*}
for $n$ large enough and some constant $c'>0$ only depending on $\eta$ (using the bounds $q_n\leq 1$ and $\overline{\Phi}(b)-q_n=o(1)$). This gives \eqref{equ:classifproba} for $b\in \R$.

When $b=b_n\to+\infty$ (so $\overline\Phi(b)\to 0$) and $\eta \in (0,1)$, the same techniques as above lead to
\[
P_\te\left[\lc(\te,\vphi^{BH}_\alpha) \ge \eta s_n\right]\leq (\eta/2)^{-1} \left(\frac{\alpha }{1-\alpha} +  \frac{\alpha}{s_n(1-\alpha)^2}\right) + e^{-c s_n} + P_\te \left[\mathcal{B}(s_n, q_n) \geq \eta s_n/2\right],
\] 
and this again shows \eqref{equ:classifproba}  (with $\overline\Phi(b)$ replaced by $0$) because 
\[
P_\te \left[\mathcal{B}(s_n, q_n) \geq \eta s_n/2\right]\leq  2q_n/\eta = o(1)
\]
by Markov's inequality.

Finally, when $b=b_n\to-\infty$ (so $\overline\Phi(b)\to 1$), we directly use $W\leq s_n$, so that
\begin{align*}
	P_\te\left[\lc(\te,\vphi^{BH}_\alpha) \ge (1+\eta)s_n
	\right]
	\leq P_\te\left[V \ge \eta s_n
	\right]\leq \eta^{-1}\left(\frac{\alpha }{1-\alpha} +  \frac{\alpha}{s_n(1-\alpha)^2}\right) ,\end{align*} 
by using again Markov's inequality and Lemma~\ref{lem:BHNR}. This shows \eqref{equ:classifproba} (with $\overline\Phi(b)$ replaced by $1$) in that case.

\subsection{~~Useful lemmas} \label{sec:BHuseful-lemmas}

\begin{lemma}\label{lem:tnstar_general}
In the setting of Theorem~\ref{thm:BHpointwise} with $\theta\in \R^n$ with $|S_\theta|=s_n\geq 1$, and $a_n^*,\delta_n$ as in Assumption~\ref{ass:generalnoise}, define 	\[\Psi_\theta: t \mapsto s_n^{-1}\sum_{i\in S_\theta}\ol{F}_{\abs{\theta_i}}(t)/\ol{F}_0(t).\] Then $\Psi_\theta$ is continuous increasing on $[0,\infty)$ ($n,\theta$ being kept fixed), with $1\leq\Psi_\theta(0)\leq 2$.
If $\alpha=\alpha_n\leq 1$ satisfies $3(n/s_n)\overline{F}_0(a_n^*)/(1-\Lambda_n)\leq \alpha_n \leq \min\braces{1,(n/s_n)\overline{F}_0(a_n^*-\delta_n)}$ for $n$ large enough (that is, \eqref{equcondalpha} holds), then, for $n$ large enough, there exists a unique point $t^*_n>0$ such that 
\begin{equation}\label{eqn:tn*}
2(1-s_n/n) + (s_n/n) \Psi_\theta(t_n^*) =3/\alpha, \end{equation}
and we have
\[ a_n^*-\delta_n\leq t_n^*\leq a_n^*, 
\]
and $\tfrac{1}{3}(1-\Lambda_n)\alpha_n s_n/n\leq \overline{F}_0(t_n^*)\leq \alpha_n s_n/n$.

%
\end{lemma}

\begin{proof}

The continuous increasing part is a direct consequence of Lemma~\ref{lem:monotoneLR}. 
For   
 $1\leq \Psi_\theta(0)\leq 2$, note that under Assumption~\ref{ass:generalnoise} necessarily $\overline{F}_0(0)= 1/2$  and $\overline{F}_0(0)\leq \overline{F}_{\abs{\theta_i}}(0)\leq 1$ for any $\theta_i$ (under Assumption~\ref{ass:generalnoise}\ref{ass:scale} we in fact have $\overline{F}_{\abs{\theta_i}}(0)=\overline{F}_0(0)=1/2$ for all $\theta_i$).

Existence of $t_n^*$ will implicitly follow from the proceeding calculations and the intermediate value theorem. To verify the asymptotics of $t_n^*$, observe that by assumption there exist $c_n\to 0,C_n\to \infty$ such that $\overline{F}_0(a_n^*)=c_ns_n/n,$ $\overline{F}_0(a_n^*-\delta_n)=C_n s_n/n$.
Then
\[ \Psi_\theta(a_n^*) = (n/s_n)c_n^{-1} s_n^{-1} \sum_{i\in S_\theta} \overline{F}_{\abs{\theta_i}}(a_n^*)= (n/s_n)c_n^{-1} (1-\Lambda_n(\theta)).\] From the assumed lower bound on $\alpha_n$ we have that $\alpha_n\geq 3c_n/(1-\Lambda_n)$, hence the right side in the above display is at least $(3/\alpha_n) (n/s_n) \geq 2+ (3/\alpha_n-2)(n/s_n)=\Psi_\theta(t_n^*)$, so that $t_n^*\leq a_n^*$ by monotonicity. 
Similarly,
\[ \Psi_\theta (a_n^*-\delta_n) = (n/s_n)C_n^{-1} s_n^{-1}\sum_{i\in S_\theta} \overline{F}_{\abs{\theta_i}} (a_n^*-\delta_n) 
\leq (n/s_n)C_n^{-1},\] so that, since $\alpha_n\leq \min\braces{1,C_n}$, we obtain $\Psi_\theta(a_n^*-\delta_n)\leq 2+(3/\alpha_n-2)(n/s_n)= \Psi_\theta(t_n^*)$, so that $t_n^*\geq a_n^*-\delta_n$.
Finally, from the definition of $t_n^*$ we have
\[ s_n^{-1} \sum_{i\in S_\theta} \overline{F}_{\abs{\theta_i}}(t_n^*) = (2+ (3/\alpha_n - 2)(n/s_n))\overline{F}_0(t_n^*).\]
The left side is upper bounded by 1 and, using that $t_n^*\leq a_n^*$, lower bounded by $1-\Lambda_n$; the claimed explicit bounds on $\overline{F}_0(t_n^*)$ follow using that $(1/\alpha_n)(n/s_n)\leq 2+(3/\alpha_n-2)(n/s_n)\leq (3/\alpha_n)(n/s_n)$. 
\end{proof}

\begin{lemma}\label{lem:BHNR}
For all $\te\in\R^n$, considering the BH procedure $\vphi^{BH}_\alpha$ \eqref{defBH} at some level $\alpha\in (0,1)$, then $V=\sum_{i=1}^n \ind{\te_i=0}\II\braces{\vphi_i^{BH}\neq 0}$  satisfies 
\begin{equation}\label{equ-EVBH}
E_\te V\leq \frac{\alpha }{1-\alpha}\sum_{i=1}^n \II\braces{\te_i\neq 0} +  \frac{\alpha}{(1-\alpha)^2}.
\end{equation}
\end{lemma}

\begin{proof}
To prove \eqref{equ-EVBH}, we follow the proof of Proposition~7.2 in \cite{neuvialroquain12} (or alternatively, that of Lemma~7.1 in \cite{Bogdan2011}). Recall that we have 
\begin{align*}
E_\te V &= E_\te \sum_{i=1}^n \II\braces{\te_i=0}\II\braces{\abs{X_i}\geq \hat{t}}
\end{align*}
with $\hat{t} = \max\{t\in \R\cup\braces{\infty}\::\: \hat{G}_n(t)\geq 2\ol{F}_0(t)/\alpha\}$ and $\hat{G}_n(t) = n^{-1}\sum_{i=1}^n \II\braces{\abs{X_i}\geq t}$. Now consider the BH procedure applied to the $X_i$'s where we have plugged $X_i=+\infty$ on $S_\theta=\{i\::\: \theta_i\neq 0\}$, that is, with threshold $\hat{t}^0 = \max\braces{t\in \R\cup\braces{\infty }\::\: \abs{S_\theta}/n + n^{-1}\sum_{i\notin S_\theta} \II\braces{\abs{X_i}\geq t} \geq 2\ol{F}_0(t)/\alpha}$. 
Clearly, we have $\hat{t}^0\leq \hat{t}$, so that
\begin{align*}
E_\te V &\leq E_\te \sum_{i=1}^n \II\braces{\te_i=0}\II\braces{\abs{X_i}\geq \hat{t}^0}.
\end{align*}
Denote $n_0=\abs{S_\theta^c}$ for short.
From classical multiple testing theory (see, e.g., Lemma~7.1 in \cite{RV2011}), the latter is the expected number of rejections of the step-up procedure with critical values $(\alpha (k+\abs{S_\theta})/n)_{1\leq k\leq n_0}$ and restricted to the $p$-value set $\braces{2\ol{F}_0(X_i), i\notin S_\theta}$. Now from Lemma~4.2 in \cite{FR2002} (applied with ``$n=n_0$'', ``$\beta=\alpha$'' and ``$\ta=\alpha/n$''), the latter equals
\begin{align*}
&\alpha \frac{n_0}{n} \sum_{i=0}^{n_0-1} {n_0-1 \choose i} (\abs{S_\theta}+i+1) i! (\alpha/n)^i\\
&\leq \alpha \sum_{i\geq 0}  (\abs{S_\theta}+i+1) \alpha^i = \frac{\alpha}{1-\alpha}\abs{S_\theta} + \alpha /(1-\alpha)^2.
\end{align*}
\end{proof}

\section{~~An empirical Bayes multiple testing procedure}\label{sec:l-vals-adaptive}
In this section we define concretely a Bayesian `$\ell$-value procedure' and show that it achieves the minimax risk adaptively, proving 
(a part of) Theorems~\ref{thm-adapt-r} and \ref{thm-adapt-clpr}. We in fact prove a generalisation of the $\ell$-value part of Theorem~\ref{thm-adapt-r} to the multiple signals setting of Section~\ref{sec:mainmult}, though remaining in the Gaussian sequence setting. Results in this section can be seen as generalisations of Lemmas 5, 7 and 9 in \cite{acr21} and corresponding earlier results in \cite{cr20}. 

\subsection{~~Definitions}\label{sec:l-vals-definitions}

For $w\in(0,1)$, let $\Pi_w=\Pi_{w,\gamma}$ denote a spike and slab prior for $\theta$, where, for $\mathcal{G}$ a distribution with density $\gamma$,
\begin{equation}
	\label{eqn:def:SpikeAndSlabPrior}
	\Pi_w = ((1-w)\delta_0+w\mathcal{G})^{\otimes n}.
\end{equation} That is, the coordinates of a draw $\theta'$ from $\Pi_w$ are independent, and are either exactly equal to $0$, with probability $(1-w)$, or are drawn from the `slab' density $\gamma$. When the Bayesian model holds, the data $X$ follows a mixture distribution, with each coordinate $X_i$ independently having density $(1-w)\phi + w g$, where $g$ denotes the convolution $\phi \star \gamma$. [In keeping with the rest of this paper, and in contrast to many papers on empirical Bayesian procedures including \cite{acr21,cr20}, we will reserve $\theta$ for the ``true'' parameter under which we analyse the performance of procedures $\vphi$, and so we use the notation $\theta'$ to denote a draw from the prior/posterior.] We consider a `quasi-Cauchy' alternative as in \cite{js05}, where $\gamma$ is defined implicitly such that 
\begin{equation} \label{eqn:def:gInQuasiCauchy} g(x)=(2\pi)^{-1/2}x^{-2}(1-e^{-x^2/2}),\:\:\: x\in\RR.
\end{equation}

The posterior distribution $\Pi_w(\cdot \mid X)$ can be explicitly derived hence, taking an empirical Bayes approach, one may estimate $\hat{w}$ by maximising the log-likelihood
\begin{align} \label{eqn:def:wHat}
	\hat{w}&=\argmax_{w\in [1/n,1]} L(w), \\
	\label{eqn:def:L}
	L(w)&= \sum_{i=1}^n \log \phi(X_i)+ \sum_{i=1}^n \log(1+w \beta(X_i)), \quad \beta(x):=\tfrac{g}{\phi}(x)-1.\end{align} Note that a maximiser can be seen to exist under the current assumptions by taking derivatives (see also Lemma~\ref{lem:monotonicity}).

For an arbitrary level $t\in (0,1)$, we consider the multiple testing procedure given by thresholding the posterior probabilities of coming from the null (also known as `$\ell$-values'),
\begin{align} 
	\label{ellvalp}
	\vphi^{\hat{\ell}}(X)&= (\ind{\ell_{i,\hat{w}}(X)<t})_{i\le n}, \\ \label{eqn:def:ellvals} \ell_{i,w}(X)&=\Pi_w(\theta'_i = 0 \mid X) = \frac{(1-w)\phi(X_i)}{(1-w)\phi(X_i)+wg(X_i)}.
\end{align}

Let us gather various other definitions as in \cite{cr20} and \cite{js04}. Useful properties of these quantities are given in Section~\ref{sec:background-for-lval}.
We denote $r(w,t)=wt/\{ (1-w)(1-t)\}$ and set
\begin{align}
	\label{eqn:def:betaxw}	\beta(x,w)&=\frac{\beta(x)}{1+w\beta(x)}, \quad x\in\RR,~w\in(0,1)\\
	\label{eqn:def:tildem} \tilde{m}(w)&= 
	-E_{\theta_1=0} \sqbrackets{\beta(X_1,w)}, \quad w\in(0,1) \\ \label{eqn:def:m_1}
	m_1(\tau,w)&
	=E_{\theta_{1}=\tau} \sqbrackets{ \beta(X_1,w)},
	\quad \tau\in\RR,~w\in(0,1)\\
	\label{eqn:def:m2} m_2(\tau,w)&=E_{\theta_1=\tau} (\beta(X_1,w)^2), \quad \tau\in\RR,~w\in(0,1)\\
	\label{eqn:def:xi}
	\xi(u)&= (\phi/g)^{-1}, \quad u\in (0,(\phi/g)(0))\\
	\label{eqn:def:zeta} \zeta(u)&= \beta^{-1}(1/u),\quad u\in(0,1).
\end{align}

\subsection{~~Main $\ell$-value result}\label{sec:proof-lvals}
We prove the following. Recall the definition \eqref{parametersetmultiscale} of the `multiple levels' signal set $\Theta(\ba,s_n)$. Recall also the definition \eqref{eqn:LambdaN} of the measure of the signal in the class $\Theta(\ba,s_n)$, which here evaluates to, with $a_n^*=\sqrt{2\log(n/s_n)}$,
\[  \Lambda_n(\ba) = s_n^{-1}\sum_{j=1}^{s_n} \overline{\Phi}(a_j-a_n^*).\]
\begin{theorem}\label{thm:lval-multilevel}
In the additive Gaussian sparse sequence model \eqref{model}, assume $\Lambda_n(\ba)$ is bounded away from 1. Then the $\ell$-value procedure \eqref{ellvalp} taken at any fixed threshold $t\in(0,1)$ achieves the risk bound
\[ \sup_{\theta \in \Theta(\ba,s_n)} \mathfrak{R}(\theta,\vphi^{\hat{\ell}}) = \Lambda_n(\ba) + o(1).\]
\end{theorem}
\begin{proof}
The lower bound is given by Theorem~\ref{th:minimax-multilevel}, and it remains to prove that 
\[ \sup_{\theta\in\Theta(\ba,s_n)} \frak{R}(\theta,\vphi^{\hat{\ell}})\leq \Lambda_n(\ba)+o(1).\]
\paragraph*{Step 1 (concentration of $\hat{w}$)}
We construct $w_-\leq w_+$ satisfying $w_-\asymp w_+\asymp (s_n/n)(\log (n/s_n))^{1/2}$ such that for $\theta\in\Theta(\ba,s_n)$, 
\[ P_{\theta}(\hat{w}\not \in (w_-,w_+))=o(1).\]
For a constant $\nu\in(0,1/2)$ we let $w_-,w_+$ be the (almost surely unique) solutions to 
\begin{align}\label{eqn:def:w-} \sum_{i\in S_\theta} m_1(\theta_{i},w_-)=(1+\nu)(n-s_n)\tilde{m}(w_-), \\
	\label{eqn:def:w+} \sum_{i\in S_\theta} m_1(\theta_{i},w_+)=(1-\nu)(n-s_n)\tilde{m}(w_+),
\end{align}
whose existence (for $n$ large) and asymptotics are yielded by Lemma~\ref{lem:existence-of-w+-}. By Lemma~\ref{lem:concentration-of-hat{w}}, for a constant $c>0$ we have
\[ P_{\theta}(\hat{w}\in(w_-,w_+))\geq 1-e^{-cs_n}=1-o(1).\]

\paragraph*{Step 2 (FNR control)}
Note that $\ell_{i,w}$ monotonically decreases as $w$ increases (see Lemma~\ref{lem:monotonicity}), so that on the event $\hat{w}\in(w_-,w_+)$ we have
$\ell_{i,w_+}\leq \ell_{i,\hat{w}}\leq \ell_{i,w_-}$.

We have (see Lemma~\ref{lemxigen})
\[ \xi(u)\leq (2\log(1/u)+2\log\log(1/u)+6\log2)^{1/2}.\]
The right side is decreasing in $u$ hence, using that $w_-\geq s_n/n$ for $n$ large enough, we have for any $t\in(0,1)$, a constant $c=c(t)$ and a sequence $\eta_n=\eta_n(t)\to 0$ 
\[ \xi(tw_-/2)\leq \sqrt{2\log(n/s_n)+2\log\log(n/s_n)+c}=\sqrt{2\log(n/s_n)}+\eta_n,\] the latter equality coming from Taylor expanding.
We deduce, using that $\xi(u)=(\phi/g)^{-1}(u)$ is decreasing (Lemma~\ref{lem:monotonicity}), that if $\abs{X_i}\geq \sqrt{2\log(n/s_n)}+\eta_n$ then \[\ell_{i,w_-}=\brackets[\Big]{1+\frac{w_-}{1+w_-}\frac{g}{\phi}(\abs{X_i})}^{-1} < \brackets[\Big]{ \frac{w_-}{2\xi^{-1}(\abs{X_i})}}^{-1}\leq t,\] 
hence, using that $\overline{\Phi}$ is Lipschitz, there exists a sequence $\xi_n\to 0$ not depending on $\theta$ such that, with $\eps_i\iidsim \cN(0,1)$, \[P_\theta(\ell_{i,w_-}\geq t)\leq P(-\eps_i\geq \theta_i - \sqrt{2\log (n/s_n)}-\eta_n) \leq \overline{\Phi}\brackets[\big]{\theta_i-\sqrt{2\log(n/s_n)}}+\xi_n.\]

We deduce, writing $S=\#\braces{i \in S_\theta : \ell_{i,\hat{w}}<t}$ and appealing to Lemma~\ref{lem:stodg}, that we may upper bound $s_n-S$ on the event $\hat{w}\geq w_-$ by a variable $N$ following a Poisson binomial distribution with parameter vector $\bm{p}=(\overline{\Phi}(a_j-a_n^*)+\xi_n)_{j\leq s_n}$. 
For some sequence $\xi_n'\to 0$ we see that the false negative rate is then
\begin{equation}\label{eqn:FNR-control-lval} \FNR(\theta,\vphi^{\hat{\ell}}) =s_n^{-1} E_\theta[s_n-S] \leq P_{\theta}(\hat{w}<w_-)+s_n^{-1}E[N]\leq  \Lambda_n(\ba)+\xi_n'.\end{equation}

A concentration argument further yields that on an event of probability tending to 1,
\begin{equation}\label{equusefullval}
	S \geq s_n(1-\Lambda_n(\ba))/2,
\end{equation}
which will be used in the next step. 
Indeed observe, using the assumption that $1-\Lambda_n$ is bounded away from 0, that for $n$ large \[EN = s_n (\Lambda_n(\ba)+\xi_n) \leq s_n \Lambda_n + (1/4)s_n(1-\Lambda_n).\]
Using this bound and Bernstein's inequality (Lemma~\ref{th:bernstein}) yields 
\[ P\brackets[\big]{N> s_n\Lambda_n + \tfrac{1}{2} s_n(1-\Lambda_n)}\leq P\brackets[\big]{N-EN > \tfrac{1}{4} s_n(1-\Lambda_n)}\to 0,\] 
hence
\begin{align*} P_\theta\brackets[\big]{S< s_n(1-\Lambda_n)/2} &= P_\theta\brackets[\big]{s_n-S > s_n\Lambda_n + s_n(1-\Lambda_n)/2} \\ &\leq P\brackets[\big]{N>s_n\Lambda_n+\tfrac{1}{2} s_n(1-\Lambda_n)}+ P_\theta(\hat{w}< w_-)\to 0. \end{align*}

\paragraph*{Step 3 (FDR control)}
Let $V$ denote the number of false positives,
\[ V = \#\braces{i\not \in S_\theta : \ell_{i,\hat{w}}<t}.\]
By \eqref{equusefullval}, let $\mathcal{A}$ be an event of probability tending to 1 on which $S\geq s_n(1-\Lambda_n)/2$ and $\hat{w}<w_+$. Define
\[ V' = \#\braces{i\not \in S_\theta : \ell_{i,w_+}<t}.\]
Using monotonicity and applying Jensen's inequality to the convex function $x\mapsto x/(a+x)$ we obtain
\begin{multline} \FDR(\theta,\vphi^{\hat{\ell}}) = E_\theta \sqbrackets[\Big]{\frac{V}{(V+S)\vee 1}} \leq E_\theta\sqbrackets[\Big]{ \frac{V}{V+S}\II_{\mathcal{A}}} + P_\theta(\mathcal{A}^c) \\ \leq  E_\theta \sqbrackets[\Big]{\frac{V'}{V'+s_n(1-\Lambda_n)/2}}+o(1)\leq \frac{E_\theta V'}{E_\theta V'+s_n(1-\Lambda_n)/2}+o(1).\end{multline}
Next, Lemma~\ref{lem:EVboundlemma} yields
\begin{equation} \label{eqn:EV'bound} E_\theta V' = (n-s_n)P_\theta (\ell_{i,w_+}\leq t)\leq 2(n-s_n) r(w_+,t)\xi(r(w_+,t))^{-3},\end{equation}
where $r(w_+,t)=w_+t(1-w_+)^{-1}(1-t)^{-1}\asymp w_+$. We note that $\xi(u)\sim (2\log(1/u))^{1/2}$ as $u\to 0$ by Lemma~\ref{lemxigen}, so that using $w_+\asymp (s_n/n)(\log n/s_n)^{1/2}$, we have
$E_\theta V'\lesssim s_n \log(n/s_n)^{-1} =o(s_n)$.
Since we are assuming $\Lambda_n$ is bounded away from 1, this yields the bound
\[ \FDR(\theta,\vphi^{\hat{\ell}})\leq \frac{o(s_n)}{o(s_n)+s_n(1-\Lambda_n)/2}+o(1)=o(1).\] Combined with the bound \eqref{eqn:FNR-control-lval} on the false negative rate, this concludes the proof.
\end{proof}

\begin{lemma}\label{lem:existence-of-w+-}
In the setting of Theorem~\ref{thm:lval-multilevel}, for $n$ large there exist solutions $w_-\leq w_+$ to \eqref{eqn:def:w-} and \eqref{eqn:def:w+} respectively. Moreover, these solutions are almost surely unique and satisfy
	\[w_\pm \asymp s_n (n-s_n)^{-1} \tilde{m}(w_\pm)^{-1} \asymp s_n(n-s_n)^{-1}(\log(n/s_n))^{1/2}\asymp (s_n/n) (\log n/s_n)^{1/2}. \] 
\end{lemma}

\begin{proof}
	We follow the proof of \cite[Lemma 5]{acr21}, adapting to allow for the weaker and mixed signals considered here, and taking advantage of not targeting a rate of convergence to simplify some aspects. 
	We claim that, for some constants $c,C>0$,
	\begin{align} 
		\label{eqn:sn<w-}\sum_{i \in S_\theta} m_1(\theta_{i},c(s_n/n)(\log n/s_n)^{1/2})&>(1+ \nu)(n-s_n)\tilde{m}(c(s_n/n)(\log(n/s_n))^{1/2}) \\
		\label{eqn:snlogsn>w+}\sum_{i \in S_\theta} m_1\brackets[\big]{\theta_{i},C(s_n/n)(\log n/s_n)^{1/2}}&<(1- \nu)(n-s_n)\tilde{m}\brackets[\big]{C(s_n/n)(\log n/s_n)^{1/2}},
	\end{align}
	at least for $n$ large enough.
	It will follow by the intermediate value theorem that there exist unique $w_-,w_+$ solving \eqref{eqn:def:w-} and \eqref{eqn:def:w+} respectively, and satisfying $(s_n/n)(\log n/s_n)^{1/2}\lesssim w_-\leq w_+\lesssim (s_n/n)(\log n/s_n)^{1/2}$, since $\tilde{m}$ is continuous, increasing and non-negative and $m_1(\tau,\cdot)$ is continuous and decreasing for each fixed $\tau$ (see Lemma~\ref{m1binf}).
	
	To prove the claim, note that also by Lemma~\ref{m1binf} we have, for some $c_0,C_0>0$, any $\mu\in\RR$ and asymptotically as $w\to 0$, 
	\begin{alignat*}{3} c_0(\log (1/w))^{-1/2}&\leq \tilde{m}(w) &&\leq C_0 (\log (1/w))^{-1/2}, \\
		&\phantom{\leq} m_1(\mu,w)&&\leq 1/w.
	\end{alignat*}
	It follows that \begin{align*} \sum_{i\in S_\theta} m_1(\theta_{i},C(s_n/n) (\log n/s_n)^{1/2}) &\leq C^{-1}n (\log (n/s_n))^{-1/2},\\ 
		(1-\nu)(n-s_n)\tilde{m}(C(s_n/n)(\log (n/s_n))^{1/2})&\gtrsim (n-s_n) (\log (n/s_n))^{-1/2},\end{align*} where the suppressed constant can be chosen independently of $C>0$ and $\nu\in(0,1/2)$, for $n$ larger than some $N=N(C)$. The inequality \eqref{eqn:snlogsn>w+} follows, for $C$ large enough.
	
	For the lower bound on \eqref{eqn:sn<w-} we observe that Lemma~\ref{m1binf} further yields, for some constants $\omega_0\in(0,1),M_0,c_1>0$ and all $w\leq \omega_0$, $\mu\geq M_0$, 
	\[ m_1(\mu,w)\geq c_1\frac{\overline{\Phi}(\zeta(w)-\mu)}{w}T_\mu(w),\]
	where $T_\mu(w)$ is a function bounded below by 1. Recall that $\zeta$ is a decreasing function satisfying $\zeta(w)\leq (2\log(1/w) + 2\log\log(1/w)+C)^{1/2}$ for a constant $C>0$ (see Lemmas~\ref{lem:monotonicity} and \ref{lemzetagen}). In particular, note by a Taylor expansion that $\zeta(c(s_n/n)(\log(n/s_n))^{1/2})\leq \sqrt{2 \log(n/s_n)}+o(1)$ and hence, using also that $\overline{\Phi}$ is Lipschitz, for some $C_1>0$ and some $o(1)$ not depending on $a_j$ or $c$ we have for $n$ large and $a_j\geq M_0$
	\[ m_1(a_j,c(s_n/n)(\log(n/s_n))^{1/2})\geq C_1c^{-1} (n/s_n)(\log(n/s_n))^{-1/2}(\overline{\Phi}(\sqrt{2\log(n/s_n)}-a_j)-o(1)).\] 
	Observe also, recalling the definition \eqref{eqn:def:betaxw}, that $\beta(x,w)\geq - \abs{\beta(0)}/(1-\abs{\beta(0)})$ is lower bounded by a constant, hence the same is true of $m_1(\tau,w)$ for all $\tau$ and $w$ (including $\tau<M_0$).
	Let $\sigma$ denote a bijection taking $i\in S_\theta$ to $\sigma(i)=j\in \braces{1,\dots,s_n}$ such that $\abs{\theta_{i}}\geq a_j$, and recall that $a_n^*=\sqrt{2\log(n/s_n)}$ in the current Gaussian setting.
	Noting that $\overline{\Phi}(a_n^*-a_j)=o(1)$ uniformly in $j$ such that $a_j\leq M_0$, we deduce that for some constants $C_2,c_2>0$ not depending on $c$ 
	\begin{align*}
		&\sum_{i \in S_\theta} m_1(\theta_{i},\frac{cs_n(\log (n/s_n))^{1/2}}{n}) \\ =& \sum_{i\in S_\theta : a_{\sigma(i)}<M_0 } m_1(\theta_{i},\frac{cs_n(\log (n/s_n))^{1/2}}{n}) + \sum_{i\in S_\theta : a_{\sigma(i)}\geq M_0} m_1(\theta_{i},\frac{cs_n(\log (n/s_n))^{1/2})}{n} \\
		\geq& -C_2s_n + \frac{c_2n}{cs_n(\log( n/s_n))^{1/2}}\sum_{j\leq s_n: a_j\geq M_0} (\overline{\Phi}(a_n^*-a_j)-o(1)) \\
		\geq& -C_2s_n + \frac{c_2n}{cs_n(\log (n/s_n))^{1/2} }\brackets[\big]{\sum_{j\leq s_n} (\overline{\Phi}(a_n^*-a_j)-o(1))} 
		\\ 
		&= -C_2s_n + \frac{c_2n}{c(\log n/s_n)^{1/2}}(1-\Lambda_n(\ba)-o(1)).
	\end{align*}
	Note that $s_n=o(n)$ implies $s_n= o(n/(\log(n/s_n))^{1/2})$. 
	Since $\Lambda_n(\ba)$ is bounded away from 1 we deduce that for $n$ large the left side of \eqref{eqn:sn<w-} is lower bounded by a constant not depending on $c$ multiplied by $c^{-1}n(\log n/s_n)^{-1/2}$, while the right side, using the earlier upper bound on $\tilde{m}(w)$, is upper bounded by a constant not depending on $c$ multiplied by $n\log(n/s_n)^{-1/2}$. Taking $c$ small enough yields the claim.
\end{proof}

\begin{lemma}\label{lem:concentration-of-hat{w}}
	Under the assumptions of Lemma~\ref{lem:existence-of-w+-}, define $\hat{w},$ $w_-$ and $w_+$ as in \eqref{eqn:def:wHat},\eqref{eqn:def:w-} and \eqref{eqn:def:w+} respectively. Then there exists $c>0$ such that 
	\[\sup_{\theta\in\Theta_{\bb}} P_{\theta}\brackets[\big]{\hat{w}\not \in(w_-,w_+)}\leq e^{-cs_n}.\]
\end{lemma}

\begin{proof}
	We adapt the proof of \cite[Lemma S-4]{cr20} or \cite[Lemma 7]{acr21} to the current setting with no polynomial sparsity and multiple signal levels $a_j$. 
	
	Let us prove, for a constant $c>0$ depending only on an upper bound for $\nu<1$, that
	\[	P_{\theta}( \hat{w}< w_-)  \leq e^{-cs_n}.	\] The proof that $P_{\theta}( \hat{w}> w_+) \leq e^{-cs_n}$ is similar (see also the similar proof below of Lemma~\ref{lem:hat{w}bound-for-small-support}), yielding the claim up to a factor of 2 which can be removed by initially considering a $c'>c$.
	
	Let $S=L'$ be the score function, that is, the derivative of the likelihood $L$ defined in \eqref{eqn:def:L}. Since $\hat{w}$ maximises $L(w)$, necessarily $S(\hat{w})\leq 0$ or $\hat{w}=1$. If $\hat{w}<w_-$ then only the former may hold, so that applying the strictly monotonic function $S$ (Lemma~\ref{lem:monotonicity}) we obtain $\braces{\hat{w}<w_-}=\braces{S(w_-)<S(\hat{w})}\subseteq \braces{S(w_-)<0}$. Hence,
	\begin{align*}
		P_{\theta}( \hat{w}< w_-) \leq  P_{\theta}( S(w_-)<0)	&=P_{\theta}( S(w_-)-E_{\theta}S(w_-) <-E_{\theta}S(w_-))\\
		&=P_{\theta}\left( \sum_{i=1}^n W_i <-E\right),
	\end{align*}
	where we have introduced the notation $W_i=\beta(X_i,w_-)-m_1(\theta_{i},w_-)$ and $E=E_{\theta}S(w_-)=\sum_{i=1}^n m_1(\theta_{i},w_-)$. For $n$ large, $\abs{W_i}\leq \mathcal{M}:=2/w_-$ a.s. (see Lemma~\ref{lem:beta-bound}), so that we may apply Bernstein's inequality (Lemma~\ref{th:bernstein}) and obtain
	\begin{align*}
		P_{\theta}( \hat{w}< w_-) \leq e^{-0.5 E^2/(V_2 + \mathcal{M}E/3)},
	\end{align*}
	where $V_2=\sum_{i=1}^n \Var(W_i)\leq \sum_{i=1}^n m_2(\theta_{i},w)$, for $m_2(\theta_{i},w)=E_{\theta} (\beta(X_i,w)^2) $. 
	In view of the definition \eqref{eqn:def:w-} of $w_-$, we have
	\[
	E= \sum_{i\in S_\theta} m_1(\theta_{i},w_-)- (n-s_n)\tilde{m}(w_-) = \nu(n-s_n) \tilde{m}(w_-).
	\]
	
	We also note, using the bounds on $m_2$ in Lemma~\ref{lem:m2bounds} that for some constants $C,M_0>0$ and $n$ larger than some universal threshold,
	\begin{align*}
		V_2
		&\leq \sum_{i\leq n:|\theta_{i}|>M_0}  m_2(\theta_{i},w_-)+\sum_{i\leq n :\abs{\theta_{i}}\leq M_0}  m_2(\theta_i,w_-)\\
		&\leq \frac{C}{w_-}\sum_{i\in S_\theta: \abs{\theta_i}>M_0 } m_1(\theta_{i},w_-) + C\sum_{i\leq n: \abs{\theta_i}\leq M_0} \frac{\overline{\Phi}(\zeta(w_-)-|\theta_i|)}{w_-^2} \\
		&\leq \frac{C}{w_-}\sum_{i\in S_\theta } m_1(\theta_{i},w_-) -\frac{C}{w_-} \sum_{i \in S_\theta : \abs{\theta_{i}}\leq M_0} m_1(\theta_{i},w_-) + Cs_n \frac{\overline{\Phi}(\zeta(w_-)-M_0)}{w_-^2}+ Cn\frac{\overline{\Phi}(\zeta(w_-))}{w_-^2},
	\end{align*}
	with $\zeta$ defined as in \eqref{eqn:def:zeta}.
	For the first term we use the definition \eqref{eqn:def:w-} of $w_-$, and for the second we use that $m_1(\theta_{i},w_-)$ is bounded below by a (negative) constant. By a standard normal tail bound (included in Lemma~\ref{lem:sub}), the bounds on $w_-$, the fact that $\sqrt{2\log(1/w_-)}\leq \zeta(w_-)\lesssim \sqrt{\log(n/s_n)}$ as $n\to \infty$ (Lemmas~\ref{lem:existence-of-w+-} and \ref{lemzetagen}; note $\phi(\zeta(w_-))\lesssim w_-$ as a consequence of the bound) and that $\tilde{m}(w_-)\asymp \zeta(w_-)^{-1}$ (Lemma~\ref{m1binf}), we deduce for some constant $M_1$ that
	\[ \overline{\Phi}(\zeta(w_-)-M_0) \asymp \frac{\phi(\zeta(w_-)-M_0)}{\zeta(w_-)} \lesssim w_-\tilde{m}(w_-)e^{M_1 \sqrt{\log(n/s_n)}},\]
	so that, bounding $M_1\sqrt{\log(n/s_n)}$ by $\log(n/s_n)$, for $n$ large the third term is upper bounded by a constant multiple of
	\[ s_n w_-^{-1} \tilde{m}(w_-) e^{M_1\sqrt{\log(n/s_n)}} \leq nw_-^{-1}\tilde{m}(w_-).\]  For the fourth term, by the same normal tail bound and the definition of $\zeta$ we have $\overline{\Phi}(\zeta(w_-))\asymp \phi(\zeta(w_-))/\zeta(w_-) \asymp w_- g(\zeta(w_-)) /\zeta(w_-)$, which is of order $w_-(\zeta(w_-))^{-3} \asymp w_-/ \zeta(w_-)^3$, hence of order $w_- \tilde{m}(w_-)/ \zeta(w_-)^2$ because $\tilde{m}(w_-)\asymp \zeta(w_-)^{-1}$.
	We deduce, recalling that $s_n=o(n)$ implies automatically that $s_n\ll n(\log(n/s_n))^{-1/2}\asymp n\tilde{m}(w_-)$,
	\[
	V_2\lesssim  n w_-^{-1}  \tilde{m}(w_-) +w_-^{-1} s_n+ n w_-^{-1} \tilde{m}(w_-)/ \zeta(w_-)^2 \lesssim  n w_-^{-1}  \tilde{m}(w_-),
	\]
	so that 
	\[
	\frac{V_2+ \mathcal{M}E/3}{E^2}\lesssim  \frac{n w_-^{-1}  \tilde{m}(w_-)}{ (\nu(n-s_n) \tilde{m}(w_-))^2} + \frac{1}{ \nu w_- (n-s_n) \tilde{m}(w_-)}\lesssim  \frac{1}{ \nu^2 n w_-\tilde{m}(w_-)} .
	\]
	This implies that
	$
	P_{\theta_0}( \hat{w}< w_-) \leq e^{-c \nu^2 n w_-\tilde{m}(w_-)}
	$
	for some constant $c>0$.  
	Now, by Lemma~\ref{lem:existence-of-w+-}, we have $n w_-\tilde{m}(w_-)\asymp s_n$, yielding the desired bound on the probability.
\end{proof}

We adapt Lemmas~\ref{lem:existence-of-w+-} and \ref{lem:concentration-of-hat{w}} slightly to apply even when $\Lambda_n(\ba)$ is not bounded away from $1$ or when $\abs{S_\theta}<s_n$, in order to accommodate the settings of Theorems~\ref{thm-adapt-clpr} and \ref{thm-adapt-cl}. Recall the definition \eqref{thetaprime} of $\Theta'_b$. 

\begin{lemma}\label{lem:hat{w}bound-for-small-support}
	Consider  the setting of Theorem~\ref{thm-adapt-cl}. 
	For constants $C,D>0$ let $\omega_1=C(\sigma_n/n)(\log(n/\sigma_n))^{1/2}$, where $\sigma_n=\max(s_n,D\log n)$.
	Then if $C$ and $D$ are suitably large, for all $n$ large enough we have
	\begin{equation}
		\label{eqn:hat{w}<omega1}	\sup_{\te\in\Theta_b'}P_{\theta}(\hat{w}>\omega_1)\leq n^{-1}.
	\end{equation} 
\end{lemma} 
\begin{proof}
	We begin with a corresponding upper bound to \eqref{eqn:snlogsn>w+}. Using the bounds, found in Lemma~\ref{m1binf}, that for some $c_0,\omega_0>0$, any $\mu\in \RR$ and any $w<\omega_0$
	\[ \tilde{m}(w)\geq c_0 (\log(1/w))^{-1/2},\quad m_1(\mu,w)\leq 1/w,\] we note that $\tilde{m}(\omega_1)\asymp (\log(1/\omega_1))^{-1/2}\asymp (\log(n/\sigma_n))^{-1/2}$ because $n/\sigma_n\to\infty$, and we further deduce that for  $n$ large enough, for any $\theta\in \Theta'_b$ with support $S_\theta$ of size $s_\theta\leq s_n$ we have
	\begin{equation} \label{eqn:m1omega1-bound}	\sum_{i\in S_\theta} m_1(\theta_i,\omega_1)<(1-\nu)(n-s_\theta)\tilde{m}(\omega_1),
	\end{equation}
	provided the constant $C$ in the definition of $\omega_1$ is large enough. 
	
	We now argue as in proving Lemma~\ref{lem:concentration-of-hat{w}}. Let $S=L'$, for $L$ as in \eqref{eqn:def:L}, denote the score function. Necessarily $S(\hat{w})\geq 0$ or $\hat{w}=0$, and we deduce 
	that $\braces{\hat{w}>\omega_1}\subset\braces{S(\omega_1)>0}$, hence
	\begin{equation*}
		P_{\theta}( \hat{w}> \omega_1) \leq P_{\theta}( S(\omega_1)-E_{\theta}S(\omega_1) >-E_{\theta}S(\omega_1))\\
		=P_{\theta}\left( \sum_{i=1}^n W_i >-E\right),
	\end{equation*}
	where $W_i=\beta(X_i,\omega_1)-m_1(\theta_{i},\omega_1)$ and $E=E_{\theta}S(\omega_1)=\sum_{i=1}^n m_1(\theta_{i},\omega_1)$. For $n$ large $|W_i|\leq \mathcal{M}:=2/\omega_1$ a.s. (see Lemma~\ref{lem:beta-bound}), so that we may apply Bernstein's inequality (Lemma~\ref{th:bernstein}) and obtain
	\begin{align*}
		P_{\theta}( \hat{w}> \omega_1) \leq e^{-0.5 E^2/(V_2 + \mathcal{M}E/3)},
	\end{align*}
	where $V_2=\sum_{i=1}^n \Var(W_i)\leq \sum_{i=1}^nm_2(\theta_{i},\omega_1)$, for $m_2(\theta_{i},w)=E_{\theta} (\beta(X_i,w)^2) $. 
	In view of \eqref{eqn:m1omega1-bound}, 
	\[
	-E=(n-s_\theta)\tilde{m}(\omega_1)- \sum_{i\in S_\theta} m_1(\theta_{i},\omega_1) > \nu(n-s_\theta) \tilde{m}(\omega_1).
	\]
	Arguing as in the proof of Lemma~\ref{lem:concentration-of-hat{w}}, one obtains 
	\[ V_2\lesssim   \omega_1^{-1}\sum_{i\in S_\theta } m_1(\theta_{i},\omega_1)  +\omega_1^{-1} s_n+ n\omega_1^{-1}\tilde{m}(\omega_1)+ n \omega_1^{-1} \tilde{m}(\omega_1)/ \zeta(\omega_1)^2 \lesssim n\omega_1^{-1}\tilde{m}(\omega_1),\] with the last bound following from \eqref{eqn:m1omega1-bound} (note that $s_n\ll n$ implies $s_n\ll n/\sqrt{\log(n/\sigma_n)}$). Inserting the bounds for $E,$ $V_2$ and $\mathcal{M}$ into the obtained bound on $P_\theta(\hat{w}>\omega_1)$ yields for a constant $c=c(\nu)$
	\[ P_\theta(\hat{w}>\omega_1) \leq e^{-cn\omega_1\tilde{m}(\omega_1)}.\] Recalling that $\tilde{m}(\omega_1)
	\asymp (\log(n/\sigma_n))^{-1/2}$ 
	we see for $c'$ not depending on $D$ that $P_\theta(\hat{w}>\omega_1)\leq e^{-c'\sigma_n}$, and \eqref{eqn:hat{w}<omega1} follows upon choosing $D=D(c')$ large enough.
\end{proof}

\subsection{~~A classification risk bound for the $\ell$-value procedure}  \label{sec:proof:thm-adapt-clpr} 
In this section we prove that for any fixed real $b$ and any $\eta>0$
\[\sup_{\te\in\Theta_b} P_\te\left[\lc(\te,\vphi^{\hat{\ell}})/s_n \ge  \overline\Phi(b)+\eta
\right]
=o(1), \]
and that the same holds with $\overline{\Phi}(b)$ replaced by 0 when $b=b_n\to+\infty$ or by 1 when $b=b_n\to-\infty$. This proves an upper bound of Theorem~\ref{thm-adapt-clpr} to come (in Section~\ref{sec:classification-in-prob}). Suppose $\te\in \Theta_b$ and let us write the classification loss $\lc(\te,\vphi^{\hat{\ell}})$  as the sum $V+(s_n-S)$, where as in the proof of Theorem~\ref{thm:lval-multilevel} we write $V=V(\hat{w})=\sum_{i \not \in S_\theta} \vphi^{\hat{\ell}}_i$, $S=S(\hat{w})=\sum_{i\in S_\theta}{\vphi^{\hat{\ell}}_i}$, with $S_\theta=\braces{i:\theta_i\neq 0}$ denoting the support of $\theta$. It suffices to show that uniformly over $\te\in\Theta_b$, for some positive sequences $\xi_n\to 0$ and $\nu_n\to 0$,
\begin{align}
	\label{tr-clfn}
	P_\te\left[s_n-S \ge (\overline\Phi(b)+\xi_n)s_n \right]&=o(1),\\
	\label{tr-clfp}
	P_\te\sqbrackets[\big]{V>\nu_n s_n} &= o(1),
\end{align}
To prove \eqref{tr-clfn}, note that we may apply Theorem~\ref{thm:lval-multilevel} and arguments in the proof thereof with $\ba=(a_n^*+b,\dots,a_n^*+b)$. In the case of a fixed $b\in \RR$, recall that we argued in proving Theorem~\ref{thm:lval-multilevel} (see before \eqref{eqn:FNR-control-lval}) that 
	on the event $\hat{w}\geq w_-$, we have
	$s_n-S \leq N$, for $N$ a Poisson binomial $N$ with parameter vector $\bm{p}=(\overline{\Phi}(a_j-a_n^*)+\xi_n)_{j\leq s_n}$ and for a sequence $\xi_n\to 0$; for the current choice of $\ba$ we see that $N\sim \text{Bin}(s_n,\overline{\Phi}(b)+\xi_n)$.
Then \eqref{tr-clfn} follows by recalling that $P(\hat{w}<w_-)\to 0$ from Lemma~\ref{lem:concentration-of-hat{w}}, and by applying Bernstein's inequality (Lemma~\ref{th:bernstein}). In the case $b=b_n\to +\infty$ the same argument applies upon replacing $\overline{\Phi}(b)$ by $0=\overline{\Phi}(b)+o(1)$. 
In the case $b=b_n\to-\infty$, since $S\geq 0$ we automatically have
\[ P_{\theta} \sqbrackets[\big]{s_n-S \geq (1+\delta_n)s_n}=0,\]
so that \eqref{tr-clfn} holds in all cases (with suitable substitutions).

To prove \eqref{tr-clfp}, one could similarly use the bound \eqref{eqn:EV'bound}. Here, to allow for the case $b=b_n\to-\infty$, we adapt this bound by appealing to Lemma~\ref{lem:hat{w}bound-for-small-support}, which yields the existence of constants $C,D$ such that for large $n$
\[ P_\theta(\hat{w}>\omega_1)\leq n^{-1}, \quad \omega_1=C\frac{\sigma_n}{n}\brackets[\big]{\log(n/\sigma_n)}^{1/2},~ \sigma_n=\max(s_n,D\log n).\]
Then, writing 
\[ V' = \#\braces{i \not \in S_\theta : \ell_{i,\omega_1}<t},\] we apply Lemmas~\ref{lem:EVboundlemma} and \ref{lemxigen} as in proving Theorem~\ref{thm:lval-multilevel} to deduce that \[ E_\theta V' \lesssim n\omega_1 \log(1/\omega_1)^{-3/2} \lesssim \sigma_n/ \log (n/\sigma_n),\] where we have used that $n/\sigma_n\to \infty$ to see that $\log(1/\omega_1)\asymp \log(n/\sigma_n)$. 	When $\sigma_n=D\log(n)$ we note that $\sigma_n/\log(n/\sigma_n)$ is upper bounded by a constant.
For some $C'>0$ we therefore have 
\begin{equation} \label{eqn:EVbound-classification} E_\theta V \leq n P_\theta(\hat{w}>\omega_1)+E_\theta V'\leq 1+E_\theta V'\leq C'\max\brackets[\Big]{1,\frac{s_n}{\log(n/s_n)}}=o(s_n),
\end{equation}
where we have used that $s_n\to\infty$ and $n/s_n\to\infty$. 
Applying Markov's inequality we deduce $P_\theta[V>\nu_ns_n]\leq \nu_n^{-1}s_n^{-1}E_\theta V$ tends to zero if $\nu_n$ tends to zero sufficiently slowly. 
This proves \eqref{tr-clfp}, including  the cases $b=b_n\to\pm\infty$, and thus concludes the proof.
\begin{remark}
	\label{rem:l-val-sparsity-preserving}
	A straightforward corollary of the above proof shows that the $\ell$-value procedure is sparsity preserving. It suffices to show that $P(V+S>2s_n)=o(1)$, or that $P(V>s_n)=o(1)$, which follows from the proof, taking $\nu_n=1$.
\end{remark}

\begin{remark}\label{sec:proof:thm-adapt-cl}
	Let us provide an alternative proof of the $\ell$-value upper bound part of Theorem~\ref{thm-adapt-cl}, similar in spirit to the just obtained in-probability bounds. Write $S_\theta=\braces{i : \theta_i\neq 0}$, $s_\theta = \abs{S_\theta}$, 
	$V=V(\hat{w})=\sum_{i\not \in S_\theta} \vphi^{\hat{\ell}}_i$ and $S=S(\hat{w})=\sum_{i \in S_\theta} \vphi^{\hat{\ell}}_i$, and note as in the previous proof of Theorem~\ref{thm-adapt-cl} (in Section~\ref{sec:proof-of-thm-adapt-cl}) that it suffices to show 
	\[ \sup_{\theta \in \Theta_b'(s_n)} E_\theta V =o(s_n),\quad \sup_{\theta\in \Theta_b'(s_n)} E_\theta [s_\theta- S] \leq s_n\overline{\Phi}(b)+o(s_n).\] 
	The first of these follows from taking a supremum in \eqref{eqn:EVbound-classification}, whose proof we note applies for $\theta\in\Theta_b'(s_n)$ not just $\theta\in\Theta_b(s_n)$.
	
	The second is trivial in the case $b=b_n\to -\infty$ since $0\leq s_\theta-S\leq s_\theta\leq s_n$ for $\theta\in\Theta_b'(s_n)$. In the cases $b\in\RR$ fixed and $b=b_n\to+\infty$, note firstly an examination of the proofs of Lemmas~\ref{lem:existence-of-w+-} and \ref{lem:concentration-of-hat{w}} reveals that for some constants $c,c'>0$, for any $\theta\in\Theta_b(s)$ with $s\leq s_n$ we have
	\[ P_\theta(\hat{w}<w_-)\leq e^{-cs}, \quad w_-:= c'(s/n)(\log(n/s))^{1/2}. \] In this setting the bound \eqref{eqn:FNR-control-lval} reads that, for some sequence $\delta_n\to 0$ which can be chosen independently of  $s$, 
	\begin{align*} E_\theta[s-S] & \le sP_{\theta}(\hat{w}<w_-)+s(\overline{\Phi}(b)+\delta_n) \\
		& \le  s e^{-cs}+s(\overline{\Phi}(b)+\delta_n) \le C(c)+s_n\overline{\Phi}(b)+o(s_n).
	\end{align*}
	The bound at stake then follows by taking the maximum over $0\le s\le s_n$ in the last display. 
	
\end{remark}

\subsection{~~Background material for the $\ell$-value procedure}\label{sec:background-for-lval}
We gather results from \cite{cr20} which are used in the proofs for this section. Some of these results were originally formulated with dependence on $g$ and a related parameter $\kappa\in[1,2]$; as in \cite{acr21}, we simplify such expressions here by substituting the explicit form \eqref{eqn:def:gInQuasiCauchy} for $g$, which has $\kappa=2$, and using the bounds $\sup_x \abs{g(x)}\leq 1/\sqrt{2\pi}$ and, for $\abs{x}\geq 2$, $x^{-2}/(2\sqrt{2\pi})\leq g(x)\leq x^{-2}/\sqrt{2\pi}$.
\begin{lemma}\label{lem:monotonicity}
	The following functions are strictly decreasing (with probability 1 in the case of random functions).
	\begin{align*}
		w&\mapsto S(w)=L'(w), \\
		w&\mapsto \ell_{i,w}(X),\\
		w&\mapsto -\tilde{m}(w),\\
		w&\mapsto m_1(\tau,w), \quad \tau\in\RR \text{ fixed,} \\
		u&\mapsto \xi(u)=(\phi/g)^{-1}(u),\\
		u&\mapsto \zeta(u).
	\end{align*}
\end{lemma}
These monotonicity results can be found in \cite{cr20}, and see \cite[Lemma 4]{acr21} for most proofs collected in one place.

\begin{lemma}[Proposition 3 in \cite{cr20}] \label{lem:EVboundlemma}
	$P_{\theta}(\ell_{i,w}\leq t) \leq 2 r(w,t)\xi(r(w,t))^{-3}$, where $r(w,t)=wt(1-w)^{-1}(1-t)^{-1}$ and $\xi$ is as in \eqref{eqn:def:xi}.
\end{lemma}

\begin{lemma}[Lemma S-20 in \cite{cr20}] \label{lem:beta-bound} Define $\beta(x,w)$ as in \eqref{eqn:def:betaxw}. Then there exists $c_1>0$ such that for any $x\in\RR$ and $w\in (0,1]$, $\abs{\beta(x,w)}\leq (\min(w,c_1))^{-1}$.
\end{lemma}

\begin{lemma}[Lemmas S-21, S-23 and S-27 in \cite{cr20}, and using Lemma~\ref{lemzetagen} below] \label{m1binf}
	Define $\tilde{m}$ and $m_1$ as in \eqref{eqn:def:tildem}, \eqref{eqn:def:m_1}. Then $\tilde{m}$ is continuous, non-negative and increasing. For fixed $\tau$ the function $w\mapsto m_1(\tau,w)$ is continuous and decreasing.
	There exist $\omega_0,c,c'>0$ such that for all $\tau\in\RR$ and all $w<\omega_0$ 
	\begin{alignat*}{3} c(\log (1/w))^{-1/2}&\leq \tilde{m}(w) &&\leq c' (\log (1/w))^{-1/2}, \\
		&\phantom{\leq} m_1(\tau,w)&&\leq 1/w.
	\end{alignat*}
	
	There exist constants $M_0, C_1>0$ and $\omega_0\in(0,1)$ such that for any $w\le \omega_0$, and any $\mu \ge M_0$, with $T_\mu(w)=1+\abs{\mu}^{-1}\abs{\zeta(w)-\abs{\mu}}$,
	\begin{align*} 
		m_1(\mu,w) & \ge
		C_1 \frac{\ol{\Phi}(\zeta(w)-\mu)}{w}T_\mu(w). 
	\end{align*}
\end{lemma}

\begin{lemma}[Lemma S-26 and Corollary S-28 in \cite{cr20}]
	\label{lem:m2bounds}
	Define $m_2$ as in \eqref{eqn:def:m2}. 
	There exist constants $C > 0$ and $\omega_0\in (0,1)$ such that for
	any $w \leq \omega_0$ and any $\mu\in\RR$,
	\[ m_2 (\mu,w)\leq  C \frac{ \overline{\Phi}(\zeta(w)-\abs{\mu})}{w^{2}}.\]
	There exist $M_0,C' > 0$ and $\omega_0 \in (0,1)$
	such that for any $w \leq \omega_0$ and any $\mu\geq M_0$
	\[	m_2 (\mu,w) \leq C' \frac{m_1 (\mu,w)}{w}.\]
\end{lemma}
\begin{lemma}[Lemma S-14 in \cite{cr20}] \label{lemzetagen}
	Consider $\zeta(w)$ as in \eqref{eqn:def:zeta}. 
	Then $\zeta(w)\sim (2\log(1/w))^{1/2}$ as $w\to 0$. More precisely, for  constants $c,C\in\mathbb{R}$ and for $w$ small enough,
	\[(2\log (1/w)+2\log \log(1/w)+c)^{1/2}\leq \zeta(w)\leq (2\log (1/w)+2\log \log(1/w)+C)^{1/2}.\]
\end{lemma}

\begin{lemma}[Lemma S-12 in \cite{cr20}] \label{lemxigen}
	Consider $\xi$ as in \eqref{eqn:def:xi}. 
	Then $\xi(u)\sim (2\log (1/u))^{1/2}$, and more precisely, for $u$ small enough, \begin{align*}
		\xi(u) &\geq \brackets[\Big]{2\log (1/u) + 2\log\log(1/u)+2\log 2}^{1/2} \\
		\xi(u) &\leq \brackets[\Big]{ 2\log(1/u) + 2\log\log (1/u) +6\log 2}^{1/2}.
	\end{align*}
\end{lemma}

The above bounds come from Lemma S-12 in \cite{cr20} 
\begin{align*} 
	\xi(u)&\geq \left(-2\log u-2\log g\left(\sqrt{-2  \log ( C u)}\right)-\log(2\pi)\right)^{1/2} ;\\
	\xi(u)&\leq  \left(-2\log u-2\log g\left(\sqrt{-4 \log u}\right)-\log(2\pi)\right)^{1/2},
\end{align*}
where $C=\sqrt{2 \pi}\sup_x\abs{g(x)}$,
by using $x^{-2}/(2\sqrt{2\pi})\leq g(x)\leq x^{-2}/(\sqrt{2\pi})$ for $\abs{x}\geq 2$.
For instance, for the lower bound we use 
\begin{align*}
	-2\log g\left(\sqrt{-2  \log ( C u)}\right)-\log(2\pi)&\geq 2\log(\sqrt{2\pi}) +2 \log(-2  \log ( C u)) -\log(2\pi)\\
	&\geq 2 \log(2) +2 \log( \log (1/u) + \log(1/C))\\
	&\geq 2\log 2 + 2 \log( \log (1/ u)),
\end{align*}
where the last inequality holds because $C\leq 1$.

\section{~~Large signals: adaptation and proofs}\label{sec:prooffastrate}

In this section, we first discuss adaptation in the large signal regime and provide the precise statements corresponding to Theorem~\ref{thmlargeadapt}. Then we turn to the proofs of both non-adaptive and adaptive results.

Recall some notation for the Subbotin location model (Example~\ref{example:SubbotinLocation}). For $\zeta>1$, the parameter set  $\Theta(r,\be) $ is defined by \eqref{classab}--\eqref{emer}, with  $M=M(r)=(\zeta r\log{n})^{1/\zeta}$ and sparsity parameter $a\le s_n/n^{1-\be}\le b$ (which we may sometimes informally write $s_n\asymp n^{1-\be}$),  for $r>\beta$ and $a,b>0$. Define
\begin{equation} \label{defvr}
	v_R= n^{-\kappa}/(\log n)^{1-1/\zeta}=v_C/n^{1-\be},
\end{equation}
where $\kappa=\kappa(r,\be,\zeta)$ is the solution of the equation
\begin{equation} \label{eqkappa}
	(r^{1/\zeta} -\kappa^{1/\zeta})^\zeta-\kappa= \beta.
\end{equation}
Recall from Theorem \ref{thmRabinovic} that $\inf_\vphi\sup_{\te\in\Theta(r,\be)} \fr(\te,\vphi) \asymp v_R$ so that $v_R= v_R(r,\be,\zeta)$ is the target rate for the $\fr$--risk; similarly
 $v_C=v_C(r,\be,\zeta)=n^{1-\be} v_R$ is the target rate for the classification loss.  

\subsection{~~Impossibility of full simultaneous adaptation to $s_n$ and $r$} \label{sec:adaptfull}

The following is a negative result showing that it is not possible to fully adapt to both signal strength $r$ and sparsity parameter $s_n$ (or $\be$). 

\begin{theorem}[Impossibility of $(r,\be)$--adaptation]
	\label{thm-noadapt}
		Consider the Subbotin location model (Example~\ref{example:SubbotinLocation}) for some $\zeta>1$ and classes $\Theta(r,\be)$ as in \eqref{classab}--\eqref{emer} for $\be\in (0,1)$ and $r>\be$.  Suppose $r_1, r_2>0$ verify
		\[ r_1>r_2 \quad\text{and}\quad \kappa(r_2,\be,\zeta)>1-\be, \]
		for $\kappa$ as in \eqref{eqkappa}. 
		Then there exists $c>0$ (depending on $r_1,r_2,\beta,\zeta$)   such that, for $v_C, v_R$ as in \eqref{defvr},
		\[   
			\,  \inf_\vphi \max_{i\in\{1,2\}} \sup_{\te\in \Theta(r_i,\be)}\,
			\frac{\fr(\te,\vphi)}{v_R(r_i,\be,\zeta)} \ge c \{ v_C(\be,r_1,\zeta)^{-1}n^{- r_2} \} \wedge v_C(\be,r_2,\zeta)^{-1}. \]
			On the two--class adaptation problem $\te\in  \Theta(r_1,\be) \cup  \Theta(r_2,\be)$, the best possible rate incurs a polynomial loss compared to the non-adaptive rate whenever $\kappa(\be,r_1,\zeta) - r_2> 1-\be$. The same result holds for the 
			classification risk, 
			with $\fr/v_R$ in the last display replaced by $E_\te\lc/v_C$. 
		\end{theorem}
The proof of Theorem \ref{thm-noadapt} can be found in Section \ref{sec:thmnoa}. 
		 
		To fix ideas, let us discuss the case of Gaussian noise $\zeta=2$. The regime of $r$'s for which $\kappa(r,\be,2)\ge 1-\be$ can then explicitly be written as the range of $r$'s with $\sqrt{r}>1+\sqrt{1-\be}$, as follows from examining the expression for $\kappa$ in the Gaussian case. This corresponds to the regime of `exact recovery' (for which the classification risk goes to $0$). In this regime, the target rate (either $E_\te \lc$ or $\fr$) is very fast, so even only one missclassification can have a large impact on the risk. And indeed, the idea behind the lower bound is to build two vectors identical except on one coordinate: when the difference between signal strengths on this coordinate is large enough, then the smallest one can be mistaken for noise with a probability that becomes comparable to or higher than the target rate (which goes to $0$ in that regime), hence the impossibility.   
		
		This provides a simple example where there is a polynomial loss due to adaptation. This striking result appears to be in part due to the discreteness of the considered risks: note for instance that considering the problem in probability instead of in expectation, the impossibility disappears, as in probability $P_\te(\lc\ge \delta_n)=P_\te(\lc\ge 1-\delta)$ for any small $\delta$ and $\delta_n=o(1)$. For versions in probability for the large signal regime we refer to Section~\ref{sec:prlarge}.
		
		If one restricts the range of possible signal strengths to $\kappa(r,\be,\zeta)< 1-\be$, then adaptation to both $\be$ and $r$ becomes possible, see Section~\ref{sec:adaptsim}.

		\subsection{~~Adaptation: known sparsity, unknown $r$ and top--$K$ procedures} \label{sec:adaptr}

		In this section, we assume that the number of {\em nonzero} coefficients of $\te$ is {\em known}. Recall 
		\begin{equation} \label{classeq}  
			\Theta_{=}(r,s_n) = \{ \te\in\ell_0[s_n]:\ |S_\te|= s_n,\ \ |\te_i|\ge M(r)  \text{ for all } i\in S_\te \},   
		\end{equation}
		which differs from $\Theta(r,\be)$ by the fact that support of $\te$ has cardinality {\em exactly} $s_n$ (instead of of order $n^{1-\be}$). It is reasonable (in fact necessary) to assume this, since it follows from the proof of Theorem \ref{thm-noadapt} that when $r$ is unknown, full adaptation to signal strength is impossible as soon as the class contains vectors of supports differing in cardinality by just $1$. 
		
		For $K\in \{1,\ldots,n\}$, the top--$K$ procedure is defined as $\vphi_i{[K]}=\ind{|X_i| \ge |X|_{(K)}}$: it rejects the coordinates corresponding to the $K$ largest signals (since the noise distribution has a density, there are almost surely no ties among observations and so it is uniquely defined with probability $1$). 
		
		\begin{theorem} \label{thmtops}
			For any $1\le s_n\le n$, the top--$s_n$ procedure $\vphi_i{[s_n]}=\ind{|X_i| \ge |X|_{(s_n)}}$ verifies, for any $\te\in \ell_0[s_n]\setminus \ell_0[s_n-1]$,
			\[ E_{\te} \lc(\te,\vphi[s_n]) \le 2 \inf_{T\in\cT}  E_{\te} \lc(\te,T), \]
			where $\cT$ is the set of all thresholding procedures with deterministic threshold. 
		\end{theorem}
		
		Theorem~\ref{thmtops} states that the top--$s_n$ procedure is within a factor $2$ of the rate of the best (deterministic--) threshold-based procedure. Since Theorem~\ref{thmRabinovic} shows that the procedure that thresholds deterministically at $t_n^*$ given by \eqref{optthresholdRab} achieves the rate $v_R$ over $\Theta(r,\be)$, the following corollary immediately follows.
		\begin{corollary}
			Consider the Subbotin location model (Example~\ref{example:SubbotinLocation}) for some $\zeta>1$ and let  $\Theta_=(r,s_n) $ be as in \eqref{classeq}, for $M=M(r)=(\zeta r\log{n})^{1/\zeta}$ and $s_n\asymp n^{1-\be}$  for $r>\beta$. 
			There exists a constant $c=c(r,\be,\zeta)$ such that, for  large enough $n$ and $v_C, v_R$ as in \eqref{defvr}, 
			\begin{align*}
				\sup_{\te\in \Theta_=(r,s_n)}  E_{\te} \lc(\vphi,\te) & \le c v_C(r,\be,\zeta), \\
				\sup_{\te\in \Theta_=(r,s_n)} E_{\te} \fr(\vphi,\te) & \le c v_R(r,\be,\zeta).
			\end{align*}  
		\end{corollary}
		\begin{remark}
			Theorem~\ref{thmRabinovic} is stated for the class $\Theta(r,\be)$. Although in principle the lower bound could be faster for the smaller class $\Theta_=(r,s_n)$, we expect that both rates coincide. 
			In any case, the optimal rate for {\em thresholding--based procedures} (with possibly random threshold) over $\Theta_=(r,s_n)$ is $v_R$ (this follow by similar arguments as in \cite{rabinovich20}, Corollary 1, where the authors consider a slightly different, quasi--Subbotin noise). This shows that adaptation is achievable at least for arbitrary thresholding procedures, and otherwise shows that the rate (at least) $v_R$ can be achieved adaptively. 
					\end{remark}
		
		We now present a result on top-$K$ procedures which is concerned with lower signal strength  and is a corollary of both Theorem~\ref{thm-notrade} on no-tradeoff and Theorem~\ref{thmtops}. It shows that the oracle top--$s_n$ procedure, despite being rate-optimal for the large signal regime, is sub-optimal for weaker signals. This suboptimality comes from the fact that its risk splits equally among false positive and false negatives, which implies the result once one notes  the procedure is  sparsity-preserving, hence illustrating the versatility of Theorem~\ref{thm-notrade} (and more generally Theorem \ref{thm-notrade-Subbotin}). 

		\begin{corollary}[Exact risk of the top $s_n$--procedure in the boundary case] \label{cor-topsub} 
		In the setting of Theorem~\ref{thm-notrade}, 
			consider the class of signals $\Theta_b=\Theta(a_b,s_n)$ for some fixed real number $b$. For $\vphi$ the top $s_n$--procedure, as $n\to\infty$,
			\[
			\sup_{\te\in \Theta_b} E_{\te} \lc(\vphi,\te)/s_n \sim \sup_{\te\in \Theta_b}  \fr(\vphi,\te)  \sim 2 \bar\Phi(b).
			\]
			In particular, the top $s_n$--procedure is rate-optimal but not sharp minimax optimal over $\Theta_b$. A similar result holds in Subbotin noise of index $\zeta>1$, with $\bar\Phi(b)$ replaced by $\bar\Phi_\zeta(b)$.
		\end{corollary}

		\subsection{~~Adaptation in Gaussian noise: unknown sparsity, known $r$} \label{sec:adapts}
		
		Let us consider the class $\Theta(r,\be)$ as in \eqref{classab}, assuming that the minimal signal strength $r>\be$ is {\em known}. Then adaptation is possible:  the idea is to use a plug-in procedure where $\beta$ is estimated first. 
In this and the next Section \ref{sec:adaptsim} we focus on the case of Gaussian noise for simplicity. 
Note that in Gaussian noise the equation \eqref{eqkappa} defining $\kappa$ is solved as
\[ \sqrt{\kappa(r,\be)}  =(\sqrt{r}-\be/\sqrt{r})/2. \] 	
In Gaussian noise, the optimal threshold $t_n^*$ in \eqref{optthresholdRab} writes, for $\be\in(0,1)$ and $r>\be$,
\begin{equation}
\ta(r,\be)  = \left(\sqrt{r}-\sqrt{\kappa(r,\be)}\right)\sqrt{2\log{n}}. \label{taurb}
\end{equation}
If we have an estimator $\hat\be$ of $\be$,  for $r$ known  $\ta(r,\hat\be)=(\sqrt{r}/2+\hat{\be}/(2\sqrt{r}))\sqrt{2\log{n}}$ is an estimator of the optimal  threshold.

		
		\begin{theorem}[Known $r$] \label{thmknownr}
		Consider the Gaussian location model (Example~\ref{ex:GaussianSequence}) and the class $\Theta(r,\be)$ as in \eqref{classab}--\eqref{emer} for $\be\in (0,1)$. Let $r>0$ be given.  			 There exists an estimator $\hat{\beta}$ such that, if one sets
			\[ \vphi_i^r = 1_{\{|X_i| \ge  (\sqrt{r}/2+\hat{\be}/(2\sqrt{r}))\sqrt{2\log{n}}\}}. \]
		then for any $\veps>0$, there exists a constant $C=C(\veps)>0$ such that for $n$ large enough, for any $\be\in[\veps, \min(r,1)- \veps]$,
	\begin{align*}
	 \sup_{\te\in\Theta(r,\be)} E_\te \lc(\vphi^r,\te) & \le C v_C(r,\be), \\
	 \sup_{\te\in\Theta(r,\be)} E_\te \fr(\vphi^r,\te) & \le C v_R(r,\be).
	\end{align*} 
		\end{theorem}
		The role of $\veps>0$ in the statement of Theorem \ref{thmknownr} is to ensure a minimal separation from the different regime of weaker signals as in Sections \ref{sec:main}--\ref{sec:mainmult}. 
The proof of Theorem \ref{thmknownr} is given in Section \ref{sec:proofknownr}. 				
		\subsection{~~Adaptation: simultaneous adaptation in regime $\kappa<1-\be$} \label{sec:adaptsim}
		
The next result shows that adaptation is possible in a range of pairs $(\be,r)$ that is arbitrarily close to those satisfying the constraints $\kappa(r,\be,2)<1-\be$ -- in Gaussian noise the latter writes $\sqrt{r}<1+\sqrt{1-\be}$ --  and $\ka>\be$, in order to be in the `large signal' regime. Similarly as for adaptation for known $r$, we introduce the set, for $\veps>0$ arbitrarily small,
\begin{equation} \label{setj}
 \cJ = \cJ(\veps)=\left\{ (\be,r):\ \, \be\in[\veps,1-\veps],\ \, \be+\veps\le r\le (1+\sqrt{1-\be})^2-\veps) \right\}. 
\end{equation}
Below we will use that for any admissibly $r$ in the above set, $r<4$. 
	\begin{theorem}[Adaptation over the range 
	$\sqrt{\be}< \sqrt{r}< 1+\sqrt{1-\be}$] 
			\label{thmadapt}
			Consider the Gaussian location model (Example~\ref{ex:GaussianSequence}) and the class $\Theta(r,\be)$ as in \eqref{classab}--\eqref{emer}. Let $\veps>0$. 
			There exists a procedure $\vphi=\vphi(X)$ such that, for $\cJ=\cJ(\veps)$ as in \eqref{setj},  there exists a finite $C=C(a,b,\veps)>0$ such that, for $n$ large enough,
			\[ \sup_{(\be,r)\in \cJ}\ \sup_{\te\in\Theta(r,\be)} \frac{E_{\te}\lc(\te,\vphi)}{v_C(r,\be)}
			\le C,
			\]
and the same holds for the $\fr$--risk, upon replacing $E_\te\lc(\te,\vphi)/v_C$ by $\fr(\te,\vphi)/v_R$. 	
		\end{theorem}
The proof of Theorem \ref{thmadapt} is given in Section \ref{sec:proofadapt}.

		\subsection{~~Proof of minimax rate for large signals, Theorem~\ref{thmRabinovic}}  \label{proof:thmRabinovic}
		First, we let, for $r>\beta>0$, 
		\begin{equation}\label{equPsirbeta}
			\Psi(\cdot; r,\beta): u\in (0,r)\mapsto \Psi(u; r,\beta)= (r^{1/\zeta}-u^{1/\zeta})^\zeta-u-\beta.
		\end{equation}
		Clearly, $\Psi$ is continuous decreasing, with $\Psi(0; r,\beta)=r-\beta>0$ and $\Psi(r/2^\zeta; r,\beta)=-\beta<0$. Hence, $\kappa\in (0,r/2^\zeta)$ such that $\Psi(\kappa; r,\beta)=0$ is well defined. 
		
		\paragraph*{Upper bound}
		
		First, if $V^*=\sum_{i\notin S_\theta} \vphi^*$ is the number of false discoveries of $\vphi^*$, we have $V^*\sim \text{Bin}(n-|S_\te|,2\overline{\Phi}_\zeta(t_n^*))$. 
		Hence, for $\te\in  \Theta(r,\be)$,
		\begin{align*}
			E_\te V^* &= 2 (n-|S_\te|) \overline{\Phi}_\zeta(t_n^*) \asymp n e^{-(t_n^*)^\zeta/\zeta}/(t_n^*)^{\zeta-1} \asymp n^{1-\beta-\kappa}/(\log n)^{1-1/\zeta},
		\end{align*}
		because $(t_n^*)^\zeta=\zeta(\beta+\kappa) \log n$  and $\overline{\Phi}_\zeta(x)\asymp e^{-x^\zeta/\zeta}/x^{\zeta-1}$ as $x\to \infty$ by Lemma~\ref{lem:sub}. 
		
		Second, if $S^*=\sum_{i\in S_\theta}\vphi^*_i$ is the number of true discoveries of $\vphi^*$, for $\te\in  \Theta(r,\be)$,
		\begin{align*}
			E_\te S^* &= \sum_{i\in S_\theta} P(|\eps_i+ \theta_i|\ge t_n^*) =  \sum_{i\in S_\theta} \big\{P(\eps_i\ge t_n^*-\theta_i)+P(-\eps_i\ge t_n^*+ \theta_i) \big\}\\
			&\ge |S_\te| \overline{\Phi}_\zeta(t_n^*- M(r)). 
		\end{align*}		
			Hence, for any $\te\in  \Theta(r,\be)$,
		$$\FNR(\te,\vphi) \leq 1-\overline{\Phi}_\zeta(t_n^*- M(r))= \overline{\Phi}_\zeta( M(r)-t_n^*) = \overline{\Phi}_\zeta( (\zeta \kappa \log n)^{1/\zeta}).$$
		Hence, 
		$\FNR(\te,\vphi) \lesssim n^{-\kappa}/(\log n)^{1-1/\zeta},$ because by Lemma~\ref{lem:sub} we have $\overline{\Phi}_\zeta(x)\asymp e^{-x^\zeta/\zeta}/x^{\zeta-1}$ when $x\to \infty$.

		For the FDR term, we need additional concentration arguments.
		Bernstein's inequality (Lemma~\ref{th:bernstein}) gives that for any $t>0$, we have
		\[ P_\te[|S^*-E_\te S^*|\ge t ]\le 2\exp\left\{- t^2/(2|S_\te|(1+o(1))+2t/3)  \right\} \leq 2\exp(-c s_n^{1/2}), \]
		by taking $t=|S_\te|^{3/4}$ and a constant $c>0$ small enough.
		This gives $S^*\geq s_n(1+o(1))$ with probability at least $1-2\exp(-c s_n^{1/2})$.
		We now distinguish between the two following cases.

				\paragraph*{Case: $1-(\beta+\kappa)> 0$}
		
		In that case, Bernstein's inequality (Lemma~\ref{th:bernstein}) provides that $V^*$ satisfies the following concentration property:
		\[ P_\te[|V^*-E_\te V^*|\ge M\sqrt{E_\te V^*} ]\le 2\exp\left\{-c [M^2 \wedge (M\sqrt{E_\te V^*})] \right\}. \]
		Taking $M$ such that $M\sqrt{E_\te V^*}=n^{(1-\beta-\kappa)3/4}$ (i.e. $M\asymp  n^{(1-\beta-\kappa)/4} (\log n)^{(1-1/\zeta)/2} $ which tends to $\infty$ since $\kappa<1-\beta$) gives $V^*\le E_\te V^*+ n^{(1-\beta-\kappa)3/4}$ with probability at least $1- 2 e^{-c n^{(1-\beta-\kappa)/4}}$. 
		Hence, by the above concentration arguments, we have for all $\te\in  \Theta(r,\be)$,
		\begin{align*} 
			\FDR(\te,\vphi^*) &\le \frac{E_\te[V^*]+n^{(1-\beta-\kappa)3/4} }{s_n(1+o(1))-s_n^{3/4}} + 2 e^{-c n^{(1-\beta-\kappa)/4}} + 2e^{-c s_n^{1/2}} \\
			&\le   (1+s_n^{-1/4})( n^{-\kappa}/(\log n)^{1-1/\zeta} + n^{-(3/4)\kappa}) + 2 e^{-c n^{(1-\beta-\kappa)/4}} + 2e^{-c s_n^{1/2}}\\
			&\lesssim n^{-\kappa}/(\log n)^{1-1/\zeta}.
		\end{align*}

		\paragraph*{Case: $1-(\beta+\kappa)\leq 0$}
In this case, we still use the concentration of $S^*$ but circumvent the concentration of $V^*$ by upper bounding the FDR with the following more direct argument: for all $\te\in  \Theta(r,\be)$,
		\begin{align*} 
			\FDR(\te,\vphi^*) &\le E_\te[ V^*/((1+o(1))s_n] + 2e^{-c s_n^{1/2}} \\
			&\lesssim n^{-\kappa}/(\log n)^{1-1/\zeta}.
		\end{align*}

This entails the result in both cases.
		 
		\paragraph*{Lower bound}
		
		We propose to establish the lower bound 
		\begin{equation}\label{lbtobeprovedfastrate}
			\inf_\vphi \sup_{\te\in\Theta(r,\be) }  \mathfrak{R} (\theta,\vphi) \gtrsim n^{-\kappa}/(\log n)^{1-1/\zeta}.
		\end{equation}
		
		For this, we apply Theorem~\ref{thm:bayesLB} 
		with $\rho=1$ and the prior $\pi= ((1-w) \delta_0 + w\delta_{M})^{\otimes n}$, $M=M(r)=(\zeta r\log n)^{1/\zeta} $, $w=\frac{a+b}{2}/n^\beta$. 
		
		\paragraph*{Case: $1-(\beta+\kappa)> 0$}
		
		We apply the bound (i) in Theorem~\ref{thm:bayesLB} for $\rho=1$, $s_n=b n^{1-\beta}$ and some $\eta=\eta_n\to 0$ to be chosen later. 
		This gives, with $M_\rho=M_1$ here,
		$$
		\inf_\vphi \sup_{\te\in\Theta(r,\be) }  \mathfrak{R} (\theta,\vphi) \geq    \frac{ M_\rho}{2b n^{1-\beta}+ M_\rho}  (1-  e^{- c\eta_n^2 M_\rho})- n P_\pi(\norm{\theta}_0> b n^{1-\beta}) - 2 P_\pi(\theta\notin \Theta(r,\be)),
		$$
		for some universal constant $c>0$, where 
		$M_\rho= \sum_{i=1}^n P_\pi[ \theta_i\neq 0,\ell_i(X)> 1/2] ,$ for $ \ell_i(X)=P_\pi(\theta_i=0\:|\: X)$, $1\leq i\leq n$.
		By definition of $\Theta(r,\be)$ and of the prior $\pi$, we have  $P_\pi(\theta\notin \Theta(r,\be))\leq P_\pi(\norm{\theta}_0> bn^{1-\beta}) + P_\pi(\norm{\theta}_0< an^{1-\beta})$.
		Also, we have 
		\begin{align*}
(n+2) P_\pi(\norm{\theta}_0> bn^{1-\beta}) &\leq (n+2) P(\mathcal{B}(n,w)-nw>((b-a)/2) n^{1-\beta})\\
&\leq (n+2) e^{-c' n^{1-\beta} }\leq 1/n^{\kappa+1},
		\end{align*}
		for some constant $c'>0$ and for $n$ large enough.  Similarly, we have $(n+2) P_\pi(\norm{\theta}_0< an^{1-\beta})\leq 1/n^{\kappa+1}$.
		%
		%
		%
		%
		%
		%
		
		Next, we have by a direct computation
		$$
		\ell_i(X)=P_\pi(\theta_i=0\:|\: X) = \frac{(1-w) \phi_\zeta(X_i)}{(1-w) \phi_\zeta(X_i)+w \phi_\zeta(X_i-M)},
		$$
		so that $\ell_i(X)> 1/2$ if and only if $\phi_\zeta(X_i-M)/\phi_\zeta(X_i)< (1-w)/w$, that is if and only if 
		$ |X_i|^\zeta - |X_i-M|^\zeta<  \zeta\log (1/w-1) $.
		Therefore, we have 
		\begin{align*}
			P_\pi[ \ell_i(X)> 1/2\:|\: \theta_i\neq 0] &=  P[ |M+\varepsilon|^\zeta - |\varepsilon|^\zeta<  \zeta\log (1/w-1) ]\\
			&\geq   P[ |M+\varepsilon|^\zeta - |\varepsilon|^\zeta<  \zeta\log (n^\beta/b-1) ],
		\end{align*}
		where $\varepsilon\sim \phi_\zeta$ and by using that $w\leq b/n^\beta$.
		Recalling $\Psi(\cdot; r,\beta)$ defined by \eqref{equPsirbeta}, we obtain
		\begin{align*}
			&P_\pi[ \ell_i(X)> 1/2\:|\: \theta_i\neq 0] \\&\geq   P[ |M+\varepsilon|^\zeta - |\varepsilon|^\zeta<  \zeta\log (n^\be/b-1) ]\\
			&\geq P[\varepsilon\in (-M,0), (M+\varepsilon)^\zeta - (-\varepsilon)^\zeta<  \zeta\log (n^\beta/b-1) ]\\
			&=P[\varepsilon\in (0,M), (M-\varepsilon)^\zeta - \varepsilon^\zeta - \zeta\beta \log n<  \zeta\log (n^\beta/b-1) - \zeta\beta \log n]\\
			&=P\left(\varepsilon\in (0,M),\Psi\left( \frac{\varepsilon^\zeta}{\zeta \log n}; r,\beta\right) < \Psi(\kappa_n) \right),
		\end{align*}
		where $\kappa_n=\Psi^{-1}\left(\frac{\log (n^\beta/b-1)-\beta \log n}{ \log n}\right)=\kappa+O(1/\log n)$.
		By the properties of $\Psi$, we get for $n$ large enough that
		\begin{align*}
			P_\pi[ \ell_i(X)> 1/2\:|\: \theta_i\neq 0] &\geq P\left(M>\varepsilon > (\zeta\kappa_n \log n)^{1/\zeta}\right)\\
			&=\ol{\Phi}_\zeta((\kappa_n \zeta \log n)^{1/\zeta}) -\ol{\Phi}_\zeta((r \zeta \log n)^{1/\zeta}).
		\end{align*}
		Now using Lemma~\ref{lem:sub}, we obtain 
		\begin{align*}
			\ol{\Phi}_\zeta((r \zeta \log n)^{1/\zeta})&\leq \frac{\phi_\zeta((r \zeta \log n)^{1/\zeta})}{(r \zeta \log n)^{1-1/\zeta}}  \lesssim n^{-r}/( \log n)^{1-1/\zeta}\\
			\ol{\Phi}_\zeta((\kappa_n \zeta \log n)^{1/\zeta}) &\geq \zeta^{-1}  \frac{\phi_\zeta(x)}{x^{\zeta-1}} = \zeta^{-1}  n^{-\kappa_n}/(\kappa_n \zeta \log n)^{1-1/\zeta}\gtrsim n^{-\kappa}/( \log n)^{1-1/\zeta}.
		\end{align*}
		Since $\kappa<r$, it follows that
		\begin{equation}\label{eqn:Ppi(l>1/2)bound}
			P_\pi(\ell_i(X)>1/2,\theta_i\neq 0)=wP_\pi(\ell_i(X)>1/2\mid \theta_i\neq 0)\gtrsim n^{-\beta} n^{-\kappa}/(\log  n)^{1-1/\zeta},
		\end{equation}
	and hence
		\begin{align}
		M_\rho&= \sum_{i=1}^n P_\pi[ \theta_i\neq 0,\ell_i(X)> 1/2] 
		=\sum_{i=1}^n w P_\pi[ \ell_i(X)> 1/2\:|\: \theta_i\neq 0] \nonumber\\& \gtrsim n^{1-\beta}  n^{-\kappa}/( \log n)^{1-1/\zeta}.\label{equMrhofastrate}
		\end{align}
		This also gives, because  	$1-(\beta+\kappa)> 0$ by assumption,	
		$$
		e^{- c \eta_n^2 M_\rho} \leq  e^{- c' \eta_n^2 n^{1-(\beta+\kappa)}/(\log n)^{1-1/\zeta}} = o(1),
		$$
		by choosing $\eta_n=1/\log n$. Combining the relations obtained above gives 
		$$
		\frac{ M_\rho}{2 b n^{1-\beta}+ M_\rho}  (1-  e^{- c\eta_n^2 M_\rho})  \gtrsim  n^{-\kappa}/( \log n)^{1-1/\zeta}
		$$
		and shows \eqref{lbtobeprovedfastrate}.
		
		\paragraph*{Case: $1-(\beta+\kappa)\leq 0$}
		
		We apply the bound (ii) in Theorem~\ref{thm:bayesLB} with $\rho=1$ and $s_n=bn^{1-\beta}$, 
		\begin{align*}
			\inf_\vphi \sup_{\te\in\Theta(r,\be) }  \mathfrak{R} (\theta,\vphi)&\geq  \left( \frac{1}{bn^{1-\beta} + 1} \wedge \frac{1  }{bn^{1-\beta}}  \right) m_{\pi}- n P_\pi(\norm{\theta}_0> bn^{1-\beta})- 2 P_\pi(\theta\notin \Theta(r,\be))\\
			&\geq  \frac{ m_{\pi}}{1+bn^{1-\beta}}   - 1/n^{\kappa+1},
		\end{align*}
		for $m_{\pi}=P_\pi(\exists i\in \{1,\dots,n\}\::\: \theta_i\neq 0,\ell_i(X)> 1/2)$, as in the 
		previous case. By using the computation \eqref{eqn:Ppi(l>1/2)bound} from the previous case, we have
		\begin{align*}
			m_{\pi}&=P_\pi(\exists i\in \{1,\dots,n\}\::\: \theta_i\neq 0\, ,\, \ell_i(X)> 1/2)\\
			&=1-(1-P_\pi(\theta_1\neq 0,\ell_1(X)> 1/2))^{n}\\
			&\geq 1-(1- C n^{-\beta} n^{-\kappa}/( \log n)^{1-1/\zeta})^n
		\end{align*}
		for some constant $C>0$. Since $1-(\beta+\kappa)\leq 0$, we have $C n^{-\beta} n^{-\kappa}/( \log n)^{1-1/\zeta}=o(1)$, hence 
		$$
		1-(1- C n^{-\beta} n^{-\kappa}/( \log n)^{1-1/\zeta})^n \sim C n^{1-\beta-\kappa}/( \log n)^{1-1/\zeta},
		$$
		and the result follows.

To completely finish the proof, it remains to prove the same bounds for the classification risk $\lc/n^{1-\be}$. The upper bound is completely analogous since 
$$
\lc(\te,\vphi^*)/n^{1-\be}= E_\te V^*/n^{1-\be} + (|S_\te|-E_\te S^*)/n^{1-\be}\lesssim n^{-\kappa}/(\log n)^{1-1/\zeta}.
$$
As for the lower bound, we apply Theorem~\ref{thm:bayesLB-classif} 
		with the same prior, which gives the correct bound by using \eqref{equMrhofastrate} (which holds beyond the case $1-(\beta+\kappa)> 0$).

\subsection{~~Proof of Theorem~\ref{thm-noadapt} [Thm~\ref{thmlargeadapt}, (i)]} \label{sec:thmnoa}
Set $\Theta_i=\Theta(r_i,\be)$ for $i=1,2$. We start with the result for $\lc$. Set $v_i=v_C(r_i,\be,\zeta), i=1,2$ as a shorthand. 
It is enough to bound from below, for $\te_1\in \Theta_1, \te_2\in \Theta_2$, and any procedure $\vphi$, the quantity 
\[\Sigma:=\Sigma(\te_1,\te_2,\vphi):=\frac{E_{\te_1} \lc(\te_1,\vphi)}{v_1} + \frac{E_{\te_2} \lc(\te_2,\vphi)}{v_2}. \]
Let us $s_{n,1}, s_{n,2}$ be two consecutive integers in $[an^{1-\be},bn^{1-\be}]$, so that $s_{n,1}=s_{n,2}-1$ (they exist for large $n$ since $a<b$). Define
\[  \te_{1,i}=  
\begin{cases}
	\, M(r_1) &\quad \text{if }\ 1\le i\le s_{n,1},\\
	\, 0 & \quad \text{otherwise}.
\end{cases}
\]
Define $\te_2$ to be the same as $\te_1$, except for $i=s_{n,1}+1=s_{n,2}$, for which one sets $\te_{2,s_{n,2}}=M(r_2)$. Using a change of measure, one can write
\begin{align*}
	E_{\te_1} \lc(\te_1,\vphi) 
	& = E_{\te_2} \left[ \lc(\te_1,\vphi(X)) \phi_\zeta(X_{s_{n,2}})/\phi_\zeta(X_{s_{n,2}}-M(r_2)) \right] \\
	& \ge E_{\te_2} \left[ \lc(\te_1,\vphi(X)) 
	\exp\{ - M(r_2)^\zeta/\zeta\}
	\ind{X_{s_{n,2}}\le M(r_2)} \right],
\end{align*}
where we use Lemma~\ref{lem:compazeta} and that $M(r_2)>0$.
From this one deduces that $\Sigma$ is bounded from below by
\begin{align*}
	& \left\{ \left( v_C(r_1,\be)^{-1} e^{- M(r_2)^\zeta/\zeta} \right) \wedge 
	v_C(r_2,\be)^{-1}  \right\} \\
	& \ \ \times E_{\te_2}\left[ \{\lc(\te_1,\vphi(X))+\lc(\te_2,\vphi(X))\}\ind{X_{s_{n,2}}\le M(r_2)} \right]\\
	& \ge \left\{ \left( v_C(r_1,\be)^{-1} e^{- M(r_2)^\zeta/\zeta} \right) \wedge 
	v_C(r_2,\be)^{-1}  \right\} \lc(\te_1,\te_2) P_{\te_2}[X_{s_{n,2}}\le M(r_2)],
\end{align*} 
where we use the triangle inequality for the classification loss $\lc$. Noting that $\lc(\te_1,\te_2)=1$ and $P_{\te_2}[X_{s_{n,2}}\le M(r_2)]=\Phi_\zeta(0)=1/2$, one obtains
\[ 2\Sigma \ge\left( v_C(r_1,\be)^{-1} n^{- r_2} \right) \wedge 
v_C(r_2,\be)^{-1}. \]
The result for the classification risk follows.  

For the $\fr$--risk, let us denote by $L_R(\te,\vphi)=\text{FDP}(\te,\vphi)+\text{FNP}(\te,\vphi)$ the loss incurred by a testing procedure $\vphi$ on a sparse vector $\te$ (the sum of the false discovery proportion and false negative proportion), so that $\fr(\te,\vphi)=E_\te L_R(\te,\vphi)$.  Similarly as above, one bounds 
\[ \Sigma_R(\te_1,\te_2,\vphi):= 
\frac{\fr(\te_1,\vphi)}{v_R(r_1,\be)} + \frac{\fr(\te_2,\vphi)}{v_R(r_2,\be)}
\]
from below by, using the same argument,
\begin{align*} 
	& \left\{ \left( v_R(r_1,\be)^{-1} e^{- M(r_2)^\zeta/\zeta} \right) \wedge 
	v_R(r_2,\be)^{-1}  \right\} \\
	& \ \ \times E_{\te_2}\left[ \{L_R(\te_1,\vphi(X))+L_R(\te_2,\vphi(X))\}\ind{X_{s_{n,2}}\le M(r_2)} \right].
\end{align*}
The vectors $\te_1,\te_2$ satisfy the conditions of Lemma~\ref{lemfdpfnp}, which implies that  $L_R(\te_1,\vphi(X))+L_R(\te_2,\vphi(X))$ is at least $1/(s_{n,1}+2)$,  regardless of $\vphi$. Using that $s_{n,1}\asymp n^{1-\be}$ and $n^{1-\be}v_R(r_i,\be)= v_C(r_i,\be)$, the result follows.

\begin{lemma} \label{lemfdpfnp}
	Let $\te_1, \te_2\in\RR^n$ with respectively $|S_{\te_1}|=s_n$ and $|S_{\te_2}|=s_n+1$ nonzero coefficients, for some $1\le s_n\le n-1$. Suppose $|S_{\te_1}\cap S_{\te_2}|=s_n$. Then for any procedure $\vphi$,
	\[ \FDP(\te_1,\vphi) + \FNP(\te_1,\vphi) +
	\FDP(\te_2,\vphi) + \FNP(\te_2,\vphi) \ge \frac{1}{s_n+2}.
	\]
\end{lemma}
\begin{proof}
	For $\te\in\{\te_1,\te_2\}$ and $\vphi$ arbitrary,  by definition of the false discovery proportion,  bounding from above the number of true discoveries by $|S_\te|$, 
	\[ \text{FDP}(\te,\vphi) \ge \frac{ N_{FP}(\te,\vphi) }{ N_{FP}(\te,\vphi)+|S_{\te}| } 
	\ge  \frac{N_{FP}(\te,\vphi) }{N_{FP}(\te,\vphi)+s_n+1}.  \]
	The last quantity is either $0$ or bounded below by $1/(s_n+2)$, since $x\to x/(x+s_n+1)$ is increasing for $x\ge 0$. Similarly,
	\[ \text{FNP}(\te,\vphi)= \frac{N_{FN}(\te,\vphi) }{|S_{\te}|}\ge \frac{N_{FN}(\te,\vphi) }{s_n+1},   \]
	which is either $0$ or bounded below by $1/(s_n+1)$. Note that
	\[ \frac{N_{FP}(\te_1,\vphi) }{N_{FP}(\te_1,\vphi)+s_n+1} +\frac{N_{FN}(\te_1,\vphi) }{s_n+1}+
	\frac{N_{FP}(\te_2,\vphi) }{N_{FP}(\te_2,\vphi)+s_n+1}+\frac{N_{FN}(\te_2,\vphi) }{s_n+1}
	\]
	cannot be zero. Indeed, if this was the case, this would mean that $N_{FP}(\te_1,\vphi)=N_{FN}(\te_1,\vphi)=0$ that is $S_{\te_1}=S_{\vphi}$ but also  $N_{FP}(\te_2,\vphi)=N_{FN}(\te_2,\vphi)=0$ that is $S_{\te_2}=S_{\vphi}$, so that $S_{\te_1}=S_{\te_2}$ and this is impossible because $S_{\te_1}\neq S_{\te_2}$ by definition. This implies that at least one of the four terms in the last display is nonzero, so the result follows, because individually all terms, if nonzero, are at least $1/(s_n+2)$.
\end{proof}

\subsection{~~Proof of Theorem~\ref{thmtops} [Thm~\ref{thmlargeadapt}, (ii)]}
$\ $\\

{\em Notation $N_{FP}(T,\theta), N_{FN}(T,\theta)$.} For a given procedure $T=T(X)$, in the proofs to follow for clarify we denote by $N_{FP}(T,\theta)$, respectively $N_{FN}(T,\theta)$, the number of false positives of $T$, resp. false negatives, if the true vector is $\theta$. \\

\begin{proof}
We now give the proof of Theorem~\ref{thmtops}. For simplicity denote $\vphi=\vphi[s_n]$. Both numbers $N_{FP}(T,\theta), N_{FN}(T,\theta)$ are random variables taking values on the integers up to $n-s_n$ (resp. $s_n$). Using that for such an integer--valued variable $N$ it holds $E[N]=\sum_{k \ge 1} P[N\ge k]$, it is enough to bound the probabilities $P[N\ge k]$.

Let us denote by $\eta_1,\ldots,\eta_{s_n}$ the random variables $\veps_i$ for $i\in S_\te$ (by this we mean $\eta_i=\veps_{j_i}$ for $j_1<\cdots<j_{s_n}$ the indices in  the support of $\te$) and similarly by $\xi_1,\ldots,\xi_{n-s_n}$ the random variables $\veps_{i}$ for $i\notin S_\te$.

First focusing on false positives, if the top--$s_n$ procedure $\vphi$ makes at least $k\ge 1$ false positives, then among the $|\xi_i|$'s, at least $k$ are at least $|X|_{(s_n)}$. For any 
$\ta>0$, 
\begin{align*}
	P_{\te}[ N_{FP}(\vphi,\te) \ge k ] & 
	\le  P_{\te}[ |\xi|_{(k)} \ge  |X|_{(s_n)} ]  \\
	& \le P_{\te}[ |\xi|_{(k)} \ge  \ta ]  + P_{\te}[ |\xi|_{(k)} \ge  |X|_{(s_n)}\, ,\, 
	|\xi|_{(k)} <  \ta] 
\end{align*}
By independence of $\xi_i$'s, one has $P_{\te}[ |\xi|_{(k)} \ge  \ta ]=P_{\te}[ \text{Bin}(n-s_n,2\bar\Phi(\ta))\ge k]$. On the other hand, if $\ta>|\xi|_{(k)} \ge  |X|_{(s_n)}$, then at least $k$ among the variables $|\te_i+\eta_i|$ are less or equal to $|\xi|_{(k)} < \ta$. Denoting by $p_i(\ta)=p_i(\ta,\te):=P_\te[|X_i|\le \ta]$ for $i\in S_\te$ and renumbering $p_1^*(\ta),\ldots ,p_{s_n}^*(\ta)$ the successive values of the $p_i(\ta)$'s, 
\begin{align*}
	P_{\te}[ N_{FP}(\vphi,\te) \ge k ] & 
	\le  P_{\te}[ \text{Bin}(n-s_n,2\bar \Phi(\ta))\ge k] 
	+ P_{\te}[ \text{PBin}(s_n,(p_j^*(\ta)_j))\ge k]. 
\end{align*}
Deduce, using the formula for the expectation once again, that
\begin{align*}
	E_{\te} N_{FP}(\vphi,\te)  & \le
	E[ \text{Bin}(n-s_n,2\bar \Phi(\ta))] + E[ \text{PBin}(s_n,(p_j^*(\ta)_j))] \\
	& = E_{\te} \left[ N_{FP}(T(\ta),\te) + N_{FN}(T(\ta),\te)\right]=
	E_{\te} \lc(T(\ta),\te),
\end{align*}
where $T(\ta)(X)_i=\ind{|X_i|\ge \ta}$ is the thresholding procedure at level $\ta$. 

Turning now to the number of false negatives of $\vphi$, one notes that if $\vphi$ makes at least $k\ge 1$ false negatives, then among the $|X_i|$'s for $i\in S_\te$, at least $k$ are below $|X|_{(s_n)}$. This means that at least $k$ among the noise variables $|\xi|_i$'s are above $|X|_{(s_n)}$ which means that $P_{\te}[ N_{FN}(\vphi,\te) \ge k ]\le P_{\te}[ |\xi|_{(k)} \ge  |X|_{(s_n)} ]$ and this quantity has already been bounded before (one may note that this reasoning also gives $N_{FP}(\vphi,\te)=N_{FN}(\vphi,\te)$). The same reasoning then gives that $E_{\te} N_{FP}(\vphi,\te) \le E_{\te} \lc(T(\ta),\te)$. By taking the infimum over $\ta>0$, this concludes the proof. 
\end{proof}

\subsection{~~Proof of Corollary \ref{cor-topsub}}

		\begin{proof}
			First note that for the procedure $\vphi$ one has the identity $E_\te \lc(\vphi,\te)=s_n\fr(\vphi,\te)$ for all $\te$ since $\vphi$ rejects exactly $s_n$ coordinates (i.e. there are no ties among $X_i$'s) with probability $1$. 
			Focusing on the $\fr$--risk,  the corresponding supremum in the statement is bounded from above, using Theorem~\ref{thmtops}, by $2\sup_{\te\in\Theta_b}\inf_{T\in\cT} \fr(T,\te)$. This is in turn bounded by twice the maximum normalised risk of the oracle thresholding procedure with threshold $\sqrt{2\log{n/s_n}}$. By Theorem~\ref{thmgeneric}, this is $2\bar\Phi(b)+o(1)$. 
			
			To obtain a matching lower bound, note that $\fr(\vphi,\te)=E_\te[N_{FP}(\vphi,\te)+N_{FN}(\vphi,\te)]/s_n=2E_\te[N_{FP}(\vphi,\te)]/s_n=2\FNR(\vphi,\te)$, where we have used that  $N_{FP}(\vphi,\te)=N_{FN}(\vphi,\te)$ as noted in the proof of Theorem~\ref{thmtops}. The supremum in the statement is thus bounded below by $2\inf_{T} \sup_{\te\in\Theta_b} \FNR(T,\te)$. Since $\vphi$ is sparsity--preserving,  Theorem~\ref{thm-notrade} in the main paper implies that the latter bound is at least $2\bar\Phi(b)+o(1)$, which concludes the proof.
		\end{proof}

\subsection{~~Proof of Theorem~\ref{thmknownr} [Thm~\ref{thmlargeadapt}, (iii)]}	
\label{sec:proofknownr}

		\begin{proof}
			Let us recall the notation $\ta(r,\be)=\{\sqrt{r}/2+\be/(2\sqrt{r})\}\sqrt{2\log{n}}$. Below we use that this  threshold verifies for $s_n\asymp n^{1-\be}$ and $v_C(r,\beta)=v_C(r,\beta,2)$ the classification rate for Gaussian noise,
			\[ (n-s_n) \overline{\Phi}(\tau(r,\beta))+s_n\overline{\Phi}(\sqrt{2r\log(n)})-\tau(r,\beta)\asymp v_C(r,\beta)\] 
			 and  denote, recalling that the constant $a$ arises in the definition of the class $\Theta(r,\be)$ in \eqref{classab},
			\[ S_n = a n^{1-\be}/\log{n}\] 
as shorthand notation. 			
			Lemma~\ref{estimsn} below tells us that for some $C_1,C_2$, writing $\beta_\pm = \beta\pm C_1/\log^2 n$, we have $\hat{\beta}\in[\beta_-,\beta_+]$ on an event $\mathcal{C}$ of probability at least $1-e^{-C_2 S_n}$.
			Let us denote by $\vphi^+$ and $\vphi^-$ the thresholding procedures with thresholds $\ta(r,\be^+)$ and $\ta(r,\be^-)$ respectively.  
			By monotonicity, $N_{FP}(\vphi^r,\te)\le N_{FP}(\vphi^-,\te)$ and $N_{FN}(\vphi^r,\te)\le N_{FN}(\vphi^+,\te)$ on the event $\cC$. Since $N_{FP}\le n$, deduce, for any $\te$ in $\Theta(r,\be)$ with $|S_{\te}|=s_n$,
			\begin{align*}
				E_\te N_{FP}(\vphi^r,\te) & \le  E_\te [N_{FP}(\vphi^-,\te)1_{\cC}]+nP_\te[\cC^c] \\
				& \le E_\te[\text{Bin}(n-s_n,2\bar\Phi(\ta(r,\be^-)))] + ne^{-C_2 S_n}\\
				& \le 2(n-S_n)\bar\Phi(\ta(r,\be^-)) + ne^{-C_2 S_n}.
			\end{align*}
			Similarly, if $i_1,\ldots,i_{s_n}$ denote the indices in the support of $\te$, for PBin$(\ba)$ the Poisson-Binomial distribution (Section \ref{sec:poissonb}) with vector of parameters $\ba$,
			\begin{align*}
				E_\te N_{FN}(\vphi^r,\te) & \le E_\te[\text{PBin}(\{P(|\te_{i_j}+\veps_{i_j}|\le \ta(r,\be^+))\}_{1\le j\le s_n})] + ne^{-C_2S_n} \\
				& \le  s_n\bar\Phi(\sqrt{2r\log{n}}-\ta(r,\be^+)) + ne^{-C_2S_n}\\
				& \le (bS_n/a)\bar\Phi(\sqrt{2r\log{n}}-\ta(r,\be^+)) + ne^{-C_2S_n}.
			\end{align*} 
			Deduce, using Lemma \ref{lemtp}, 
			since $\ta(r,\be^+)$ and $\ta(r,\be^-)$ are only of order $\log^{-3/2}(n)$ away from $\ta(r,\be)$, that replacing the later quantities by $\ta(r,\be)$ in the previous bounds only leads to multiplicative factors that are $1+O(1)$ as $n\to\infty$. So
			\[ E_\te \lc(\vphi^r,\te) \le \{2(n-S_n)\bar\Phi(\ta(r,\be))+\frac{bS_n}{a}\bar\Phi(\sqrt{2r\log{n}}-\ta(r,\be))\}(1+O(1))+2ne^{-C_2S_n}. \]
			As the first term in the last bound is of order $v_C(r,\be)$ (which decays at most at a polynomial rate in $n$) and the last one is exponentially small in $n$, the result for $\lc$ follows.
			
The result for the $\fr$--risk follows quite easily. Let us briefly sketch the argument. It is enough to show that with overwhelming probability, $N_{FN}(\vphi^r,\theta)$ is no more than, say, $s_n/2$ over $\te\in\Theta(r,\be)$. Indeed, on the corresponding event, we must then have $\FDP(\vphi^r,\te)\le 2 \lc(\vphi^r,\te)/s_n$ and one concludes by taking expectations. To show the claimed high probability bound on $N_{FN}(\vphi^r,\theta)$, it suffices to note that $\vphi^r$ thresholds by definition at $\ta(\hat\be,r)$,  which is less than $\sqrt{2r_1\log{n}}$ for some $r_1<r$ with overwhelming probability (for $n\ge N(\veps)$ large enough) thanks to Lemma \ref{estimsn} and noting that $\be\to\ta(r,\be)$ is away from $r$ for $\be$ away from $r$. Now it suffices to note that the probability that $N_{FN}(\vphi^r,\theta)>s_n/2$ is binomial with success probability strictly less than $1/4$ (say), which concludes the proof. 
		\end{proof}

		
		\begin{lemma} \label{estimsn}
			Let $\veps>0$. In the setting of Theorem~\ref{thmknownr},  there exists an estimator $\hat\be$ and $C_1, C_2>0$ such that for $n\ge N(\veps)$ large enough,		
			\[ \sup_{\te\in \Theta(r,\be)} 
			P_\te\left[ |\hat{\be}-\be | > C_1/\log^2{n} \right] 
			\le e^{-C_2 a n^{1-\be}/\log{n}}.
			\]
		\end{lemma}
		\begin{proof}[Proof of Lemma~\ref{estimsn}]
			We consider the $\ell$-value procedure from \cite{cr20}, see also Section~\ref{sec:proof-lvals} (one could also use a BH procedure): 	
			recall that this procedure is defined as 
			\begin{equation} \label{lvals}
				\vphi^{\ell}(X_i)=\vphi^{\ell}_{\hat{w}}(X_i)=1\{ \ell(X_i;\hat{w},g) \le t \},
			\end{equation}
			for some fixed small $t$, say $t=1/4$, and $\ell(x;w,g)=(1-w)\phi(x)/\{(1-w)\phi(x)+wg(x)\}$, where we take the quasi-Cauchy density $g$ and where $\hat{w}$ is the marginal maximum likelihood estimator for $w$. 
		
We first derive 	the bound in the Lemma for $\te\in\Theta(r,\be)$ with $|S_\te|=s_n$. 
By Lemmas~\ref{lem:existence-of-w+-} and ~\ref{lem:concentration-of-hat{w}}, (noting that the set of large signals considered here is a subset of the signals considered for these Lemmas), if $\mathcal{B}=\{\hat{w}\in[w_-,w_+]\}$,
			\[ P_{\te}[\mathcal{B}^c] \le 2e^{-Cs_n},
			\] 
			as well as $w_-\asymp w_+ \asymp (s_n/n)\tilde{m}(n/s_n)^{-1}\asymp (s_n/n)\sqrt{\log(n/s_n)}$.  It follows, using \[ \vphi_{\hat{w}}^\ell(X_i)\ge \vphi_{w_2}^\ell(X_i)\] by monotonicity (Lemma~\ref{lem:monotonicity}) that $\vphi^\ell$ equals $1$ on $X_i$'s larger than $\xi(r(w_+,t))$. But 
			\[\xi(r(w_+,t)) \le \sqrt{2\be \log{n}}+o(1)=M(\be)+o(1)\] by Lemma \ref{lemxigen}. This bound in turn is smaller than $M(\be+\veps/2)$ for $n\ge N(\veps)$. Now recall that signals are at least $M(r)$ by definition of the class and that we assume $r\ge \be+\veps$. Bernstein's inequality now gives that there are at most $s_n/\log{s_n}$ signals that are below $M(\be+\veps/2)$ with probability at least $1-\exp(-Cs_n/\log{s_n})$. 
			Deduce	that if $N_{FN}(\vphi^\ell,\te)$ is the number of false negatives of the procedure $\vphi^\ell$ then with $K_n=s_n/\log{s_n}\geqa s_n/\log{n}$,
			\[ P_{\te}\left[ N_{FN}(\vphi^\ell,\te) > K_n \right]
			\le \exp\{ -CK_n\}. 
			\]
		
			Let us now turn to the number $N_{FP}(\vphi^\ell,\te)$ of false positives of the procedure $\vphi^\ell$. On the event $\mathcal{B}$, we have $\vphi_{\hat{w}}^\ell(X_i)\le \vphi_{w_-}^\ell(X_i)$ by monotonicity (Lemma~\ref{lem:monotonicity}) so $N_{FP}(\vphi^\ell,\te)\le N_{FP}(\vphi^\ell_{w_-},\te)$ on $\mathcal{B}$. On the other hand, $N_{FP}(\vphi^\ell_{w_-},\te)$ equals in distribution a $\text{Bin}(n-s_n,2\bar\Phi(\xi(r(w_-,t))))$. Using that 
			\[ \bar\Phi(\xi(r(w_-,t))) \asymp t (s_n/n)\zeta(w_-)^{-2},\]
			and combining this with Bernstein's inequality one gets, with $K_n'=s_n/\zeta(w_-)^2$, for suitable constants $c, C>0$,
			\[  P_{\te}\left[ N_{FP}(\vphi^\ell,\te) > cK_n' \right]
			\le \exp\{ -CK_n'\}. 
			\]
Note that $K_n\asymp K_n' \asymp s_n/\log{n}$ and let us set
			\[ \hat{s}=\sum_{i=1}^n \vphi^{\ell}(X_i). \]
The above bounds on $N_{FN}(\vphi^\ell,\te), N_{FP}(\vphi^\ell,\te)$ imply			\[ 
			P_\te\left[ |\hat{s}-s_n| > C_1  s_n/\log{n} \right] \le e^{-Cs_n/\log{n}}\le e^{-C'S_n}.
			\]
Setting $\hat{\be}:=\log(n/\hat{s})/\log{n}$, for $\hat{s}$ as above, one gets
			\[ 
			P_\te\left[ |\hat{\be}-\be | > C_1/\log^2{n} \right] \le e^{-C_2S_n}.
			\]
This bound is valid for any $\te\in \Theta(r,\be)$ with $|S_\te|=s_n$. Since the bound is independent on $s_n$, it is also valid over $\Theta(r,\be)$, which yields the result.			
		\end{proof}			 

\subsection{~~Proof of Theorem~\ref{thmadapt} [Thm~\ref{thmlargeadapt}, (iv)]}	
	\label{sec:proofadapt}

		\begin{proof}
			Let us consider the family of thresholding estimates $T(a)$ for some $a>0$ defined as $T(a)(X)_i=\ind{|X_i|>M(a)}$ for $i=1,\ldots,n$, where $M(a)=\sqrt{2a\log{n}}$. A bound for the classification risk of $T(a)$ is recalled in Lemma~\ref{lemt}.
			
			We define an estimator  using Lepski's method. 
			One defines the grid, for $\delta_n=1/\log{n}$,
			\begin{align*}
				\cA_n & =\left\{ \delta_n, 2\delta_n,\ldots, 
				 \lfloor 4/\delta_n \rfloor \delta_n \right\}, 
				 = \left\{ a_k:= k\delta_n, \ k\in \cK_n \right\},
			\end{align*}
			with $\cK_n=\{1,\ldots,  \lfloor 4/\delta_n \rfloor\}$.
			
			For the `true' signal strength $r$ and any $l\in \cK_n$, let us set
			\begin{equation} \label{bv}
				v_l = 2n\bar{\Phi}\left(\sqrt{a_l}\sqrt{2\log{n}}\right),\qquad B(l) = 2bn^{1-\be}\left\{\bar\Phi\left(\{\sqrt{r}-\sqrt{a_l}\}\sqrt{2\log{n}}\right) \wedge 1\right\}.
			\end{equation} 
			
			Define, for $L=8$, 
			\begin{equation} \label{defkach}
				\hat{k} = \max\left\{k\in \cK_n,\ \lc(T(a_k),T(a_l))\le L v_l,\ \text{for all }\ l\le k
				\right\},
			\end{equation}
			Also define, for $B(k), v_k$ as in \eqref{bv}, and $\cK_n$ as above,
			\begin{align}
				k^*& = \max \left\{k\in\cK_n,\ \ B(k)\le v_k \right\}.\label{ket}
			\end{align}
			The estimator is defined, for $\hat{k}$ as in \eqref{defkach}, 
			 as 
			 \[ \vphi(X) = T(a_{\hat k}). \]
			One bounds the classification risk of $\vphi$ for $\te_0\in\Theta(r,\be)$ by
			\[E_{\te_0}\lc(\vphi,\te_0) \le E_{\te_0}\lc(T(a_{\hat{k}}),T(a_{k^*})) +E_{\te_0}\lc(T(a_{k^*}),\te_0).\]  Lemma \ref{lemt} now gives
			$E_{\te_0}\lc(T(a_{k^*}),\te_0)\le B(k^*)+v_{k^*}\le 2v_{k^*}$, 
			while by the triangle inequality the other term is bounded by  
			\[ E_{\te_0}[\lc(T(a_{\hat{k}}),T(a_{k^*}))
			\ind{\hat{k}< k^*}]
			+ E_{\te_0}[\lc(T(a_{\hat{k}}),T(a_{k^*}))\ind{\hat{k}> k^*}]=:(i)+(ii). \]
			Dealing first with the term (ii), on the event $\{\hat{k}>k^*\}$, by definition 
			$\lc(T(a_{\hat{k}}),T(a_{k^*}))\le Lv_{k^*}$. To bound (i), one splits
			\begin{align*}
				(i) & = \sum_{p=0}^{k^*-1} E_{\te_0}[\lc(T(a_p),T(a_{k^*})) \ind{\hat{k}= p}] \le n \sum_{p=0}^{k^*-1} P_{\te_0}[\hat{k}= p],
			\end{align*}
			using the rough bound $\lc(\te,\te')\le n$. Using  the definition of $\hat{k}$ once again, since $p\le k^*-1$,
			\begin{align*}
				P_{\te_0}[\hat{k} = p] 
				& \le \sum_{l=0}^{k^*-1} P_{\te_0}[\lc(T(a_{k^*}),T(a_l)) > L v_l ] \\
				& \le 2k^*\max_{0\le l\le k^*} P_{\te_0}[\lc(T(a_l),\te_0) > L v_l/2 ],
			\end{align*}
			bounding the sum by $k^*$ times the maximum and using the triangle inequality. By Lemma \ref{lctproba},
			\[ (i) \le 4n(k^*)^2e^{-v_{k^*}/4}, \]
			using that for $l\le k^*$, one has $v_l\ge v_{k^*}$.  We have obtained, for any $\te_0\in \Theta(r,\be)$, 
			\[ E_{\te_0} \lc(\vphi,\te_0) \le  (2+L)v_{k^*}+4n(k^*)^2e^{-v_{k^*}/4}.\]
			To conclude it is enough to verify that, for $(\be,r)\in\cJ$, we (uniformly) have  $v_{k^*}\asymp v_C(r,\be)$ and that $v_C(r,\be)$ is larger than a small power of $n$, so that the last term in the above display is indeed a $o(v_{k^*})$. Define $a_{l^*}$ to be the (or `a' in case there are two) closest point on the grid $\cA_n$ to $a^*:=\be+\kappa(r,\be)$. Then
			\[ v_{l^*} \sim \frac{2n^{1-a_{l^*}}}{\sqrt{2\pi}\sqrt{2a_{l^*}\log{n}}}\sim\frac{2n^{1-a^*+(a^*-a_{l^*})}}{\sqrt{2\pi}\sqrt{2a_{l^*}\log{n}}}\asymp \frac{n^{1-a^*+O(\delta_n)}}{\sqrt{\log{n}}},\]
			and 
			\[ B(l^*) \sim \frac{2 bn^{1-\be} n^{-(\sqrt{r}-\sqrt{a_{l^*}})^2}}{\sqrt{2\pi}\sqrt{2(\sqrt{r}-\sqrt{a_{l^*}})\log{n}}}\asymp \frac{n^{1-\be-(\sqrt{r}-\sqrt{a^*})^2+O(\delta_n)}}{\sqrt{\log{n}}}.\]
			One checks using the definition of $a^*$ that $a^*=\be+(\sqrt{r}-\sqrt{a^*})^2$ which shows that $v_{l^*}\asymp B_{l^*}$. Since for finite $K$ we have $v_{l^*\pm K}\sim e^{\mp K} v_{l^*}$ and since $l\to B(l)$ is increasing, we deduce $|l^*-k^*|$ is finite so that $v_{k^*}$ is of the same order as $v_{l^*}\asymp n^{1-\be-\kappa(r,\be)}/\sqrt{\log{n}}\asymp v_C(r,\be)$. This concludes the proof for the classification risk.
			
To deduce the result for the $\fr$--risk, one uses a similar argument as in the proof of Theorem \ref{thmknownr}: as the threshold is in the large signal regime, the denominator of the FDP is of order $s_n$ with high probability. Details are the same as in the earlier proof and are omitted.	
		\end{proof}

		\begin{lemma}\label{lemt}
			Let $T(a)$ for some $a>0$ be defined as $T(a)(X)_i=\ind{|X_i|>M(a)}$ for $i=1,\ldots,n$, where $M(a)=\sqrt{2a\log{n}}$. Then, for $1\le s_n\le n$, and any $\te$ with $|S_\te|=s_n$ and $|\te_i|\ge M(r)$ for $i\in S_\te$, 
			\[ E_\te \lc(\te,T(a)) \le 2(n-s_n)\bar\Phi\left(\sqrt{a}\sqrt{2\log{n}}\right)+ s_n\{2\bar\Phi\left((\sqrt{r}-\sqrt{a})\sqrt{2\log{n}}\right)\wedge 1 \}. \]
		\end{lemma}
		\begin{proof}
			For any  $\te$ as in the statement of the lemma, we have
			\begin{align*}
				E_{\te}\lc(T(a),\te) & = \sum_{i\notin S_{\te}} P_{\te}[|X_i|>M(a)]
				+  \sum_{i\in S_{\te}} P_\te[|X_i|\le M(a)].
			\end{align*}
			If $\te\in \ell_0[s_n]\setminus \ell_0[s_n-1]$, for fixed $a>0$,
	the first sum equals $2(n-s_n)\bar\Phi(M(a))$. 
			On the other hand, if $\te\in \ell[s_n]\setminus \ell[s_n-1]$, for fixed $a>0$, the second sum in the last display is 
			\begin{align*}
				& \sum_{i\in S_\te} P_{\te}(|X_i|<M(a))=
				\sum_{i\in S_\te} \{ \Phi(M(a)-\te_i)-\Phi(-M(a)-\te_i) \}\\
				& \le \sum_{i\in S_\te}\{2\bar\Phi(|\te_i|-M(a)) \wedge 1\}.
			\end{align*}
			To conclude one uses the condition on nonzero signals and that $\bar\Phi$ is decreasing.
		\end{proof}

		\begin{lemma} \label{lctproba} 
		Let $v_l$ be defined by \eqref{bv} and $k^*$ by \eqref{ket}. 	For any $\te_0\in \Theta(r,\be)$ and any $l\le k^*$, one has 
	\[ P_{\te_0}[\lc(T(a_l),\te_0) > 4 v_l ] \le 2 e^{-v_l/4}. \]
		\end{lemma}
		\begin{proof}
			Note $E_{\te_0}\lc(T(a_l),\te_0) \le B(l)+v_l\le B(k^*)+v_l\le v_{k^*}+v_l\le 2v_l$ using $l\le k^*$ and that $l\to B(l)$ (respectively $l\to v_l$) is increasing (respectively decreasing). So
			\begin{align*}
				& P_{\te_0}[\lc(T(a_l),\te_0) > 4 v_l ]\\
				& = P_{\te_0}[\lc(T(a_l),\te_0) - E_{\te_0}\lc(T(a_l),\te_0)> 4 v_l - E_{\te_0}\lc(T(a_l),\te_0)]\\
				& \le P_{\te_0}\Big[ \sum_{i\notin S_{\te_0}} \left\{ \ind{|X_i|>M(a_l)}-P_{\te_0}(|X_i|>M(a_l)) \right\} > v_l \Big]\\
				& \ \ + P_{\te_0}\Big[
				\sum_{i\in S_{\te_0}} \left\{\ind{|X_i|\le M(a_l)}-P_{\te_0}(|X_i|\le M(a_l))\right\} >v_l \Big].
			\end{align*}
			Applying Bernstein's inequality to each term in the previous expression concludes the proof, noting that the sum of variances for the first term  is bounded by $v_l$ and for the second term by $B(l)\le v_l$ when $l\le k^*$. 
		\end{proof}
		
		\begin{lemma}\label{lemtp}
			In the setting of Lemma \ref{lemt}, for $\kappa$ as in \eqref{eqkappa}, let $\ta$ be such that, for $(R_n)$ any sequence with $R_n^2=o(\log{n})$,
			\[ \ta = \sqrt{2\log{n}}  \left( \sqrt{r}-\sqrt{\kappa} \right) + o(R_n).\]
			Then the thresholding procedure $\vphi_\ta(X)_i=\ind{|X_i|>\ta}$ for $i=1,\ldots,n$, satisfies
			\[ \sup_{\te\in \Theta(r,\be)} E_\te \lc(\te,\vphi_\ta) \leqa  n^{1-\be-\kappa}/\sqrt{\log{n}}. \]
		\end{lemma}
		\begin{proof}
			One applies Lemma~\ref{lemt} with $\sqrt{a}=\sqrt{r}-\sqrt{\ka}+R_n/\sqrt{2\log{n}}$. Using the bound $\bar\Phi(x)\le \phi(x)/x$ for $x>0$, one bounds from above each of the terms appearing in Lemma~\ref{lemt} by $O(n^{1-\be-\ka})/\sqrt{\log{n}}$, which gives the result.
		\end{proof}

\section{~~The Poisson-binomial distribution and concentration inequalities}
\label{sec:poissonb}

Let $\operatorname{PBin}[\mathbf{a}]$ denote the Poisson-binomial distribution of parameter $\mathbf{a}=(a_1,\ldots,a_S)$, for $S\ge 1$: it is the distribution of $\sum_{i=1}^S Z_i$, where $Z_i$ are independent $\text{Be}(a_i)$ random variables.

\begin{lemma}\label{lem:stodg}
	Let $\mathbf{a}=(a_1,\ldots,a_S), \mathbf{b}=(b_1,\ldots,b_S)$ be vectors in $[0,1]^S$ for some $S\ge 1$. Suppose $a_i\le b_i$ for all $1\le i\le S$. Then, for any integer $k$,
	\[ P[\operatorname{PBin}[\mathbf{a}]\ge k] \le P[\operatorname{PBin}[\mathbf{b}]\ge k]. \]
	That is, $\operatorname{PBin}[\mathbf{b}]$ stochastically dominates $\operatorname{PBin}[\mathbf{a}]$.
\end{lemma}
\begin{proof}
	By transitivity it is enough to prove the result in the case $\mathbf{a}$ and $\mathbf{b}$ differ only by one coordinate, say (by symmetry) the first one: that is, let $\mathbf{a}=(a_1,\ldots,a_S)$ and $\mathbf{b}=(b_1,a_2,\ldots,a_S)$. If $X_1\sim \operatorname{PBin}[\mathbf{a}]$ and $X_2 \sim \operatorname{PBin}[\mathbf{b}]$, one can write $X_1=\veps+T_1$ and $X_2=\xi+T_2$ (in distribution), where $T_1, T_2$ are equal in law, and $\veps\sim \operatorname{Be}(a_1), \xi\sim \operatorname{Be}(b_1)$, independently of $T_1, T_2$ respectively. For any integer $k$,
	\begin{align*}
		P[X_1\ge k] & =
		P[T_1+1\ge k\given  \veps=1]a_1+P[T_1\ge k\given  \veps=0](1-a_1) \\
		& \le  P[T_1+1\ge k\given  \veps=1]b_1+P[T_1\ge k\given  \veps=0](1-b_1) =P[X_2\ge k],
	\end{align*} 
	where the second line uses that $P[T_1+1\ge k\given  \veps=1]-P[T_1\ge k\given  \veps=0]=P[T_1\ge k-1]-P[T_1\ge k]\ge 0$ and $a_1\le b_1$, which concludes the proof.
\end{proof}
\begin{lemma}\label{lem:pbcondg}
	Let $\mathbf{Y}=(Y_1,\ldots,Y_S)$ be a vector of independent random variables taking values in $[0,1]$, and let $\bm{p}=(p_j)\in[0,1]^S$ be a deterministic vector such that  $E[Y_j]=p_j$ for all $1\le j\le S$. 
	Conditionally on $\bm{Y}$, let $\xi_j\sim \operatorname{Be}(Y_j)$, $j\leq S$ be independent Bernoulli random variables. Then
	\[ \sum_{i=1}^S\xi_j \sim \operatorname{PBin}[\bm{p}]. \]
\end{lemma}
\begin{proof}
	Observe that the variables $\xi_j$, defined to be independent given $\mathbf{Y}$, are also independent unconditionally, because the $Y_j$'s are independent. It now suffices to note that $\xi_j$ has Bernoulli distribution, with parameter $E[\xi_j]=E[E[  \xi_j\given Y_j]]=E[Y_j]=p_j$. 
\end{proof}

\begin{lemma}[Bernstein's inequality]\label{th:bernstein}
	Let $W_i$, $1\leq i \leq n$ centered independent variables with $|W_i|\le \mtc{M}$ and $\sum_{i=1}^n\Var(W_i)\le V$, then for any $A>0$,
	\[ P\left[ \sum_{i=1}^n W_i >A \right] \le \exp\left\{-\frac12 A^2/(V+\mtc{M}A/3) \right\}.\]
	In particular, let $\cP\sim \operatorname{PBin}[\bm{p}]$, with $\bm{p}\in[0,1]^S$, $S\ge 1$, and set $\mu:=\sum_{j=1}^S p_j$. Then for any $0\le \delta\le 1$,
	\[ P[ \cP \ge \mu(1+\delta)] \le e^{-\mu\delta^2/3}. \]
\end{lemma}
\begin{proof}
	The first part of the lemma is the standard Bernstein inequality. The second part follows, for $\xi_j$ independent Bernoulli variables with parameter $p_j$, by setting $W_j=\xi_j-E[\xi_j], A=\mu\delta, V=\sum_{j=1}^S p_j= \mu, \cM=1$ and noting, since $0\le \delta\le 1$,
	\[\exp(-\frac{\mu^2\delta^2}{2\mu+2\mu\delta/3})\le \exp(-3\mu\delta^2/8)\le \exp(-\mu\delta^2/3). \qedhere\] 
\end{proof}

\section{~~Properties of Subbotin distributions} \label{sec:applem}
\begin{lemma}\label{lem:sub}
	Let $\phi_\zeta(\cdot)$ be defined by \eqref{eqn:def:Subbotin-density} for some $\zeta>1$. Then $\ol{\Phi}_\zeta(x)=\int_x^{+\infty} \phi_\zeta(u) du$  has the following properties:
	\begin{itemize}
		\item[\textbullet] for any $x>0$, we have
		\begin{align}
			\ol{\Phi}_\zeta(x)&< \phi_\zeta(x) / x^{\zeta-1}\label{equ:phi1} \,;
		\end{align}
		\item[\textbullet] for any $t\in (0, 1/2)$ s.t. $\ol{\Phi}_\zeta^{-1}(t)\geq 1$, we have 
		\begin{align}
			\ol{\Phi}_\zeta^{-1}(t)&\leq 
			( \zeta\log (1/t)-\zeta \log L_\zeta)^{1/\zeta}
			\,\label{equ:phi2};
		\end{align}
	\end{itemize}
	The following holds:
	\begin{itemize}
		\item[\textbullet] for any $x>0$, 
		\begin{align}
			\ol{\Phi}_\zeta(x)&\geq \frac{\phi_\zeta(x)}{x^{\zeta-1}} \bigg[ 1+ (\zeta-1)x^{-\zeta} \bigg]^{-1} \label{equ:phi3}\,; \\
			\ol{\Phi}_\zeta(x) &\geq \frac{\phi_\zeta(x)}{x^{\zeta-1}} \zeta^{-1} \label{equ:phi4} \:\mbox{ if $x\geq 1$}\,; 
		\end{align}
		\item[\textbullet] for any $t\in (0, 1/2)$ s.t. $\ol{\Phi}_\zeta^{-1}(t)\geq 1$, we have 
		\begin{align}
			\ol{\Phi}_\zeta^{-1}(t)&\geq 
			\left( 0\vee \bigg\{ \zeta \log (1/t) -\zeta\log (\zeta L_\zeta )- (\zeta-1) \log\left(\zeta \log(1/t)-\zeta \log L_\zeta \right) \bigg\}\right)^{1/\zeta}
			\label{equ:phi5}.
		\end{align}
	\end{itemize}
\end{lemma}

\begin{proof}
	
	Let $\psi(u)=u^\zeta/\zeta+\log L_\zeta$, so that, restricting to $u>0$,
	\begin{equation*}
		\phi_\zeta(u)=e^{-\psi(u)},\quad 
		\psi'(u)=u^{\zeta-1},\quad
		\psi^{-1}(v)=(\zeta v-\zeta \log L_\zeta)^{1/\zeta},\quad 
		\psi''(u)=(\zeta-1)u^{\zeta-2},
	\end{equation*}
	and thus for instance $\psi'\circ\psi^{-1}(v)=(\zeta v-\zeta \log L_\zeta)^{1-1/\zeta}$, $ \psi''(u)/(\psi'(u))^2=(\zeta-1)u^{-\zeta}$.
	Inequality \eqref{equ:phi1} holds because $\ol{\Phi}_\zeta(x)=\int_x^{+\infty} e^{-\psi(u)} du < (\psi'(x))^{-1}\int_x^{+\infty} \psi'(u) e^{-\psi(u)} du = \phi_\zeta(x) / \psi'(x)$. Expression \eqref{equ:phi2} follows from \eqref{equ:phi1} applied with $x=\ol{\Phi}_\zeta^{-1}(t)\geq 1$. Indeed, the latter entails $L_\zeta^{-1} \exp\brackets{-\ol{\Phi}_\zeta^{-1}(t)^\zeta/\zeta}>t x^{\zeta-1}\geq t$, from which \eqref{equ:phi2} follows by applying the function $\log(\cdot)$ to both sides of the inequality.
	To prove \eqref{equ:phi3}, write for any $x>0$,
	\begin{align*}
		\frac{\psi''(x)}{\psi'(x)^2} \ol{\Phi}_\zeta(x)& \geq \int_x^{+\infty}  \frac{\psi''(u)}{\psi'(u)^2} e^{-\psi(u)} du = \bigg[ -\frac{e^{-\psi(u)}}{\psi'(u)} \bigg]_{x}^\infty- \ol{\Phi}_\zeta(x)= \frac{\phi_\zeta(x)}{\psi'(x)}- \ol{\Phi}_\zeta(x),
	\end{align*}
	by using an integration by parts. Expressions \eqref{equ:phi3} and \eqref{equ:phi4} follow.
	Finally, let us prove \eqref{equ:phi5}. From \eqref{equ:phi4} used with $x=\ol{\Phi}_\zeta^{-1}(t)$, we get $\zeta t (\ol{\Phi}_\zeta^{-1}(t))^{\zeta-1}\geq e^{-\psi(\ol{\Phi}_\zeta^{-1}(t))}$ and thus $-\log(\zeta t) - (\zeta-1)\log(\ol{\Phi}_\zeta^{-1}(t))\leq \psi(\ol{\Phi}_\zeta^{-1}(t))$. Hence, by \eqref{equ:phi3}, we obtain 
	\[\log L_\zeta \vee\left( -\log(\zeta t) - (1-1/\zeta)\log(\zeta\log (1/t)-\zeta \log L_\zeta )\right)\leq \psi(\ol{\Phi}_\zeta^{-1}(t))\]
	from which \eqref{equ:phi5} follows.
\end{proof}

\begin{lemma}\label{lem:dConvex}
	For any $a\geq 0$ and any $x\in \RR$,	\begin{align*} \phi_\zeta(x-a)/\phi_\zeta(x) &\geq \exp\brackets[\big]{a \abs{x-a}^{\zeta-1} \sign(x-a)}, \\ 
		\phi_\zeta(x+a)/\phi_\zeta(x) &\leq \exp\brackets[\big]{-a \abs{x}^{\zeta-1} \sign(x)}.
	\end{align*}
	
\end{lemma}
\begin{proof}
	Since $\zeta>1$, observe that $-\log \phi_\zeta(x)= \zeta^{-1} \abs{x}^\zeta +c$ is differentiable (even at $0$) with increasing derivative $\abs{x}^{\zeta-1}\sign(x)$. It follows, using $a\geq 0$, that
	$-\log \phi_\zeta(x)\geq -\log \phi_\zeta(x-a) + a \abs{x-a}^{\zeta-1} \sign(x-a)$, and hence the first claim. The second follows by substituting $x+a$ for $x$.
\end{proof}

\begin{lemma}\label{lem:compazeta}
Let $\zeta>1$. 	For any $a\geq 0$ and any $x\le a$,
	\[ \phi_\zeta(x)/\phi_\zeta(x-a) \ge \exp\brackets[\big]{-a^\zeta/\zeta}.\]
\end{lemma} 
\begin{proof}
The ratio in the statement is $\exp(-|x|^\zeta/\zeta+(a-x)^\zeta/\zeta)$ by definition for $x\le a$. The map $x\to (a-x)^\zeta-|x|^\zeta$ is decreasing on $(-\infty,a]$: indeed, its derivative for $x\neq 0$ is  $-\zeta(a-x)^{\zeta-1}-\zeta\sign(x)|x|^{\zeta-1}$ which is negative for $x\neq 0$, using that $u\to u^{\zeta-1}$ is increasing on $[0,+\infty)$ since $\zeta>1$.
\end{proof}

\section{~~In-probability classification risk bounds} \label{sec:classification-in-prob}
In this section we give several in-probability bounds for the classification risk, both for general signals and in the large signal case for classes $\Theta(r,\be)$ as in \eqref{classab}.  Upper bounds have been given already for the $\ell$-value and BH procedures, hence the focus here is on proving in-probability lower bounds. These results are interesting on their own, and in the general signal case they can also be used to provide an alternative proof of the lower bound part of the main Theorem \ref{th:minimax-multilevel} for general signals.


\subsection{~~Classification:  sharp adaptive minimaxity in probability}\label{sec:classprob}

Recall the definition of the classification  loss 
\begin{equation*}
 \lc(\te,\vphi) = \sum_{i=1}^n \left(\ind{\te_i=0}\ind{\vphi_i\neq 0}
+ \ind{\te_i\neq 0}\ind{\vphi_i = 0}\right). 
 \end{equation*}

\begin{theorem} \label{thm-adapt-clpr}
	 If $a_b=\sqrt{2\log{n/s_n}}+b$ for a fixed real $b$,  then for $\Theta_b=\Theta(a_b;s_n)$ and any $\eta>0$,
	 \[ {
\inf_{\vphi} \sup_{\te\in\Theta_b} P_{\te}\left[\lc(\te,\vphi)/s_n \ge  \overline{\Phi}(b)-\eta \right] =
1+o(1).} \]
If $b=b_n\to-\infty$, the last display holds with $\overline{\Phi}(b)$ replaced by 1. 

There exist procedures $\vphi$ achieving this bound, in that for any fixed real $b$ and any $\eta>0$
\[\sup_{\te\in\Theta_b} P_\te\left[\lc(\te,\vphi)/s_n \ge  \overline\Phi(b)+\eta
\right]
=o(1), \]
and the same holds with $\overline{\Phi}(b)$ replaced by 0 when $b=b_n\to+\infty$ or by 1 when $b=b_n\to-\infty$.
In particular this is true for $\vphi=\vphi^{\hat{\ell}}$ the empirical Bayes $\ell$-value procedure \eqref{ellvalp} or, under the polynomial sparsity assumption \eqref{polspa}, the BH procedure $\vphi^{BH}_\alpha$ \eqref{defBH} for $\alpha=\alpha_n = o(1)$ with $-\log \alpha = o((\log n)^{1/2})$.
\end{theorem}
The upper bounds for the BH procedure and $\ell$-value procedure can be found in Sections~\ref{proofBHclassif}~and~\ref{sec:proof:thm-adapt-clpr} respectively. The lower bound of Theorem~\ref{thm-adapt-clpr} follows from Theorem~\ref{thm:clarg} (to follow), which gives a corresponding bound over the more general set $\Theta(\ba,s_n)$ of \eqref{parametersetmultiscale}. Specifically, by taking $\rho=1$, it says that for $0\leq \eps\leq 1$ and $n$ large enough that $s_n>4$, we have
	\[ \inf_\vphi \sup_{\theta\in\Theta(\ba,s_n)} P_{\te} \brackets[\Big]{\lc(\te,\vphi)\geq \min\brackets[\big]{1- \tfrac{1+\eps}{\sqrt{s_n}},(\Lambda_n -\eps )} s_n}\geq 1-\exp(-\sqrt{s_n}\eps^2/3).\]  Since $\Theta(\ba,s_n)=\Theta_b$ if $\ba=(a_n^*+b,\dots,a_n^*+b)$, and in the Gaussian model $\Lambda_n(\ba)=\overline{\Phi}(b)$ for this $\ba$, we deduce the claim. 

\subsection{~~Classification and large signals, in probability results} \label{sec:prlarge}

The following result can be proved using similar arguments as Theorem \ref{thm-adapt-clpr}, adapting these to the large signals regime. We state it in Gaussian noise for simplicity and omit the proof.
\begin{theorem} \label{thm:inprlarge}
Consider the same setting as Theorem \ref{thmRabinovic}, in the Gaussian noise case $\zeta=2$ and with $\be\in(0,1)$ and $r>\be$. 

If $n^{1-\be -\kappa}/\sqrt{\log{n}}$ is bounded away from zero, for $M_n$ going to $\infty$ arbitrarily slowly,
\[ \inf_{\vphi} \sup_{\te\in \Theta(r,\be)} 
P_{\te}\left(\lc(\te,\vphi)\ge  M_n  n^{1-\be -\kappa}/\sqrt{\log{n}}
\right) \to 1, \quad (n\to\infty).\] 
If on the other hand $n^{1-\be-\kappa}/\sqrt{\log{n}}$ is bounded from above, it holds
\[ \inf_{\vphi} \sup_{\te\in \Theta(r,\be)} 
P_{\te}\left(\lc(\te,\vphi)\neq 0
\right) \asymp  n^{1-\be-\kappa}/\sqrt{\log{n}}, \quad (n\to\infty).\]
\end{theorem}
Aside from proving an in-probability lower bound that matches the in-expectation bound from Theorem \ref{thmRabinovic} in the first part of the statement, we quantify the vanishing probability of the existence of at least one misclassified label in the `exact recovery' regime in the second part of the statement.

\subsection{~~Classification: lower bounds for weighted loss}


Before proving the main result of this section, Theorem \ref{thm:clarg}, let us show briefly how this result can be used to provide an alternative proof of Theorem~\ref{th:minimax-multilevel} (which is an `in-expectation' result).
Recall from \eqref{equLrho} the definition for $\rho>0$ of a weighted, non-symmetric loss function
\begin{equation} \label{def:wcl}
	L_\rho(\te,\vphi) = \sum_{i=1}^n \left\{\ind{\te_i=0}\ind{\vphi_i\neq 0}
	+ \rho \ind{\te_i\neq 0}\ind{\vphi_i = 0}\right\}. 
\end{equation}
Standard classification loss $\lc$ corresponds to $\rho=1$. 
	For $\eps>0$, we may apply Theorem~\ref{thm:clarg} with a sequence $\rho=\rho_n\to \infty$ to obtain
		\[ \inf_\vphi \sup_{\theta\in\Theta(\ba,s_n)} P_{\te\in \Theta(\ba,s_n)} \brackets[\Big]{L_\rho(\te,T)\geq \min\brackets[\Big]{1- \tfrac{1+\eps}{\sqrt{s_n}},\Lambda_n(\ba)-\eps}\rho s_n}\geq 1-\exp(-\sqrt{s_n}\eps^2/3).\]
		Then Lemma~\ref{lem:mtclr} applied with this $\rho_n$ and $\lambda=\min(1-(1+\eps)s_n^{-1/2},\Lambda_n(\ba) -\eps )$ further yields 
	\[ \inf_\vphi \sup_{\theta\in\Theta(\ba,s_n)} \frak{R}(\theta,\vphi) \geq \brackets[\Big]{\lambda \wedge \frac{\rho_n \lambda}{1+\rho_n\lambda}} (1-\exp(-\sqrt{s_n}\eps^2/3)).\]
	In view of the fact that $\lambda=\Lambda_n(\ba) - \eps $ for $n$ large enough,
	\[ \inf_\vphi \sup_{\theta\in\Theta(\ba,s_n)} \frak{R}(\theta,\vphi) \geq  \Lambda_n(\ba)-\eps +o(1).\]
	The left side does not depend on $\eps$, so taking the limit as this tends to zero yields the lower bound on the minimax risk obtained already as Theorem~\ref{th:minimax-multilevel} (first bullet).

\begin{lemma} \label{lem:mtclr}
	For any procedure $\vphi$, any $\te\in\ell_0[s]$ with $s\ge 1$, and any $\la<1$ and $\rho>0$,
	\[ \fr(\te,\vphi) \ge  \left(\la\wedge\frac{\rho\la}{1+\rho\la} \right) 
	P_{\te}(L_\rho(\te,\vphi)\ge \la\rho s). \]
\end{lemma}

\begin{proof}
	If $\lambda\leq 0$ the bound holds trivially. Otherwise, let us write $D_n(X)=\sum_{i=1}^n \vphi_i(X)$ for the total number of rejections, and denote by $s_\theta$ the number of nonzero coefficients of $\te$. Define  
	\[ Q(X)=Q(X,\vphi)=\sum_{i=1}^n \left\{ \ind{\te_i=0} \frac{\vphi_i(X)}{1\vee D_n(X)}
	+ \ind{\te_i\neq 0} \frac{1-\vphi_i(X)}{s} \right\},
	\]
	so that $\fr(\te,\vphi)\ge E_\te Q(X)$ , using $s_\theta\le s$. Let $\cA_n=\{D_n(X)\le (1+\delta)s\}$, for $\delta>0$. On the one hand, since $1\vee D_n(X)\le (1+\delta)s$ on $\cA_n$,
	\begin{align*}
		Q(X)\ind{\cA_n} & \ge \ind{\cA_n} \sum_{i=1}^n \left\{\ind{\te_i=0} \frac{\vphi_i(X)}{(1+\delta)s} + \frac{\rho}{\rho} \ind{\te_i\neq 0} \frac{1-\vphi_i(X)}{s} 
		\right\} \\
		& \ge \ind{\cA_n} \left(\frac1{1+\delta}\wedge \frac1\rho\right)\frac1{s} L_\rho(\te,\vphi).
	\end{align*}
	On the other hand, if $\cA_n^c$ denotes the complement of $\cA_n$, 
	\begin{align*}
		Q(X) \ind{\cA_n^c} & \ge   \ind{\cA_n^c} \sum_{i=1}^n 
		\ind{\te_i=0} \frac{\vphi_i(X)}{D_n(X)} \\
		& \ge \ind{\cA_n^c}  \frac{\sum_{i=1}^n \vphi_i(X) - 
			\sum_{i=1}^n \ind{\te_i\neq 0}\vphi_i(X)}{D_n(X)} \\
		&  \ge \ind{\cA_n^c}\frac{D_n(X) - s_\theta}{D_n(X)}
		\ge \ind{\cA_n^c}\frac{D_n(X) - s}{D_n(X)}\ge \frac\delta{1+\delta}\ind{\cA_n^c}.
	\end{align*}
	Combining the previous bounds and setting $\cC_n=\{L_\rho(\te,\vphi)\ge \la\rho  s\}$, we obtain
	\begin{align*} 
		Q(X) & \ge \left(\frac1{1+\delta}\wedge \frac1{\rho}\right)\frac1{s}L_\rho(\te,\vphi) \ind{\cA_n}
		+\frac{\delta}{1+\delta}\ind{\cA_n^c}\\
		& \ge \left(\frac1{1+\delta}\wedge \frac1{\rho}\right)\rho \la   \ind{\cA_n}\ind{\cC_n}
		+\frac{\delta}{1+\delta}\ind{\cA_n^c}\ind{\cC_n}\\
		& \ge 
		\left[ \left(\frac{\la\rho}{1+\delta}\wedge \la\right) \wedge \frac{\delta}{1+\delta} 
		\right] \ind{L_\rho(\te,\vphi)\ge  \la \rho s},
	\end{align*}
	where the second line uses that $1\ge  \ind{\cC_n}$ and uses the definition of $\cC_n$ to bound $L_\rho(\te,\vphi)$ from below, and the third line that $\ind{\cA_n} + \ind{\cA_n^c}=1$. 
	Setting $\delta=\la\rho>0$ and taking the expectation under $P_\te$ on both sides of the inequality leads to the result.
\end{proof}

Recall the definition \eqref{parametersetmultiscale} of $\Theta(\ba,s_n)$ as comprising vectors $\theta\in\ell_0[s_n]$ with non-zero components of absolute values at least $a_j$, and the definition \eqref{eqn:LambdaN} of $\Lambda_n(\ba)=s_n^{-1}\sum_{j=1}^{s_n} F_{a_j}(a_n^*)$.
\begin{theorem}[Lower bound for weighted classification losses] \label{thm:clarg}
	Grant Assumption~\ref{ass:generalnoise}\ref{ass:location} and fix $\ba\in \RR_+^{s_n}$. Then for any integer sequence $\rho=\rho_n$ satisfying \[1\leq \rho \leq (n/2s_n)\overline{F}_0(a_n^*-\delta_n)-1,\] any $0\le \eps \le 1$ and any $s_n>4$, we have
	\[ \inf_\vphi \sup_{\theta\in\Theta(\ba,s_n)} P_{\te} \brackets[\Big]{L_\rho(\te,\vphi)\geq \min\brackets[\big]{1- \tfrac{1+\eps}{\sqrt{s_n}},(\Lambda_n(\ba) -\eps )}\rho s_n}\geq 1-\exp(-\sqrt{s_n}\eps^2/3).\] 
	where the infimum is over all possible multiple testing procedures $\vphi=\vphi(X)$.
\end{theorem}

Note that the right side in the statement of Theorem~\ref{thm:clarg} only depends on $s_n, \veps$, so that one may optimise the left side with respect to $\rho$. In particular, in view of Assumption~\ref{ass:generalnoise}, we may choose some sequence $\rho\to \infty$, or take $\rho =1$. A version holds under Assumption~\ref{ass:generalnoise}\ref{ass:scale}; we omit the details.


Some ideas of the proof are inspired from \cite{sucandes16}, who derived an asymptotically sharp bound for the minimax in-probability risk in terms of the quadratic loss and Gaussian noise. Weighted classification loss is, like quadratic loss, a sum over coordinates, so one can split the global loss into blocks and define a least favourable prior in a similar way as for the former. There are two important differences:  firstly, our working with weighted classification losses leads to the study of different Bayes estimators, and secondly, we shall study the Bayes risk globally instead of reducing the problem to the study of Bayes risks over blocks.  The latter turns out to be necessary, as unlike for the quadratic loss, there is no ``concentration'' for the loss over a given block for (weighted or unweighted) classification losses.

\begin{proof}  
	The proof is similar to that of Theorem~\ref{th:minimax-multilevel}; we give a stand-alone proof here.
	Since $L_\rho\geq 0$,  there is nothing to prove if $\Lambda_n\leq \eps$, so assume that $\Lambda_n> \eps$. 
	Let $q= \floor{n/s_n}$ and let $n'=qs_n$.
	Let $\cP_\cL$ be the set of all prior distributions on $\Theta(\ba,s_n)$. For any $\lambda \in (0,1)$ we have
	\begin{equation} \label{eqn:ForCLarg} \inf_{\vphi} \sup_{\theta\in\Theta(\ba,s_n)}
		P_{\te}\left[L_\rho(\te,\vphi)\ge \la \rho s_n\right]  \ge \sup_{\pi\in\cP_\cL} \inf_{\vphi} 
		P_{\pi}\left[L_\rho(\te,\vphi)\ge \la \rho s_n\right],  \end{equation}
	where $P_\pi$ denotes the distribution of $(\te,X)$ in the Bayesian setting $\te\sim \pi$ and $X\given \te\sim P_\te$.  Let us define a specific prior $\pi$ as a product prior over $s_n$ blocks of consecutive coordinates $Q_1=\{1,2,\ldots, q\}, Q_2=\{q+1,\ldots,2q\},\ldots, Q_{s_n}=\{(s_n-1)q+1,\ldots,n'\}$. We write $Q_\infty$ for the (possibly empty) set $\{n'+1,\dots,n\}$. Let $\be[m,a]$ denote the vector of $\R^q$, with coordinates defined, for $1\le i, m \le q$ and $a>0 $, by 
	\[ (\be[m,a])_i=
	\begin{cases}
		&  a \quad\ \quad\:  \text{if } i=m;\\
		&  0  \quad\ \quad \:\:\;\text{if } i\neq m.
	\end{cases}
	\]
	In words, $\be[m,a] \in \R^q$ is the $1$-sparse vector with its only nonzero coordinate, at position $m$, equalling $a$. 
	Over each block $Q_j$, $1\leq j\leq s_n$, one takes the following prior:  
	first draw an integer $I_j$ from the  uniform distribution $\cU(Q_j)$ over the block $Q_j$ and next set $\be_{I_j}=a_j$  
	as above.   
	That is, $\pi_j$ generates $\be^j\in\RR^q$, to be identified with $(\be_i)_{i\in Q_j}$, according to	\begin{equation}\label{lfprior}
		I_j\sim \cU(Q_j),\qquad \be^j\given I_j \sim \delta_{\be[I_j\, \text{mod} \, q,a_j]},
	\end{equation} 
	with `mod $q$' meaning modulo the integer $q$. 
	If $Q_\infty$ is non-empty, set $\beta^\infty\sim \delta_0$. 
	By definition, $\pi$ belongs to $\cP_\cL$. 
	Using the fact that the classification loss is a sum of losses over all coordinates,  one may rewrite $L_{\rho}(\te,\vphi)=\sum_{j=1}^{s_n} L^j + L^\infty$, with $L^j$ the contribution of block $Q_j$. For $1\leq j\leq s_n$, we note that for $\theta$ sampled from $\pi$,
	\[ L^j =L^j_\rho(\te,\vphi)=\rho \ind{\vphi_{I_j}=0} +\sum_{i\in Q_j,\, i\neq I_j} \ind{\vphi_{i}\neq 0}.\]
	Then for any integer $\rho$ and $\la\in(0,1)$, noting that $L^\infty\geq 0$ we see that
	\begin{align*}
		P_{\pi}\left[L_\rho(\te,\vphi)\ge  \la \rho s_n\right] & \ge 
		P_{\pi}\left[\sum_{j=1}^{s_n} L^j \ge \la \rho s_n\right]  = 1 - P_{\pi}\left[\sum_{j=1}^{s_n} L^j < \la \rho s_n\right].
	\end{align*}
	For a multiple testing procedure $\vphi$, let $A^\vphi=\{i:\, \vphi_i(X)\neq 0\}$ denote its support and $A^\vphi_j=A^\vphi\cap Q_j$ its support within the $j$th block. 
	Let us consider the event $\cE$ defined as
	\[ \cE = \cE(\vphi)=\cE(\vphi,\rho,\la)
	=\left\{ L^1+\cdots + L^{s_n} < \la \rho s_n \right\}.\]
	To complete the proof, it is enough to bound $\sup_\vphi P_\pi[\cE]$ from above for a suitable $\lambda$.  Let us further define
	\[ N_\vphi(X,\te) = \sum_{j=1}^{s_n} \ind{I_j\in A_j^\vphi,\  \abs{A_j^\vphi}\le \rho}.\]
	
	One next notes that the following inequality holds:
	\begin{equation} \label{tr:ineqcl}
		L^1+\cdots + L^{s_n} \ge (s_n-N_\vphi(X,\te))\rho.
	\end{equation}
	To check this, it suffices to verify that for any index $j$ for which the indicator equals $0$ in the sum defining $N_\vphi(X,\te)$, the corresponding loss $L^j$ is at least $\rho$. The latter is true because if $I_j\notin A_j^\vphi$ we have  $\rho\ind{\vphi_{I_j}=0}=\rho$, and if $\abs{A_j^\vphi}\ge \rho+1$ then $\ind{\vphi_{i}\neq 0}$ equals $1$ at least $\rho$ times for $i\in Q_j, i\neq I_j$. 
	This leads to the desired inequality, which itself further implies that, for any integer $\rho\ge 1$,
	\begin{equation*}
		\cE \subset \{N_\vphi(X,\te)\ge  s_n - \la s_n\}. 
	\end{equation*}  
	We wish to bound $P_\pi[\cE]=E_\pi(P_\pi[\cE\given X])$ from above uniformly in $\vphi$, where $E_\pi$ denotes the (Bayesian) expectation under $P_\pi$, and it suffices to bound from above 
	\[ \sup_\vphi E_\pi P_\pi[\cE \given X] \le \sup_\vphi E_\pi
	P_\pi[N_\vphi(X,\te)\ge  (1 - \la) s_n \given X]. \]
	Conditional on $X$, the distribution of $N_\vphi(X,\theta)$ is Poisson-binomial, that is,  the law of a sum of  $s_n$ independent Bernoulli variables $Z_j$ with parameters $p^\vphi_j(X)$ given by \[  p^\vphi_j(X) = 
	\ind{\abs{A_j^\vphi}\le \rho} P_\pi[I_j\in A_j^\vphi \given X].\]	
	Let us now investigate the posterior distribution $P_\pi[\cdot \given X]$, restricted to coordinates $\theta_i,$ $i\in\braces{1,\dots,n'}=\cup_{j=1}^{s_n} Q_j$. By definition  the prior distribution is a product over the blocks $Q_j$. The model, that is the law of $X\given \theta$, is also of product form, so by Bayes' formula the posterior $\pi[\cdot\given X]$ is a product $\otimes_{j=1}^{s_n} \pi_j[\cdot\given X^j]$, where $X^j$ denotes the observations $(X_i : i\in Q_j)$ over the block $Q_j$. Writing $w_i^j(X)=P_\pi(I_j=i \mid X^j)$ for the posterior probability that $I_j=i$, we have
	\begin{align*} \pi_j[\cdot\given X^j]&=\sum_{i\in Q_j} w_i^j(X) \delta_{\be[i \text{ mod } q,a_j]}, 
		\\
		w_i^j(X)=& \frac{f_{a_j}(X_i)/f_0(X_i)}{\sum_{k\in Q_j} f_{a_j}(X_k)/f_0(X_k)}= \frac{h(X_i,a_j)}{\sum_{k\in Q_j} h(X_k,a_j)},\\ 
		h(x,a):=& f_a(x)/f_0(x).
	\end{align*} 
	By definition, $P_\pi[I_j\in A_j^{\vphi}\given X]=\sum_{i\in A_j^{\vphi}} h(X_i,a_j)/\sum_{k\in Q_j} h(X_k,a_j).$  Among (frequentist) estimators $\vphi(X)$ such that $\abs{A_j^{\vphi}}=\#\braces{i : \vphi_i(X)\neq 0} \leq \rho$, this expression is maximal for any estimator whose support $A_j^{\vphi}$ over the block $Q_j$ is equal to the set $A_j^X$  consisting of indices $i_1^j,\cdots, i_\rho^j$ corresponding to the $\rho$ largest observations among $(X_i : i\in Q_j)$; to see this, recall that Assumption~\eqref{ass:generalnoise} tells us that $h(X_k,a_j)$ is increasing in $X_k$. 
	

	Recalling $ p^\vphi_j(X) = 
	\ind{ \abs{A_j^\vphi}\le \rho} P_\pi[I_j\in A_j^\vphi \given X]$, for any procedure $\vphi$ and all $j$, we have
	\[ p^\vphi_j(X)  \le P_\pi[ I_j \in A_j^X\given X]=\pi_j[I_j\in A_j^X\given X^j]=:p_j^*(X^j). \]
	Write $\mathbf{p^\vphi}$ for the vector of $p^\vphi_j$'s, and similarly for $\mathbf{p^*}$. Lemma~\ref{lem:stodg} implies that, given $X$, $\operatorname{PBin}[\mathbf{p^*}]$ stochastically dominates $\operatorname{PBin}[\mathbf{p^\vphi}]$ (see the definition of the Poisson-Binomial distribution just before Lemma~\ref{lem:stodg}),  so that 
	\begin{align*}
		&  P_\pi[N_{\vphi}(X,\te)\ge (1-\la)s_n\given X]
		=  P_\pi[\operatorname{PBin}[\mathbf{p^\vphi}]\ge (1-\la)s_n\given X] \\
		& \qquad \ \le P_\pi[\operatorname{PBin}[\mathbf{p^*}]\ge (1-\la)s_n\given X],
	\end{align*}
	Let us note that $E_\pi p_j^*(X^j)=P_\pi[I_j\in A_j^X]$ and 
	\begin{equation*}
		P_\pi[I_j\in A_j^X] = P_\pi[X_{I_j} \sim f_{a_j} \text{ belongs to the $\rho$ largest among coordinates of }  X^j].
	\end{equation*} 
	Lemma~\ref{lem:ControlOfpn} (coupled with the fact that $I_j$ has a uniform distribution over $Q_j$ and is independent of $X$) tells us that the probability on the right is $p_j:= \overline{F}_{a_j}(a_n^*) + \eta_{j,n}$, 
	for some sequences $\eta_{j,n} = \eta_{j,n}(\ba)$ tending to zero as $n\to \infty$, uniformly in $j\leq s_n$ as $n \to \infty$. Write $\eta_n= s_n^{-1}\sum_{j\leq s_n} \eta_{j,n}.$
	
	This implies that the parameters $p_j^*(X^j)$ of the Poisson-binomial 
	$\operatorname{PBin}[\mathbf{p^*}]$ have expectations given by the vector $\bm{p}=(p_j)_{j\leq n}$ in the Bayesian model under $E_\pi$. Taking  the expectation under $E_\pi$ in the last but one display and using Lemma~\ref{lem:pbcondg} leads to
	\[ \sup_{\vphi} P_\pi[N_{\vphi}(X,\te)\ge (1-\la)s_n]
	\le P[ \operatorname{PBin}[\bm{p}]\ge (1-\la)s_n ].
	\] Note that $\operatorname{PBin}(\bm{p})$ has mean $s_n(1-\Lambda_n + \eta_n)$.
	
	If $1-\Lambda_n +\eta_n\geq s_n^{-1/2}$, define $\lambda$ by
	\[ 1-\la=(1+\eps) (1-\Lambda_n +\eta_n ),\] which, recalling that we have excluded already the case $\Lambda_n\leq \eps$, indeed has a solution $\lambda=\lambda_n(\eps,\ba)\in (0,1)$, at least for $n$ large enough that $\eta_n<\eps^2/2\leq \eps \Lambda_n /2$. We apply Bernstein's inequality for deviations of the Poisson Binomial distribution  from its mean  (Lemma~\ref{th:bernstein}) to obtain 
	\[ 
	P[ \operatorname{PBin}[\bm{p}]\ge (1-\la)s_n ]
	\le \exp(-s_n (1-\Lambda_n + \eta_n))\eps^2/3)	\le \exp(-s_n^{1/2} \eps^2/3).
	\]
	If instead $1- \Lambda_n +\eta_n< s_n^{-1/2}$, we set
	\[ 1- \lambda = (1+\eps)s_n^{-1/2}\]
	and Bernstein's inequality (Lemma~\ref{th:bernstein}), with the upper bound $s_n^{1/2}$ for the variance and expectation of $\operatorname{PBin}[\bm{p}]$, yields
	\[ P[\operatorname{PBin}[\bm{p}]\geq (1-\lambda)s_n]\leq \exp\brackets[\Big]{-\frac{\eps^2 s_n^{1/2}/2}{1+\eps/3}}\leq \exp(-\eps^2 s_n^{1/2}/3).\]
	Combining, and returning to \eqref{eqn:ForCLarg}, we deduce that
	\[ \inf_\vphi \sup_{\theta\in\Theta(\ba,s_n)} P_\theta \brackets[\Big]{L_\rho(\te,T)\geq (1- \lambda)\rho s_n}\geq 1-\exp(-\sqrt{s_n}\eps^2/3),\]
	for  $1-\lambda=\max\braces[\Big]{\tfrac{1+\eps}{\sqrt{s_n}},(1+\eps)\brackets[\big]{1-\Lambda_n + \eta_n}}$.
	Noting that $1- (1+\eps)(1-\Lambda_n +\eta_n)\geq \Lambda_n -\eps,$ at least for $n$ large, concludes the proof.
\end{proof}

\section{~~Further discussion and comparison of Assumption~\ref{ass:generalnoise} with  \cite{rabinovichpreprint}}\label{sec:rab}


The closest previous model to ours that we know of comes from \cite{rabinovichpreprint}, in which the authors consider data defined for some function $f$ by
\begin{align*} X_i =  \begin{cases} W_i & \theta_i = 0 \\
		f(W_i) & \text{otherwise},\end{cases} \\
	W = (W_i : i\leq n)\sim \mathbb{P}_n.
\end{align*}
This helped inspire our Assumption~\ref{ass:generalnoise} but there are some key differences:
\begin{enumerate}
	\item They do not assume $\mathbb{P}_n$ is of product form, i.e.\ the $X_i$ need not be independent.
	\item They have a single function $f$ controlling the non-null signals, in contrast to our setting where each non-null $X_i$ has its own $\theta_i$ (though they could in principle use the non-iid nature of the $W_i$ to partially accommodate this).
	\item In place of Assumption~\ref{ass:generalnoise}, they assume that $f$ is non-decreasing and $f(w)\geq w$ for all $w$ in the support of the variables $W_i$. They assume the distribution of each $W_i$ is non-atomic.
	\item Their analysis is for `top $K$' procedures only.
\end{enumerate}

	The key advantage of their assumptions is that they can accommodate non-independent data. For their results to be practical, they assume in Section 4.1.1 that some concentration inequality holds. The key reason we need independence is for concentration arguments, so it is possible that our results could be generalised to hold under some version of their concentration assumptions. 
	
	Similarly, while they do not assume anything of the form \eqref{eqn:modelassumption2} and \eqref{eqn:modelassumption3}, we need this only to ensure we get matching upper bounds to the lower bounds, which they do not target.
	
	Finally, the key advantage of our assumptions is that we do not need to restrict our attention to top-$K$ procedures: we are able to prove that procedures based on absolute value thresholding are asymptotically optimal under \eqref{eqn:monotonicity-location}--\eqref{eqn:monotonicity-scale}.
	In contrast, note that their assumption that $f$ is increasing and $f(w)\geq w$ does not justify that top-$K$ procedures will be optimal. Indeed, if $f(w)=\ceil{w}$, then (due to the non-atomicity assumption) the optimal procedure will not be a top-$K$ procedure but rather will be
	\[ (\vphi_i(X))_{i\leq n} = (\II\braces{ X_i\in \mathbb{Z}})_{i\leq n},\]
	and will have $\FDR$ and $\FNR$ both equal to zero. Note also that top-$K$ procedures are provably suboptimal for boundary data in our setting: see Remark~\ref{rem:topKsuboptimal} and Corollary \ref{cor-topsub}.

\end{document}